\documentclass[12pt]{amsart}

\usepackage{amsmath, amsthm, amssymb}
\usepackage{graphicx}
\usepackage{mathbbol}
\usepackage{txfonts}
\usepackage[usenames]{color}
\usepackage{pgfplots}
\usepackage{comment}
\usepackage{esint}
\usepackage{cancel}
\usepackage{enumerate}
\usepackage{xcolor}
\usepackage[colorlinks,linkcolor=blue]{hyperref}

\allowdisplaybreaks[4]

\newtheorem{thm}{Theorem}[section]
\newcommand{\bt}{\begin{thm}}
\newcommand{\et}{\end{thm}}

\newtheorem{conj}[thm]{Conjecture}

\newtheorem{cor}[thm]{Corollary}   
\newcommand{\bc}{\begin{cor}}
\newcommand{\ec}{\end{cor}}

\newtheorem{lem}[thm]{Lemma}   
\newcommand{\bl}{\begin{lem}}
\newcommand{\el}{\end{lem}}

\newtheorem{prop}[thm]{Proposition}
\newcommand{\bp}{\begin{prop}}
\newcommand{\ep}{\end{prop}}

\newtheorem{defn}[thm]{Definition}
\newcommand{\bd}{\begin{defn}}    
\newcommand{\ed}{\end{defn}}

\newtheorem{rmrk}[thm]{Remark}   

\newcommand{\br}{\begin{rmrk}}
\newcommand{\er}{\end{rmrk}}

\newcommand{\mina}{\operatorname{MinA}}

\newcommand{\Scal}{\operatorname{Scalar}}

\newcommand{\be}{\begin{equation}}

 \newcommand{\ee}{\end{equation}}

\newcommand{\N}{\mathbb{N}}

\newcommand{\R}{\mathbb{R}}

\newcommand{\Z}{\mathbb{Z}}

\newcommand{\diam}{\operatorname{Diam}}

\newcommand{\mass}{{\mathbf M}}

\newcommand{\area}{\operatorname{Area}}
\newcommand{\vol}{{\rm vol}}

\newcommand{\Sph}{{\mathbb S}}         

\newcommand{\og}{\overline{\Gamma}}
\newcommand{\on}{\overline{\nabla}}

\newcommand{\I}{{\bf I}}

\begin{document}

\title[Compactness of warped circles over spheres with NNSC]{Compactness of sequences of warped product circles over spheres with nonnegative scalar curvature}


\author{Wenchuan Tian}
\address{Department of Mathematics, University of California, Santa Barbara, CA93106-3080}
\email{tian.wenchuan@gmail.com}

\author{Changliang Wang}
\address{School of Mathematical Sciences and Institute for Advanced Study, Tongji University, Shanghai 200092, China}
\email{wangchl@tongji.edu.cn}

\date{}


\begin{abstract}
Gromov and Sormani conjectured that a sequence of three dimensional Riemannian manifolds with nonnegative scalar curvature and some additional uniform geometric bounds should have a subsequence which converges in some sense to a limit space with some generalized notion of nonnegative scalar curvature. In this paper, we study the pre-compactness of a sequence of three dimensional warped product manifolds with warped circles over standard $\Sph^2$ that have nonnegative scalar curvature, a uniform upper bound on the volume, and a positive uniform lower bound on the $\mina$, which is the minimum area of closed minimal surfaces in the manifold. We prove that such a sequence has a subsequence converging to a $W^{1, p}$ Riemannian metric for all $p<2$, and that the limit metric has nonnegative scalar curvature in the distributional sense as defined by Lee-LeFloch.
\end{abstract}

\maketitle

\tableofcontents

\section{ Introduction}

\noindent In \cite{Gromov-Plateau} and \cite{Gromov-Dirac}, Gromov conjectured that a sequence of Riemannian manifolds with nonnegative scalar curvature, $\Scal \ge 0$, should have a subsequence which converges in some weak sense to a limit space with some generalized notion of ``nonnegative scalar curvature".  In light of the examples constructed by Basilio, Dodziuk, and Sormani in \cite{BDS-sewing}, the $\mina$ condition in (\ref{defn-MinA}) below was added to prevent collapsing happening, and the conjecture was made more precise at an IAS Emerging Topics Workshop co-organized by Gromov and Sormani as follows \cite{Sormani-Scalar}:

\begin{conj}\label{Scalar-Compactness}
Let $\{M_j^3\}_{j=1}^\infty$ be a sequence of
 closed oriented three dimensional Riemannian manifolds without boundary satisfying
\be
 \Scal_j \geq 0, \ \ {\rm Vol}(M_j) \le V, \ \ \diam(M_j) \le D,
\ee
\be \label{defn-MinA}
\mina(M^3_j)=\inf\{\area(\Sigma)\, : \, \Sigma \textrm{ closed min surf in } M_j^3\,\} \ge A_0>0.
\ee
Then there exists a subsequence which is still denoted as $\{M_j\}_{j=1}^\infty$ that converges in the volume preserving intrinsic flat sense to a three dimensional rectifiable limit space $M_\infty$.
Furthermore, $M_\infty$ is a connected geodesic metric space, that has Euclidean tangent cones almost everywhere, and has nonnegative generalized scalar curvature.
\end{conj}

In a joint work with Jiewon Park \cite{Park-Tian-Wang-18}, the authors confirmed Conjecture \ref{Scalar-Compactness} for sequences of rotationally symmetric Riemannian manifolds $(M^3_j, g_j)$. In our proof the $\mina$ condition provides a uniform lower bound for the warping functions in the closed region between any two minimal surfaces. As a result, we can prevent counter examples like the sequence of round spheres shrinking to a point, and we can also prevent the formation of thin tunnels between two non-collapsed regions. The regularity of the limit metric is high, and the convergence of the sequence of warping functions is strong. In particular, in \cite{Park-Tian-Wang-18} we proved that the limit warping function is Lipschitz and that the sequence of warping functions converges to the limit function in the $W^{1,2}$ norm in closed regions away from the two poles.

In this paper, we study the $\Sph^2\times_f \Sph^1$ warped product case of the Conjecture \ref{Scalar-Compactness}. We consider the following:
\begin{defn}\label{Defn-Main}
{\rm
Let $\{(\Sph^2\times \Sph^1, g_j)\}_{j=1}^\infty$ be a sequence of Riemannian manifold such that
\be\label{eqn-circle-over-sphere}
g_{j}=g_{\Sph^2} +f_j^2 g_{\Sph^1}=dr^2+\sin(r)^2 d\theta^2+f_j^2 d\varphi^2, \text{ for }j=1,2, 3, ...
\ee
where $g_{\Sph^2}$ and $g_{\Sph^1}$ are the standard metrics on $\Sph^2$ and $\Sph^1$ respectively, and the function $f_j: \Sph^2\to (0,\infty)$ is smooth for each $j$. Here $r$ and $\theta$ are the geodesic polar coordinate for $\Sph^2$. We also use the notation $\Sph^2 \times_{f_{j}} \Sph^1$ to denote  $(\Sph^2\times \Sph^1, g_{j})$.
}
\end{defn}

We consider the convergence of the warping function and prove the sharp regularity of the limit warping function in the following theorem:
\begin{thm}\label{Intro-thm: f_infty positive}
Let $\{\Sph^2 \times_{f_j} \Sph^1\}_{j=1}^\infty$ be a sequence of warped product Riemannian manifolds such that each $\Sph^2 \times_{f_j} \Sph^1$ has non-negative scalar curvature. If we assume that
\be\label{eqn-Intro-thm-main-condition}
{\rm Vol} (\Sph^2 \times_{f_j} \Sph^1) \leq V \text{ and }\mina(\Sph^2 \times_{f_{j}} \Sph^1) \geq A >0, \ \  \forall j \in \N,
\ee
then we have the following:
\begin{enumerate}[{\rm (i)}]
   \item After passing to a subsequence if needed, the sequence of warping functions $\{f_j\}_{j=1}^\infty$ converges to some limit function $f_\infty$ in $L^{q}(\Sph^2)$ for all $q \in [1, \infty)$.
   \item The limit function $f_\infty$ is in $W^{1, p}(\Sph^2)$, for all $p$ such that $1 \leq p <2$.

   \item The essential infimum of $f_{\infty}$ is strictly positive, i.e. $\inf\limits_{\Sph^2} f_{\infty} >0$.
   \item If we allow $+\infty$ as a limit, then the limit
             \be
            \overline{f_\infty}(x) := \lim_{r \rightarrow 0} \fint_{B_{r}(x)} f_{\infty}
            \ee
         exists for every $x \in \Sph^2$. Moreover, $\overline{f_{\infty}}$ is lower semi-continuous and strictly positive everywhere on $\Sph^2$, and $\overline{f_\infty} = f_\infty $ a.e. on $\Sph^2$.
         \end{enumerate}
\end{thm}

The definition of essential infimum is given in Definition \ref{Defn-Ess-Inf}. In the proof of convergence properties in items (i) and (ii) in Theorem \ref{Intro-thm: f_infty positive}, we only need nonnegative scalar curvature condition and volume uniform upper bound condition. In the proof of part (iii) of Theorem \ref{Intro-thm: f_infty positive}, we make essential use of $\mina$ condition combined with the spherical mean inequality [Proposition \ref{prop: spherical mean inequality}], Min-Max minimal surface theory and a covering argument. This is an interesting new way of applying the $\mina$ condition to prevent collapsing. Then the part (iv) follows from (iii) and an interesting ball average monotonicity property [Proposition \ref{prop: ball average non-increasing}]. The ball average monotonicity is obtained from spherical mean inequality by using the trick as in the proof of Bishop-Gromov volume comparison theorem.

\begin{rmrk}\label{Intro-rmrk-best-regularity}
{\rm
The extreme example constructed by Sormani and authors in \cite{STW-ex} shows that the $W^{1, p}$ regularity for $1 \leq p <2$ is sharp for the limit warping function $f_\infty$.
}
\end{rmrk}

By applying Theorem \ref{Intro-thm: f_infty positive} and the spherical mean inequality [Proposition \ref{prop: spherical mean inequality}], we obtain:. 
\begin{prop}\label{Intro-prop: warping function uniform lower bound}
Let $\{\Sph^2 \times_{f_j} \Sph^1\}_{j=1}^\infty$ be a sequence of warped product manifolds such that each $\Sph^2 \times_{f_j} \Sph^1$ has non-negative scalar curvature, and the sequence satisfies conditions in $(\ref{eqn-Intro-thm-main-condition})$.
Then there exists $j_0 \in \N$ such that $f_j(x)\geq \frac{e_\infty}{4}>0$, for all $j \geq j_0$ and $x \in \Sph^2$, where $e_\infty = \inf_{\Sph^2}f_\infty >0$ obtained in Theorem \ref{Intro-thm: f_infty positive}.
\end{prop}

As an application of Proposition \ref{Intro-prop: warping function uniform lower bound}, we have:
\begin{cor}\label{Intro-cor: systole lower bound}
Let $\{\Sph^2 \times_{f_j} \Sph^1\}_{j=1}^\infty$ be a sequence of warped product manifolds such that each $\Sph^2 \times_{f_j} \Sph^1$ has non-negative scalar curvature, and the sequence satisfies conditions in $(\ref{eqn-Intro-thm-main-condition})$.
Then the systoles of $\Sph^2 \times_{f_j} \Sph^1$, for all $j\in \N$, have a uniform positive lower bound given by $\min\left\{2\pi, \frac{e_\infty}{2}\pi\right\}$, where $e_\infty := \inf\limits_{\Sph^2} f_\infty > 0$ obtained in Theorem \ref{Intro-thm: f_infty positive}.
\end{cor}

The systole of a Riemannian manifold is defined to be the length of the shortest closed geodesic in the manifold [Definition \ref{defn: systole}]. In order to estimate systole of warped product manifolds: $\Sph^2 \times_{f_j} \Sph^1$, in Lemma \ref{lem: geodesic dichotomy} we establish an interesting dichotomy property for closed geodesics in a general warped product manifold $N \times_f \Sph^1$ with $\Sph^1$ as a typical fiber, with metric tensor as $g = g_N + f^2 g_{\Sph^1}$, where $(N, g_N)$ is a $n$-dimensional complete Riemannian manifold without boundary and $f$ is a positive smooth function on $N$. The dichotomy property in Lemma \ref{lem: geodesic dichotomy} has its own interests independently, and shall be useful in other studies of closed geodesics in such warped product manifolds. 

The convergence of the warping functions in Theorem \ref{Intro-thm: f_infty positive} leads to the convergence of the Riemannian metrics, we prove the following:

\begin{thm} \label{Intro-thm-Lq-convergence}
Let $\{\Sph^2 \times_{f_j} \Sph^1\}_{j=1}^\infty$ be a sequence of warped product Riemannian manifolds such that each $\Sph^2 \times_{f_j} \Sph^1$ has non-negative scalar curvature. If we assume that
\be
{\rm Vol} (\Sph^2 \times_{f_j} \Sph^1) \leq V \text{ and }\mina(\Sph^2 \times_{f_{j}} \Sph^1) \geq A >0, \ \  \forall j \in \N,
\ee
Then there exists a subsequence $g_{j_k}$ and a (weak) warped product Riemannian metric $g_\infty \in W^{1, p}(\Sph^2 \times \Sph^1, g_0)$ for $p \in [1, 2)$ such that
\be
 g_{j_k} \rightarrow g_\infty \ \ \text{ in} \ \  L^{q}(\Sph^2\times \Sph^1, g_0), \ \ \forall q \in [1, \infty).
\ee
\end{thm}

Theorem \ref{Intro-thm-Lq-convergence} is proved in \S\ref{subsect-W1p-limit-metric}.
The definition of a (weak) warped product Riemannian metric is given in Definition \ref{defn-limit-metric}, and the spaces $ L^{q}(\Sph^2\times \Sph^1, g_0)$ and $ W^{1, p}(\Sph^2 \times \Sph^1, g_0)$ are defined in Definition \ref{defn-W1p}. The $\mina$ condition is used to prevent $g_{j_k}$ converging to a non-metric tensor in $W^{1, p}(\Sph^2 \times \Sph^1, g_0)$, with the help of the non-collapsing property of $f_\infty$ in the item (iii) in Theorem \ref{Intro-thm: f_infty positive}.

In the limit space we calculate the scalar curvature as a distribution using the definition by Lee and LeFloch \cite{Lee-LeFloch}, and we prove the following:
\begin{thm}\label{Intro-thm-distr-scalar}
The limit metric $g_\infty$ obtained in Theorem \ref{Intro-thm-Lq-convergence} has nonnegative distributional scalar curvature on $\Sph^2\times \Sph^1$  in the sense of Lee-LeFloch. \cite{Lee-LeFloch}. Moreover, the total scalar curvatures of $g_j$ converge to the distributional total scalar curvature of $g_\infty$.
\end{thm}

Theorem \ref{Intro-thm-distr-scalar} is proved in \S\ref{subsect-nonnegative-distr-scalar}. In general, it is still an interesting and difficult problem to formulate suitable notions of generalized (or weak) nonnegative scalar curvature in Conjecture \ref{Scalar-Compactness}. A natural candidate is the volume-limit notion of nonnegative scalar curvature. But recently Kazara and Xu constructed a sequence of warped product metrics on $\Sph^2 \times \Sph^1$ whose limit space does not have nonnegative scalar curvature in the sense of volume-limit in Theorem 1.3 in \cite{Kazaras-Xu}. There are other candidates, like Gromov's polyhedron comparison notion \cite{Gromov-Dirac, Li-polyhedral} and Burkhardt-Guim's Ricci flow notion \cite{Burkhardt-Guim-GAFA} of nonnegative scalar curvature for $C^0$-metrics.  However, as mentioned in Remark \ref{Intro-rmrk-best-regularity}, the $W^{1, p}$ regularity, for $1 \leq p < 2$, is the best regularity for our limit metrics, and in general our limit metrics are not continuous. Lee and Lefloch \cite{Lee-LeFloch} defined the scalar curvature distribution for $W^{1, 2}_{loc}$-metrics. Our limit metric $g_\infty$ obtained in Theorem \ref{Intro-thm-Lq-convergence} does not satisfy the regularity requirement in \cite{Lee-LeFloch}, but when we add up different terms in the integrand, the divergent terms cancel with each other and the scalar curvature is still well defined as a distribution. This is discussed in detail in Remark \ref{rmrk-LL-divergence}. Interestingly, we obtain the continuity of distributional total scalar curvature in Theorem \ref{Intro-thm-distr-scalar}. More importantly, the scalar curvature distribution of Lee-LeFloch enables us to see the concentration of scalar curvature on the singular set, see \S4.4 in \cite{STW-ex}.


In Appendix \ref{appendix}, we study pre-compactness of the sequence of warped product spheres over circle $(M^3_j, g_j)$, that is, $M^3_j$ are diffeomorphic to $\Sph^1 \times \Sph^2$ with warped product metric tensors
\be\label{eqn-sphere-over-circle}
g_j = g_{\Sph^1} + h_j^2 g_{\Sph^2}, \ \ \textrm{where} \ \ h_j: \Sph^1 \rightarrow (0, \infty).
\ee
The study of this case is similar to the rotationally symmetric case studied in \cite{Park-Tian-Wang-18}. The key is to obtain a uniform bound for the norm of gradient of $h_j$ from nonnegative scalar curvature condition [Lemma \ref{lem-scal}]. By combining this with uniform diameter upper bound and the $\mina$ condition, we prove that a subsequence of $\{h_j\}^\infty_{j=1}$ converges in $C^0$ and $W^{1, 2}$ sense to a bounded positive Lipschitz function $h_\infty: \Sph^1 \rightarrow (0, \infty)$ [Theorem \ref{thm-sphere-over-circle-convergence}]. Moreover, we prove that the limit $W^{1, 2}$ Riemannian metric $g_\infty = g_{\Sph^1} + h^2_\infty g_{\Sph^2}$ has nonnegative distributional scalar curvature in the sense of Lee-LeFloch [Theorem \ref{thm-scalar-appendix}].

The proof of Theorem \ref{thm-sphere-over-circle-convergence} is similar to that of Theorems 4.1 and 4.8 in \cite{Park-Tian-Wang-18}. We include it here to show the difference with the rotationally symmetric case and the difference with Theorem \ref{Intro-thm: f_infty positive} and Theorem \ref{Intro-thm-Lq-convergence}.

The proof of Theorem \ref{thm-scalar-appendix} shows that in this case the regularity requirement in Lee-LeFloch \cite{Lee-LeFloch} is essential for the definition of the scalar curvature as a distribution. This provides an interesting contrast with the proof of Theorem \ref{Intro-thm-distr-scalar}.

The article is organized as follows: in Section \ref{sec-consequences}, we derive several analysis properties of warping functions $f_j$ from the uniform geometric bounds of metric $g_j$ as in (\ref{eqn-circle-over-sphere}). In particular, we show that metrics $g_j$ in (\ref{eqn-circle-over-sphere}) have nonnegative scalar curvature if and only if the warping functions $f_j$ satisfy the differential inequality [Lemma \ref{lem: nonnegative scalar curvature condition}]:
\be\label{eqn-diff-inequality}
\Delta f_j \leq f_j, \ \ \textrm{on} \ \ \Sph^2,
\ee
where $\Delta$ is the Lapacian on the standard round sphere $\Sph^2$, taken to be the trace of the Hessian. Moreover, a positive number $V$ is a uniform upper bound of volumes of metrics $g_j$ in (\ref{eqn-circle-over-sphere}) if and only if $f_j$ satisfy [Lemma \ref{lem: volume upper bound condition}]
\be\label{eqn-integral-bound}
\int_{\Sph^2} f_j d\vol_{g_{\Sph^2}} \leq \frac{V}{2\pi}.
\ee

It is well-known that the spherical mean property of (sub, sup)-harmonic functions plays important roles in the study of these functions. Inspired by this, we prove a spherical mean inequality for functions $f_j$ satisfying the differential inequality (\ref{eqn-diff-inequality}) [Proposition \ref{prop: spherical mean inequality}]. It turns out that the spherical mean inequality is very important in the proof of non-collapsing property in Section \ref{sect-positivity-limit}, in particular, in the proof of Proposition \ref{prop-L1-bound-mina-from-upper}. Furthermore, by employing the trick in the proof of Bishop-Gromov volume comparison theorem, we prove a ball average monotonicity property for $f_j$ [Proposition \ref{prop: ball average non-increasing}], which helps us to obtain lower semi-continuity of the limit warping function $f_\infty$ in Proposition \ref{prop: average limit exists}.

In Section \ref{sect-W1p-limit}, we study the convergence of a sequence $\{f_j\}^\infty_{j=1}$ of positive functions on $\Sph^2$ satisfying (\ref{eqn-diff-inequality}) and (\ref{eqn-integral-bound}). We prove that there exists a subsequence of such sequence $\{f_j\}$ and a function $f_\infty \in W^{1, p}(\Sph^2) \, (1 \leq p <2)$ such that the subsequence converges to $f_\infty$ in $L^{q}(\Sph^2)$ for any $q \geq 1$ [Proposition \ref{prop-W1p-limit}]. The proof of this convergence result is very different from that in cases of warped product metrics as in \cite{Park-Tian-Wang-18} and in (\ref{eqn-sphere-over-circle}). Because warping functions $h_j$ in \cite{Park-Tian-Wang-18} and in (\ref{eqn-sphere-over-circle}) have one variable, whereas $f_j$ in (\ref{eqn-circle-over-sphere}) have two variables,  it is more difficult to obtain sub-convergence of $\{ f_j \}$, and we make use of the Moser-Trudinger inequality in (25) in \cite{Onofri-82}. The regularity of the limit function $f_\infty$ is weaker than $h_\infty$. The extreme example constructed by Sormani and authors in \cite{STW-ex} shows that the $W^{1, p}$ regularity for $1 \leq p <2$ is sharp for $f_\infty$.

In Section \ref{sect-positivity-limit}, we use the $\mina$ condition to show that the limit function $f_\infty$ has positive essential infimum [Theorem \ref{thm: f_infty positive}] and that the warping functions $f_j$ have a positive uniform lower bound [Proposition \ref{prop: f_j lower bound}]. This enables us to define weak warped product Riemnnian metric $g_\infty$ on $\Sph^2 \times \Sph^1$ in Definition \ref{defn-limit-metric}, and is crucial in the study of geometric convergence of warped product circles over sphere with metric tensor as in (\ref{eqn-circle-over-sphere}). Moreover, as a consequence of Proposition \ref{prop: f_j lower bound}, we obtain a positive uniform lower bound for the systole of the warped product manifolds $\Sph^2 \times_{f_j} \Sph^1$ [Proposition \ref{prop: systole uniform lower bound}].

The $\mina$ condition can be viewed as a noncollapsing condition. As shown in \cite{Park-Tian-Wang-18} and in Lemma \ref{lem-lower-bound} below, it is not difficult to see this in cases of metric tensors as in \cite{Park-Tian-Wang-18} and (\ref{eqn-sphere-over-circle}). In the case of metric tensors as in (\ref{eqn-circle-over-sphere}), however, the implication of the $\mina$ condition  is much more complicated. We need to use  the Min-Max minimal surface theory of Marques and Neves (see e.g. \cite{MN-LNM-2018}), the maximum principle for weak solutions (Theorem 8.19 in \cite{GT-PDE-book}), and the spherical mean inequality obtained in Proposition \ref{prop: spherical mean inequality}, in order to obtain noncollapsing from the $\mina$ condition.

In Section \ref{sect-distr-scal}, we prove that a subsequence of $\{g_j\}^\infty_{j=1}$, with $g_j$ as in (\ref{eqn-circle-over-sphere}) having nonnegative scalar curvatures and uniform upper bounded volumes and satisfying $\mina$ condition, converges to a weak metric tensor $g_\infty \in W^{1, p}(\Sph^2 \times \Sph^1, g_0) \, (1 \leq p <2)$ in the sense of $L^{q}(\Sph^2 \times \Sph^1, g_0)$ for all $q \geq 1$ [Theorem \ref{thm-Lq-convergence}]. Moreover, we prove that the limit metric $g_\infty$ has nonnegative distributional scalar curvature in the sense of Lee-LeFloch [Theorem \ref{thm-distr-scalar}].

Note that in the case of metric tensors as in \cite{Park-Tian-Wang-18} and (\ref{eqn-sphere-over-circle}), we need the diameter uniform upper bound condition in addition to nonnegative scalar curvature condition and the $\mina$ condition for getting convergence [Theorem 1.3 in \cite{Park-Tian-Wang-18} and Theorem \ref{thm-sphere-over-circle-convergence}], whereas in the case of metric tensors as in (\ref{eqn-circle-over-sphere}), we need the volume uniform upper bound condition instead of the diameter uniform upper bound condition [Theorem \ref{thm-Lq-convergence}].

\noindent
{\bf Acknowledgements:}

The authors would like to thank the Fields Institute for hosting the {\em Summer School on Geometric Analysis} in July 2017 where we met Professor Christina Sormani and she started to guide us working on the project concerning compactness of manifolds with nonnegative scalar curvatures. We are grateful to Professor Sormani for her constant encouragement and inspiring discussions. In particular, Professor Sormani suggested us the method of spherical means, and it turns out to be very useful in the study of warping functions in Theorem \ref{Intro-thm: f_infty positive}. We thank Brian Allen for discussions and interest in this work. Wenchuan Tian was partially supported by the AMS Simons Travel Grant. Changliang Wang was partially supported by the Fundamental Research Funds for the Central Universities and Shanghai Pilot Program for Basic Research.



\section{ Consequences of the geometric hypotheses on $\Sph^2 \times_f \Sph^1$}\label{sec-consequences}

In this section we prove several consequences of the uniform geometric bounds. In Subsection \ref{subsect-basic-consequences}, we derive the differential inequality satisfied by the warping function $f_j$ and prove that the uniform volume bounds on sequence of Riemannian manifolds implies the uniform $L^1 $ norm of the warping function.

In Subsection \ref{subsect-spherical-mean-inequality}, we prove the spherical mean inequality for the warping function $f$ [Proposition \ref{prop: ball average non-increasing}], which is our main analytic tool. In Subsection \ref{subsect-ball-average}, we prove a ball average monotonicity property for the warping function $f$ [Proposition \ref{prop: spherical mean inequality}].

The implication of the $\mina$ condition is more complicated we discuss that in Section \ref{sect-positivity-limit}.


\subsection{Basic consequences of the hypotheses}\label{subsect-basic-consequences}

\begin{lem}[Non-negative scalar curvature condition]\label{lem: nonnegative scalar curvature condition}
The scalar curvature of warped product manifolds $\Sph^2 \times_{f_j} \Sph^1$ are given by
\be\label{eqn-scalar}
\Scal_{j}=2-2\frac{\Delta f_j}{f_j},
\ee
where $\Delta$ is the Laplacian on $\Sph^2$ with respect to the standard metric $g_{\Sph^2}$, taken to be the trace of the Hessian (without the negative sign).

Thus $\Sph^2 \times_{f_j} \Sph^1$ have nonnegative scalar curvature if and only if
\be
\Delta f_{j} \leq f_{j}.
\ee
\end{lem}
\begin{proof}
By using the Ricci curvature formula for warped product metrics as in Proposition 9.106 of \cite{Besse}, we can easily obtain the scalar curvature of $\Sph^2 \times_{f_j} \Sph^1$ as $\Scal_{j}=2-2\frac{\Delta f_j}{f_j}.$ Then the second claim directly follows, since $f_j >0$.
\end{proof}

\begin{lem}[Volume upper bound condition]\label{lem: volume upper bound condition}
The warped product manifolds $\Sph^2 \times_{f_j} \Sph^1$ have volume $ { \rm Vol} (\Sph^2 \times_{f_j} \Sph^1) \leq V$ if and only if
\be
\int_{\Sph^2} f_j d {\rm vol}_{\Sph^2} \leq \frac{V}{2\pi}.
\ee
\end{lem}
\begin{proof}
The Riemannian volume measure of $g_j$ is given by
\be
d\vol_{g_j} = f_j d\vol_{g_{\Sph^2}} d\vol_{g_{\Sph^1}}.
\ee
Thus the volume of $\Sph^2 \times_{f_j} \Sph^1$ is given by
\be
{\rm Vol} ( \Sph^2 \times_{f_j} \Sph^1) = \int_{\Sph^2 \times \Sph^1} f_j d\vol_{g_{\Sph^2}} d\vol_{\Sph^1} = 2\pi \int_{\Sph^2} f_j d\vol_{g_{\Sph^2}}.
\ee
Then the claim directly follows.
\end{proof}


\subsection{Spherical mean inequality}\label{subsect-spherical-mean-inequality}
In this subsection, we prove a spherical mean inequality [Proposition \ref{prop: spherical mean inequality}] for the smooth functions $f$ on $\Sph^2$ satisfying the differential inequality $\Delta f \leq f$. By Lemma \ref{lem: nonnegative scalar curvature condition}, this is equivalent to studying the warping function of warped product manifolds $\Sph^2 \times_f \Sph^1$ with nonnegative scalar curvature. The spherical mean inequality plays an important role in the proof of Proposition \ref{prop-L1-bound-mina-from-upper}.

The derivation of the spherical mean value inequality is similar to that of the mean value property of harmonic functions. We start with the following lemma.
\begin{lem}\label{lem: spherical mean derivative}
Let $f$ be a smooth function on $\Sph^2$. Consider the spherical mean given by
\begin{equation}
\phi(r) : = \fint_{\partial B_{r}(p)} f ds,
\end{equation}
where $B_{r}(p)$ is the geodesic ball in the standard $\mathbb{S}^{2}$ with center $p$ and radius $r$. The derivative of $\phi(r)$ satisfies
\begin{equation}
\frac{d}{dr}\phi(r) = \frac{1}{2 \pi \sin r}\int_{B_{r}(p)}  \Delta f d{\rm vol}_{\Sph^2}.
\end{equation}
\end{lem}
\begin{proof}
Using the geodesic polar coordinate $(r, \theta)$ on $\mathbb{S}^{2}$ centered at $p$, one can write $\phi(r)$ as
\begin{equation}
\phi(r) = \frac{\int^{2\pi}_{0} f(r, \theta) \sin r d\theta}{2\pi \sin r} = \frac{\int^{2\pi}_{0}f(r, \theta)d\theta}{2\pi}.
\end{equation}
Then taking derivative with respective to $r$ gives
\begin{eqnarray}
\phi^{\prime}(r)
& = & \frac{1}{2\pi} \int^{2\pi}_{0} \frac{\partial f}{\partial r} d\theta \\
& = & \frac{1}{2\pi} \int^{2\pi}_{0} \langle \nabla f, \partial_{r} \rangle \\
& = & \frac{1}{2\pi \sin r} \int^{2\pi}_{0} \langle \nabla f, \partial_{r} \rangle \sin r d\theta \\
& = & \frac{1}{2\pi \sin r} \int_{\partial B_{r}(p)} \langle \nabla f, \partial_{r} \rangle ds\\
& \overset{Stokes}{=} & \frac{1}{2\pi \sin r} \int_{B_{r}(p)} \Delta f d{\rm vol}_{\mathbb{S}^{2}}.
\end{eqnarray}
\end{proof}

Now we use Lemma \ref{lem: spherical mean derivative} to prove the spherical mean inequality.

\begin{prop}\label{prop: spherical mean inequality}
Let $f$ be a smooth function on $\Sph^2$ satisfying $\Delta f \leq f$. Then for any fixed $p\in \mathbb{S}^{2}$ and $0<r_{0} < r_{1} \leq \frac{\pi}{2}$, one has
\begin{equation}
\fint_{\partial B_{r_{1}}(p)} f ds - \fint_{\partial B_{r_{0}}(p)} f ds \leq \frac{\|f\|_{L^{2}(\mathbb{S}^{2})}}{\sqrt{2\pi}} (r_{1}-r_{0}),
\end{equation}
where $B_{r}(p)$ is the geodesic ball in the $\mathbb{S}^{2}$ with center $p$ and radius $r$.

Moreover, by taking limit as $r_{0} {\rightarrow} 0$, one has
\begin{equation}
\fint_{\partial B_{r}(p)} f ds - f(p) \leq \frac{\|f\|_{L^{2}(\mathbb{S}^{2})}}{\sqrt{2\pi}}r,
\end{equation}
for any $0< r\leq \frac{\pi}{2}$.
\end{prop}

\begin{proof}
By Lemma \ref{lem: spherical mean derivative} and the assumption $\Delta f \leq f$, one has
\begin{equation}
\phi^{\prime}(r) \leq \frac{1}{2\pi \sin r} \int_{B_{r}(p)} f d {\rm vol}_{\mathbb{S}^{2}}.
\end{equation}
Integrating this differential inequality for $r$ from $r_{0}$ to $r_{1}$ gives
\begin{eqnarray}
\phi(r_{1}) - \phi(r_{0})
& \leq & \int^{r_{1}}_{r_{0}} \left( \frac{1}{2\pi \sin r} \int_{B_{r}(p)} f d {\rm vol}_{\mathbb{S}^{2}}\right) dr \\
& \leq & \int^{r_{1}}_{r_{0}} \left( \frac{1}{2\pi \sin r} \|f\|_{L^{2}(\mathbb{S}^{2})} \sqrt{{\rm Area}(B_{r}(p))}\right) dr \\
& = &  \frac{\|f\|_{L^{2}(\mathbb{S}^{2})}}{\sqrt{2\pi}} \int^{r_{1}}_{r_{0}} \frac{ \sqrt{1-\cos r}}{\sin r} dr \\
& = &  \frac{\|f\|_{L^{2}(\mathbb{S}^{2})}}{\sqrt{2\pi}} \int^{r_{1}}_{r_{0}} \frac{1}{\sqrt{1+\cos r}} dr \\
& \leq & \frac{\|f\|_{L^{2}(\mathbb{S}^{2})}}{\sqrt{2\pi}} \int^{r_{1}}_{r_{0}} 1 dr   \qquad \left( 0< r_{0}< r_{1} \leq \frac{\pi}{2}\right) \\
& = & \frac{\|f\|_{L^{2}(\mathbb{S}^{2})}}{\sqrt{2\pi}}(r_{1} - r_{0}).
\end{eqnarray}
\end{proof}


\subsection{Ball average monotonicity}\label{subsect-ball-average}

In this subsection, we further derive a ball average monotonicity [Proposition \ref{prop: ball average non-increasing}] for a smooth function on $\Sph^2$ satisfying $\Delta f \leq f$. The proof uses the spherical mean inequality [Proposition \ref{prop: spherical mean inequality}] and the trick as in the proof of Bishop-Gromov volume comparison theorem. This ball average monotonicity is used in Proposition \ref{prop: average limit exists} to prove that the ball average limit as $r\to 0$ exists everywhere for the limit function.

\begin{lem}
Let $f$ be a smooth function on $\Sph^2$ satisfying $\Delta f \leq f$ and $\|f\|_{L^{2}(\Sph^2)} \leq C \sqrt{2\pi}$, where $C$ is a  positive constant. For any fixed $x\in \Sph^2$, the spherical mean
\begin{equation}
\fint_{\partial B_{r}(x)} \left( f- C r \right) = \frac{\int_{\partial B_{r}(x)}(f - Cr)}{2\pi \sin r}
\end{equation}
is a non-increasing function in $r$ for $r\in (0,\frac{\pi}{2}]$
\end{lem}
\begin{proof}
The spherical mean inequality in Proposition \ref{prop: spherical mean inequality} says that for any $x\in \Sph^2$ and $0 < r_0 < r_1 \leq \frac{\pi}{2}$,
\begin{equation}
\fint_{\partial B_{1}(x)} f  - \fint_{\partial B_{r_{0}}(x)} f \leq \frac{\|f\|_{L^{2}(\Sph^2)}}{\sqrt{2 \pi}} (r_1 - r_0) \leq C (r_1 - r_0).
\end{equation}
By rearranging this inequality, we obtain that for any fixed $x\in \Sph^2$,
\begin{equation}\label{eqn: spherical mean non-increasing}
\fint_{\partial B_{r_1}(x)} (f - C r_1) \leq \fint_{\partial B_{r_0}(x)} (f - C r_0), \quad \forall 0 < r_0 \leq r_1 \leq \frac{\pi}{2}.
\end{equation}
This completes the proof.
\end{proof}

Combine this spherical mean monotonicity with the trick as in the proof of Bishop-Gromov volume comparison theorem, we obtain the following ball average monotonicity.

\begin{prop}\label{prop: ball average non-increasing}
Let $f$ be a smooth function on $\Sph^2$ satisfying $\Delta f \leq f$ and $\|f\|_{L^{2}(\Sph^2)} \leq C \sqrt{2\pi}$, then $\forall 0 < r < R \leq \frac{\pi}{2}$,
\begin{equation}
\fint_{B_{R}(x)} \left( f(y) - C d(y, x) \right) d{\rm vol}(y) \leq \fint_{B_{r}(x)} \left( f(y) - C d(y, x) \right) d{\rm vol}(y),
\end{equation}
where $d(y, x)$ is the distance between $y$ and $x$ in the standard $\Sph^2$.
\end{prop}
\begin{proof}
{\bf Step 1}.
      \begin{eqnarray}
        &  &  \int_{B_{r}(x)} \left( f(y) - C d(y, x) \right) d{\rm vol}(y)  \\
       & = & \int^{r}_{0} \left( \int_{\partial B_{s}(x)} (f - C s) \right) ds \\
       & = & \int^{r}_{0} (2 \pi \sin s) \left( \fint_{\partial B_{s}(x)} (f - C s) \right) ds \\
       & \geq & \fint_{\partial B_r (x)} (f - C r) \cdot \int^{r}_{0} 2\pi \sin s ds \quad  (\text{by} \ \ (\ref{eqn: spherical mean non-increasing}) \ \  \text{and} \ \  s \leq r ) \\
       & = & {\rm Vol} (B_{r}(x)) \fint_{\partial B_r (x)} (f - C r).
      \end{eqnarray}
      So
      \begin{equation}\label{eqn: ball average}
      \fint_{B_{r}(x)} (f(y) - C d(y, x)) d{\rm vol}(y) \geq \fint_{\partial B_r(x)} (f(y) - C r)
      \end{equation}

{\bf Step 2}. Let $A_{r, R}(x) = B_R(x) \setminus B_r(x)$. Similar as in step 1, we have
      \begin{eqnarray}
      &   & \int_{A_{r, R}(x)} (f(y) - C d(y, x)) d{\rm vol}(y) \\
      & = & \int^{R}_{r} \left( \int_{\partial B_{s}(x)} (f - C s) d\sigma \right) ds  \\
      & = & \int^{R}_{r} (2\pi \sin s) \left( \fint_{\partial B_{s}(x)} (f - C s) d\sigma \right) ds  \\
      & \leq & \fint_{\partial B_{r}(x)} (f- C r) d\sigma \cdot \int^{R}_{r} (2\pi \sin s) ds  \quad (\text{by} \ \ (\ref{eqn: spherical mean non-increasing}) \ \  \text{and} \ \ s \geq r ) \\
      & = & \vol(A_{r, R}(x)) \fint_{\partial B_{r}(x)} (f- C r) d\sigma
      \end{eqnarray}
      So
       \begin{equation}\label{eqn: annulus average}
       \fint_{A_{r, R}(x)} (f(y) - C d(y, x)) d{\rm vol}(y) \leq \fint_{\partial B_{r}(x)} (f - C r) d\sigma.
       \end{equation}

 {\bf Step 3}. By combining (\ref{eqn:  ball average}) and (\ref{eqn: annulus average}), we obtain that for $0 < r < R \leq \frac{\pi}{2}$
      \begin{equation}
      \fint_{A_{r, R}(x)} (f(y) - C d(y, x)) d{\rm vol}(y) \leq   \fint_{B_{r}(x)} (f(y) - C d(y, x)) d{\rm vol}(y).
      \end{equation}

{\bf Step 4}.
      \begin{eqnarray}
       &  & \int_{B_{R}(x)} (f - C d(y, x)) d{\rm vol}(y)   \\
      & = & \int_{B_{r}(x)} (f - C d(y, x)) d {\rm vol}(y) + \int_{A_{r, R}(x)} (f - C d(y, x)) d{\rm vol}(y)  \\
      & \leq & \int_{B_{r}(x)} (f - C d(y, x)) d {\rm vol}(y) \\
      & &  + {\rm Vol}(A_{r, R}(x)) \cdot \fint_{B_{r}(x)} (f(y) - C d(y, x)) d{\rm vol}(y)   \\
      & = & \left( {\rm Vol} (B_{r}(x)) + \vol(A_{r, R}(x)) \right) \fint_{B_{r}(x)} (f(y) - C d(y, x)) d{\rm vol}(y) \\
      & = & {\rm Vol} (B_{R}(x)) \fint_{B_{r}(x)} (f(y) - C d(y, x)) d{\rm vol}(y).
      \end{eqnarray}
      This completes the proof.
\end{proof}



\section{$W^{1, p}$ limit of warping function for $1 \leq p<2$}\label{sect-W1p-limit}

In this section, we study the $L^q$ pre-compactness of a sequence of positive smooth functions $f_j $ satisfying the inequalities
\be\label{eqn-warping-function-condition}
\Delta f_j \leq f_j, \quad \int_{\Sph^2}f_j d {\rm vol}_{\Sph^2} \leq \frac{V}{2\pi}, \quad \forall j \in \N.
\ee
Here $V$ is a positive constant. By Lemmas \ref{lem: nonnegative scalar curvature condition} and \ref{lem: volume upper bound condition}, the inequlities in (\ref{eqn-warping-function-condition})  are equivalent to the requirements that the Riemannian manifolds $ \Sph^2 \times_{f_j} \Sph^1 $ have nonnegative scalar curvature and uniform volume upper bound.

In Subsection \ref{subsect-W1p-limit-function}, we prove that a sequence of positive smooth functions $f_j$ on $\Sph^2$ satisfying requirements in (\ref{eqn-warping-function-condition}) has a convergent subsequence in $L^q(\Sph^2)$ for any $1 \leq q <+\infty$, and that the limit function is in $W^{1, p}(\Sph^2)$ for any $1 \leq p <2$ [Proposition \ref{prop-W1p-limit}].

In Subsection \ref{subsect-lower-semi-continuous-representative}, we apply the ball average monotonicity property obtained in Proposition \ref{prop: ball average non-increasing} to prove that the limit function has a lower semi-continuous representative [Proposition \ref{prop: average limit exists}, Remark \ref{rmrk-lower-semi-continuous}].

\subsection{$W^{1, p}$ limit function for $p<2$}\label{subsect-W1p-limit-function}
We first derive the gradient estimate for the sequence of function $ \ln f_j $ in Lemma \ref{lem-l2-gradient-ln}, which is used to obtain $L^{p}$ estimate for $f_{i}$ by using Moser-Trudinger inequality in Lemma \ref{Lem-Moser-Trudinger}.
\begin{lem}\label{lem-l2-gradient-ln}
Let $\{f_{j }\}_{j=1}^\infty$ be a sequence of positive functions on $\Sph^2$ satisfying
\be
\Delta f_{j } \leq f_j , \quad \forall j \in \N.
\ee
We have
\be
\|\nabla \ln f_j \|_{L^2(\Sph^2)}^2\leq {\rm Vol} (\Sph^2), \quad \forall j \in \N.
\ee

\end{lem}

\begin{proof}
Note that
\be\label{eqn-Laplace-ln}
\Delta \ln f_j=\frac{\Delta f_j}{f_j}-\frac{|\nabla f_j |^2}{f_j^2}.
\ee
By equation (\ref{eqn-Laplace-ln}) and the assumption, we have
\be
|\nabla \ln f_j |^2=\frac{|\nabla f_j |^2}{f_j^2}
=\frac{\Delta f_j }{f_j }-\Delta \ln f_j \leq 1-\Delta \ln f_j .
\ee	
Integrating it over $\Sph^2$, and using Stokes' theorem, we get
\be
\|\nabla \ln f_j \|_{L^2(\Sph^2)}^2=\int_{\Sph^2}|\nabla \ln f_j |^2\leq {\rm Vol} (\Sph^2).
\ee

\end{proof}

\begin{lem} \label{Lem-Moser-Trudinger}
Let $\{f_{j}\}_{j=1}^\infty$ be a sequence of positive functions on $\Sph^2$ satisfying
\be
\Delta f_j \leq f_j, \quad \int_{\Sph^2}f_j d {\rm vol}_{\Sph^2} \leq \frac{V}{2\pi}, \quad \forall j \in \N.
\ee
Then we have
\be
\| f_j \|_{L^p(\Sph^2)}^p\leq 4\pi \exp\left(\frac{Vp}{8\pi^2}+\frac{p^2}{4}\right),
\ee
for all $j \in \N$ and $ p \in [1, +\infty)$.
	
\end{lem}

\begin{proof}
By the Moser-Trudinger inequality (inequality (25) in \cite{Onofri-82}), for any smooth function $\psi:\Sph^2\to \R$ we have
\be\label{eqn-Moser-Trudinger}
\int_{\Sph^2} e^\psi d{\rm vol}_{\Sph^2}
\leq
4\pi \exp\left(\frac{1}{4\pi}\int_{\Sph^2}\left(\psi+\frac{1}{4}|\nabla\psi|^2\right)d {\rm vol}_{\Sph^2}\right).
\ee
Here $\nabla$ is the Levi-Civita connection of the standard metric $g_{\Sph^2}$ and $d{\rm vol}_{\Sph^2}$ is the volume form on $\Sph^2$ with respect to the standard metric $g_{\Sph^2}$. Take $\psi=p\ln f_j$, then we have
\begin{eqnarray}
\| f_j \|_{L^p(\Sph^2)}^p
& = & \int_{\Sph^2} f_j ^p d{\rm vol}_{\Sph^2} \\
& = &
\int_{\Sph^2} e^{p\ln f_j }d{\rm vol}_{\Sph^2} \\
& \leq  &
4\pi \exp\left(\frac{1}{4\pi}\int_{\Sph^2} \left(p\ln f_j +\frac{p^2}{4}|\nabla \ln f_j |^2\right)d{\rm vol}_{\Sph^2}\right).
\end{eqnarray}
By the fact that $\ln x \leq x, \forall x>0$, we have
\be
\int_{\Sph^2}\ln f_j \leq \int_{\Sph^2} f_j \leq \frac{V}{2\pi}.
\ee
On the hand, by Lemma \ref{lem-l2-gradient-ln} we have
\be\int_{\Sph^2}|\nabla \ln f_j |^2 \leq \vol (\Sph^2)=4\pi.\ee
This completes the proof.
\end{proof}

Next, we show that such sequence of function is uniformly bounded in $W^{1,p}(\Sph^2)$ for $p\in [1,2)$.
\begin{lem}\label{lem-Sobolev-norm-bound}
Let $\{ f_j \}_{j=1}^\infty$ be a sequence of positive functions on $\Sph^2$ satisfying \be
\Delta f_j \leq f_j, \quad \int_{\Sph^2}f_j  d {\rm vol}_{\Sph^2} \leq \frac{V}{2\pi}, \quad \forall j \in \N.
\ee
Then the sequence is uniformly bounded in $W^{1, p}(\Sph^2)$ for $p \in [1, 2)$, i.e. for each $p \in [1, 2)$, there exists a constant $C(p)$ such that
\be
\|f_j \|_{W^{1, p}(\Sph^2)} \leq C(p), \quad \forall j \in \N.
\ee
\end{lem}

\begin{proof}
For any $1\leq p <2$,
 \begin{equation}
 |\nabla f_{j} |^{p} = |\nabla \ln  f_{j}  |^{p} \cdot  | f_{j}  |^{p}.
 \end{equation}
 The Cauchy-Schwarz inequality implies that
 \begin{eqnarray}
 & & \|\nabla f_{j} \|_{L^{p}(\mathbb{S}^{2})} \\
 & = & \left(\int_{\mathbb{S}^{2}} |\nabla \ln  f_{j}  |^{p} \cdot  | f_{j}  |^{p}\right)^{\frac{1}{p}} \\
 & \leq & \|\nabla \ln f_{j} \|_{L^{2}(\mathbb{S}^{2})} \cdot  \| f_{j}  \|_{L^{\frac{2p}{2-p}}(\mathbb{S}^{2})} \\
  & \leq & \|\nabla \ln f_{j} \|_{L^{2}(\mathbb{S}^{2})} \cdot \left(  \|f_{j} \|_{L^{\frac{2p}{2-p}}(\mathbb{S}^{2})} + {\rm Vol}(\Sph^2) \right) \\
  & \leq & \left(\vol(\mathbb{S}^{2})\right)^{\frac{1}{2}} \left(  (4\pi)^{\frac{2-p}{2p}} \exp\left(\frac{V}{8\pi^{2}} + \frac{p}{2(2-p)}\right) + {\rm Vol}(\Sph^2) \right).
  \end{eqnarray}
  Here in the last step, we used Lemma \ref{lem-l2-gradient-ln} and Lemma \ref{Lem-Moser-Trudinger}.
  Moreover, by Lemma~\ref{Lem-Moser-Trudinger} again, for each $p \in [1, 2)$, $\| f_j \|_{L^{p}(\Sph^2)}$ is uniformly bounded for all $j \in \N$.
  Hence for each $p \in [1, 2)$, $\|f_j \|_{W^{1, p}(\Sph^2)}$ is uniformly bounded for all $j \in \N$.
  \end{proof}

We use the uniform $W^{1,p}(\Sph^2)$ bound to prove convergence in the following lemma.
\begin{lem}\label{lem-W1p-limit}
Let $\{ f_j \}_{j=1}^\infty$ be a sequence of positive functions on $\Sph^2$ satisfying
\be
\Delta f_j \leq f_j, \quad \int_{\Sph^2}f_j d {\rm vol}_{\Sph^2} \leq \frac{V}{2\pi}, \quad \forall j \in \N.
\ee
Then for each fixed $p \in [1, 2)$, there exists a subsequence $\{ f_{j^{(p)}_k} \}_{k=1}^\infty$ and   $f_{\infty, p} \in W^{1, p}(\mathbb{S}^2) $  such that
\be
f_{j^{(p)}_k} \rightarrow f_{\infty, p}, \quad \text{in} \ \ L^{q}(\Sph^2),
\ee
for each $1 \leq q < \frac{2p}{2-p}$.

Moreover, for any $\varphi \in C^{\infty}(\Sph^2)$,
\begin{equation*}
\int_{\Sph^2} \left( f_{j^{(p)}_k} \varphi + \langle \nabla f_{j^{(p)}_k}, \nabla \varphi \rangle \right) d\vol_{g_{\Sph^2}} \rightarrow \int_{\Sph^2} \left( f_{\infty, p} \varphi + \langle \nabla f_{\infty, p}, \nabla \varphi \right) d\vol_{g_{\Sph^2}},
\end{equation*}
as $j^{(p)}_{k} \rightarrow \infty$, where $\nabla f_{\infty, p}$ is the weak gradient of $f_{\infty, p}$.

\end{lem}

\begin{proof}
For each fixed $p \in [1, 2)$, by using Rellich-Kondrachov compactness theorem, the uniform estimate of Sobolev norms in Lemma \ref{lem-Sobolev-norm-bound} implies that there exists a subsequence of $\{ f_j \}$, which is still denoted by $\{ f_j \}$, converging to $f_{\infty, p}$ in $L^{q}(\Sph^2)$ for $1 \leq q < \frac{2p}{2-p}$. Then by the weak compactness in $L^p$ space (see, e.g. Theorem 1.42 in \cite{EG-text}), we can obtain that $f_{\infty, p} \in W^{1, p}(\Sph^2)$. Indeed, $\|f_j\|_{W^{1, p}(\Sph^2)} \leq C$ for all $j \in \N$ implies that $\| f_j \|_{L^{p}(\Sph^2)}$ and $\|\nabla f_j\|_{L^{p}(\Sph^2)}$ are both uniformly bounded. Then the weak compactness in $L^p$ space implies that there exist a further subsequence, denoted by $f_{j^{(p)}_k}$, and $X \in L^{p}(\Sph^2, {\rm T}\Sph^2)$ such that
\be
\nabla f_{j^{(p)}_k} \rightharpoonup X \quad \text{in} \ \ L^{p}(\Sph^2, {\rm T}\Sph^2),
\ee
i.e.
\be\label{eqn-gradient-weak-convergence}
\int_{\Sph^2} \langle \nabla f_{j^{(p)}_k} , Y \rangle d\vol_{g_{\Sph^2}} \rightarrow \int_{\Sph^2} \langle X, Y \rangle d\vol_{g_{\Sph^2}}, \quad \forall Y \in C^{\infty}(\Sph^2, {\rm T}\Sph^2).
\ee
On the other hand,
\be
\int_{\Sph^2} \langle \nabla f_{j^{(p)}_k} , Y \rangle d\vol_{g_{\Sph^2}} = \int_{\Sph^2} f_j {\rm div}Y d\vol_{g_{\Sph^2}} \rightarrow \int_{\Sph^2} f_{\infty, p} {\rm div} Y d\vol_{g_{\Sph^2}},
\ee
since $f_{j^{(p)}_k} \rightarrow f_{\infty, p}$ in $L^{p}$. Thus,
\be
\int_{\Sph^2} f_{\infty, p} {\rm div} Y d\vol_{g_{\Sph^2}} = \int_{\Sph^2} \langle X, Y \rangle d\vol_{g_{\Sph^2}}, \quad \forall Y \in C^{\infty}(\Sph^2, {\rm T}\Sph^2).
\ee
Therefore, $X = \nabla f_{\infty, p}$ is the gradient of $f_{\infty, p}$ in the sense of distribution, and so $f_{\infty, p} \in W^{1, p}(\Sph^2, g_{\Sph^2})$.
For any $\varphi \in C^{\infty}(\Sph^2)$, by taking $Y = \nabla \varphi$ in (\ref{eqn-gradient-weak-convergence}), we obtain
\be
\int_{\Sph^2} \left( f_{j^{(p)}_k} \varphi + \langle \nabla f_{j^{(p)}_k}, \nabla \varphi \rangle \right) d\vol_{g_{\Sph^2}} \rightarrow \int_{\Sph^2} \left( f_{\infty, p} \varphi + \langle \nabla f_{\infty, p}, \nabla \varphi \right) d\vol_{g_{\Sph^2}}.
\ee
\end{proof}

Now we use Lemma \ref{lem-W1p-limit} and diagonal argument to find a subsequence converging in $L^q$ for all $q\geq 1$ and prove the following proposition:

\begin{prop}\label{prop-W1p-limit}
Let $\{ f_j \}_{j=1}^\infty$ be a sequence of positive functions on $\Sph^2$ satisfying \be
\Delta f_j \leq f_j, \quad \int_{\Sph^2}f_j d {\rm vol}_{\Sph^2} \leq \frac{V}{2\pi}, \quad \forall j \in \N.
\ee
Then there exists a subsequence $\{ f_{j_k} \}_{k=1}^\infty$ and   $f_{\infty} \in W^{1, p}(\mathbb{S}^2)$ for all $p \in [1, 2)$,  such that
\be
f_{j_k} \rightarrow f_{\infty}, \quad \text{in} \ \ L^{q}(\Sph^2), \ \  \forall q \in [1, \infty).
\ee

Moreover, for any $\varphi \in C^{\infty}(\Sph^2)$,
\be\label{eqn-weak-convergence}
\int_{\Sph^2} \left( f_{j_k} \varphi + \langle \nabla f_{j_k}, \nabla \varphi \rangle \right) d\vol_{g_{\Sph^2}} \rightarrow \int_{\Sph^2} \left( f_{\infty} \varphi + \langle \nabla f_{\infty}, \nabla \varphi\rangle \right) d\vol_{g_{\Sph^2}},
\ee
as $j_{k} \rightarrow \infty$, where $\nabla f_{\infty}$ is the weak gradient of $f_{\infty}$.
\end{prop}
\begin{proof}
The proof is a diagonal argument. We apply Lemma \ref{lem-W1p-limit} for $p = 2- \frac{1}{n+1}, n=1, 2, 3, \dots$.

For $n=1$, by applying Lemma \ref{lem-W1p-limit} to $\{f_j\}_{j=1}^\infty$ and $p =2 - \frac{1}{2}$, we obtain a subsequence, denoted by $f_{j^{(1)}_k, 1}$, and $f_{\infty, 1} \in W^{1, 2-\frac{1}{2}}$ such that
\be
f_{j^{(1)}_k, 1} \rightarrow f_{\infty, 1} \ \ \text{in} \ \ L^{q}(\Sph^2), \ \ \forall 1 \leq q < 6, \ \ \text{as} \ \ k \rightarrow \infty.
\ee

For $n=2$, by applying Lemma \ref{lem-W1p-limit} to the subsequence $ \left\{ f_{j^{(1)}_k, 1} \right\}_{k=1}^\infty$ and $p = 2 - \frac{1}{3}$, we obtain a subsequence, $\left\{ f_{j^{(2)}_k, 2} \right\}_{k=1}^\infty \subset \left\{ f_{j^{(1)}_k, 1} \right\}_{k=1}^\infty$, and $f_{\infty, 2} \in W^{1, 2 - \frac{1}{3}}$ such that
\be
f_{j^{(2)}_k, 2} \rightarrow f_{\infty, 2} \ \ \text{in} \ \ L^{q}(\Sph^2), \ \ \forall 1 \leq q < 10, \ \ \text{as} \ \ k \rightarrow \infty.
\ee

Then by repeating this process for $n=3, 4, 5, \dots$, we can obtain a family of decreasing subsequence $\left\{ f_{j^{(n)}_k, n} \right\}_{k=1}^\infty \subset \left\{ f_{j^{(n-1)}_k, n-1} \right\}_{k=1}^\infty$ and $f_{\infty, n} \in W^{1, 2 - \frac{1}{n+1}}$ for all $n \in \N$, such that for each fixed $n \in \N$
\be
f_{j^{(n)}_k, n} \rightarrow f_{\infty, n} \ \ \text{in} \ \ L^{q}(\Sph^2), \ \ \forall 1 \leq q < 4n+2, \ \ \text{as} \ \ k \rightarrow \infty.
\ee

Now we take the diagonal subsequence $\left\{ f_{j_k} := f_{f_{j^{(k)}_k, k}}  \mid k \in \N \right\}$. By the construction of $f_{j_k}$ and $4k+2 \rightarrow +\infty$ as $k\rightarrow +\infty$, we have that $\{f_{j_k}\}$ is a Cauchy sequence in $L^{q}(\Sph^2)$ for all $q \in [1, \infty)$. Thus there exists $f_\infty \in L^{q}(\Sph^2)$ such that
\be
f_{j_k} \rightarrow f_{\infty} \ \ \text{in} \ \ \in L^{q}(\Sph^2), \ \ \text{as} \ \ k \rightarrow \infty, \ \ \forall q \in [1, \infty).
\ee
Then by the uniqueness of $L^2$ limit, $f_\infty = f_{\infty, n}$ in $L^{2}(\Sph^2)$ for all $n \in \N$. Furthermore, because $f_{\infty, n} \in W^{1, 2 - \frac{1}{n+1}} (\Sph^2)$ and $ 2 - \frac{1}{n+1} \rightarrow 2^-$ as $n \rightarrow \infty$, we see that the $L^{p}$ norm of the weak derivative of $f_\infty$ is bounded for any $p \in [1, 2)$. Thus $f_\infty \in W^{1, p} (\Sph^2)$ for all $p \in [1, 2)$.

Finally, the last claim in (\ref{eqn-weak-convergence}) follows from that $\left\{ f_{j_k} \right\}_{k=1}^\infty \subset \left\{ f_{j^{(1)}_k, 1} \right\}_{k=1}^\infty$ and the corresponding convergence in Lemma \ref{lem-W1p-limit} for $p = 2- \frac{1}{2}$, in particular for the subsequence $\left\{ f_{j^{(1)}_k, 1} \right\}_{k=1}^\infty$.
\end{proof}

\begin{rmrk}\label{rmrk-best-regularity}
{\rm
The extreme example constructed by Christina Sormani and authors in \cite{STW-ex} shows that $W^{1, p}$ regularity for $p<2$ is the best regularity we can expect for $f_\infty$ in general (see Lemma 3.4 in \cite{STW-ex}).
}
\end{rmrk}

\subsection{Lower semi-continuous representative of the limit function}\label{subsect-lower-semi-continuous-representative}

For the limit function $f_\infty$ obtained in Proposition \ref{prop-W1p-limit}, Lebesgue-Besicovitch differential theorem implies that
\begin{equation}\label{eqn-LBDT}
\lim_{r\rightarrow 0} \fint_{B_{r}(x)} f_{\infty} d\vol_{g_{\Sph^2}} = f_\infty (x)
\end{equation}
holds for a.e. $x \in \Sph^2$ with respect to the volume measure $d\vol_{g_{\Sph^2}}$. In Proposition \ref{prop: average limit exists}, by applying the ball average monotonicity property in Proposition \ref{prop: ball average non-increasing}, we will show that the limit of ball average in (\ref{eqn-LBDT}) actually exists for all $x \in \Sph^2$, and that the limit produces a lower semi-continuous function.

\begin{prop}\label{prop: average limit exists}
Let $\{f_{j}\}_{j=1}^\infty$ be a sequence of smooth positive functions on $\Sph^2$ satisfying
\be
\Delta f_j \leq f_j, \quad \int_{\Sph^2}f_j d {\rm vol}_{\Sph^2} \leq \frac{V}{2\pi}, \quad \forall j \in \N.
\ee
Then the limit function, $f_\infty$,  obtained in Proposition \ref{prop-W1p-limit},  has the following properties.
\begin{enumerate}[(i)]
 \item For each fixed $x \in \Sph^2$,
            the ball average
           \be
           \fint_{B_{r}(x)} \left( f_\infty (y) - C d(y, x) \right) d{\rm vol}(y)
            \ee
            is non-increasing in $r \in \left(0, \frac{\pi}{2} \right)$,
             where $C$ is a positive real number such that $\sup_{j\in \N}\| f_ j \|_{L^{2}(\Sph^2)} \leq C \sqrt{2\pi}$.
             Note that the existence of such $C$ is guaranteed by Lemma \ref{Lem-Moser-Trudinger}.
\item Consequently, the limit
\be
\overline{f_\infty}(x) := \lim_{r \rightarrow 0} \fint_{B_{r}(x)} f_{\infty} = \lim_{r \rightarrow 0 } \fint_{B_{r}(x)} \left( f_\infty (y) - C d(y, x) \right) d{\rm vol}(y)
\ee
exists, allowing $+\infty$ as a limit, for every $x \in \Sph^2$. Moreover, $\overline{f_{\infty}}$ is a lower semi-continuous function on $\Sph^2$.
\end{enumerate}
\end{prop}

\begin{proof}
By Lemma \ref{Lem-Moser-Trudinger}, there exists $C \in \R$ such that
\be
\|f_j\|_{L^{2}(\Sph^2)} \leq C \sqrt{2\pi}, \ \ \forall j \in \N.
\ee
Then by applying Proposition \ref{prop: ball average non-increasing} to functions $f_j$, we obtain that for any fixed $x\in \Sph^2$
\begin{equation}
\fint_{B_{R}(x)} (f_j (y) - K d(y, x)) d{\rm vol}(y) \leq \fint_{B_{r}(x)} (f_j (y) - C d(y, x)) d{\rm vol}(y)
\end{equation}
holds for any $0 < r < R < \frac{\pi}{2}$ and all $j \in \N$.

By Proposition \ref{prop-W1p-limit} $f_j \rightarrow f_\infty$ in $L^{1}(\Sph^2)$.
Then for any fixed $x\in \Sph^2$, and any fixed $0 < r < R  < \frac{\pi}{2}$, by taking the limit as $j \rightarrow +\infty$, we obtain
\begin{equation}\label{eqn: ball average monotonicity}
\fint_{B_{R}(x)} \left( f_\infty (y) - C d(y, x) \right) d{\rm vol}(y) \leq \fint_{B_{r}(x)} \left( f_\infty (y) - C d(y, x) \right) d{\rm vol}(y),
\end{equation}
 So for each fixed $x \in \Sph^2$, the ball average
\be
\fint_{B_{r}(x)} \left( f_\infty (y) - C d(y, x) \right) d{\rm vol}(y)
\ee
is non-increasing for $r \in \left(0, \frac{\pi}{2} \right)$. Therefore, for any $x \in \Sph^2$ the limit
\begin{equation}
\lim_{r\rightarrow 0}  \fint_{B_{r}(x)} \left( f_\infty (y) - C d(y, x) \right) d{\rm vol}(y)
\end{equation}
exists as a finite number or $+\infty$.

On the other hand, by direct calculation
\begin{equation}
\fint_{B_{r}(x)} d(y, x) d{\rm vol}(y) = \frac{\int^{r}_{0} 2\pi s \sin s ds}{\int^{r}_{0} 2\pi \sin (s) ds} = \frac{\sin r - r \cos r}{1 - \cos r} \rightarrow 0,
\end{equation}
as $r \rightarrow 0$. Thus the limit
\be
\overline{f_{\infty}}(x) := \lim_{r \rightarrow 0} \fint_{B_{r}(x)} f_{\infty}
                                  = \lim_{r\rightarrow 0}  \fint_{B_{r}(x)} \left( f_\infty (y) - C d(y, x) \right) d{\rm vol}(y)
\ee
exists for all $x \in \Sph^2$.

For each fixed $0< r < \frac{\pi}{2}$, we have that $\fint_{B_{r}(x)} (f_\infty (y) - C d(y, x)) d{\rm vol}(y)$ is a continuous function of $x \in \Sph^2$, since $f_\infty \in L^2(\Sph^2)$, $C d(y, x)\leq C\pi$, and ${\rm Area}(B_{r}(x)) = 2\pi \sin r$ for all $x \in \Sph^2$. Then by the monotonicity in (\ref{eqn: ball average monotonicity}), we have
\be
\overline{f_\infty}(x) = \sup_{r>0}  \fint_{B_{r}(x)} \left( f_\infty (y) - C d(y, x) \right) d{\rm vol}(y).
\ee
In other words, $\overline{f_{\infty}}$ is the supremum of a sequence of continuous function. Thus $\overline{f_\infty}$ is lower semi-continuous.
\end{proof}

\begin{rmrk}\label{rmrk-lower-semi-continuous}
{\rm
Recall that by (\ref{eqn-LBDT}), $\lim\limits_{r\rightarrow 0} \fint_{B_{r}(x)} f_{\infty} d\vol_{g_{\Sph^2}} = f_\infty (x)$ hold for a.e. $x \in \Sph^2$, thus $\overline{f_\infty} (x) = f_\infty (x)$ holds for a.e. $x \in \Sph^2$. So as a $W^{1, p}$ function, $f_\infty$ has a lower semi-continuous representative $\overline{f_\infty}$.
}
\end{rmrk}


\section{Positivity of the limit warping functions}\label{sect-positivity-limit}

In this section, we prove that the limit warping function $f_\infty$ has a positive essential infimum, provided that the Riemannian manifold $\Sph^2 \times_{f_j} \Sph^1$ satisfies both requirements in (\ref{eqn-warping-function-condition}) and the $\mina$ condition [Theorem \ref{thm: f_infty positive}]. The main tools we use in the proof of Theorem \ref{thm: f_infty positive} include the maximum principle, the Min-Max minimal surface theory of Marques and Neves, and the spherical mean inequality we obtained in Proposition \ref{prop: spherical mean inequality}.

The maximum principle for weak solutions (Theorem 8.19 in \cite{GT-PDE-book}) requires $W^{1,2}$ regularity, but in general we only have $f_\infty \in W^{1, p}(\Sph^2)$ for $p<2$ [Remark \ref{rmrk-best-regularity}]. To overcome this difficulty, in Subsection \ref{subsect-cutoff}, we consider the truncation of warping functions $\bar{f}^K_j$ as defined in Definition \ref{defn-cut-off}, and obtain a $W^{1, 2}(\Sph^2)$ limit function $\bar{f}^K_\infty$ for the sequence of truncated function $\bar{f}^K_j$ [Lemma \ref{Cutoff-Lemma-Limit}]. This enables us to apply maximum principle for weak solutions  (Theorem 8.19 in \cite{GT-PDE-book}) to $\bar{f}^K_\infty$, and prove that either $\inf \bar{f}^K_\infty >0$ or $\bar{f}^K_\infty \equiv 0$ on $\Sph^2$ [Proposition \ref{prop-cutoff-infimum}].

In Subsection \ref{subsect-bound-mina}, we use Min-Max minimal surface theory of Marques and Neves and the spherical mean inequality in Proposition \ref{prop: spherical mean inequality} to obtain an upper bound for $\mina(\Sph \times_f \Sph^1)$ in terms of $L^1$ norm of the warping function $f$, provided that the $L^2$ norm of $f$ is sufficiently small [Proposition \ref{prop-L1-bound-mina-from-upper}].

In Subsection \ref{subsect-positivity-limit}, we use Proposition \ref{prop-cutoff-infimum} and Proposition \ref{prop-L1-bound-mina-from-upper} to prove Theorem \ref{thm: f_infty positive}. Moreover, as an application of Theorem \ref{thm: f_infty positive}, we obtain a positive uniform lower bound for warping functions $f_j$, if the warped product manifolds $\Sph^2 \times_{f_j} \Sph^1$ satisfy requirements in (\ref{eqn-warping-function-condition}) and the $\mina$ condition [Proposition \ref{prop: f_j lower bound}].


\subsection{$ W^{1, 2} $ regularity of limit of truncated warping functions}\label{subsect-cutoff}

We define the truncation of a function firstly:
\bd\label{defn-cut-off}
Let $f: \Sph^2 \to \R$ be a positive smooth function. Let $K>0$ be a real number,  for each $x \in \Sph^2$, we define
\be
\bar{f}^K(x)=
\begin{cases}
f(x), & \text{ if }  \ \ f(x)<K,\\
K, & \text{ if } \ \ f(x)\geq K.
\end{cases}
\ee
Then $\bar{f}^K$ is a positive continuous function on $\Sph^2$ with the maximal value not greater than $K$.
\ed
From the definition we can prove the following lemma:
\bl\label{Lemma-Cutoff-Single-Function}
Let $f: \Sph^2 \to \R$ be a positive smooth function, and let $K>0$ be a regular value of the function $f$. If
\be
\Delta f\leq f
\ee
then for all $u\in W^{1,2}(\Sph^2)$ such that $u\geq 0$ we have
\be
-\int_{\Sph^2} \langle \nabla u, \nabla \bar{f}^K \rangle \leq \int_{\Sph^2} u \bar{f}^K.
\ee
\el

\begin{proof}
By Theorem 4.4 from \cite{EG-text}, we have for all $K>0$
\be
\nabla \bar{f}^K=\begin{cases}
\nabla f,  & \text{a.e. on } \{f(x)<K\},\\
0, & \text{a.e. on } \{f (x)\geq K\}.
\end{cases}
\ee
As a result we have
\be
\begin{split}
-\int_{\Sph^2} \langle \nabla u, \nabla \bar{f}^K \rangle &=-\int_{\{ f<K\}} \langle \nabla u, \nabla f \rangle\\
&=\int_{\{f <K\}} u \Delta f-\int_{\partial\{f<K\}} u \partial_{\nu} f.
\end{split}
\ee
Here, since $K$ is a regular value of $f$, from the Regular Level Set Theorem we know that the level set $ \{f=K\}=\partial \{f<K\}$ is am embedded submanifold of dimension $1$ in $\Sph^2$. Hence we can apply Stokes' theorem to get the last step. Moreover, since $\nu $ is the outer unit normal vector on the boundary of the set $ \{f<K\}$, we have
\be
 \partial_{\nu} f \geq 0.
\ee
Hence we can drop the boundary term to get the inequality
\be
-\int_{\Sph^2} \langle \nabla u, \nabla \bar{f}^K \rangle \leq \int_{\{f <K\}} u \Delta f.
\ee
Since
\be
 \Delta f \leq f,
\ee
we have
\be
-\int_{\Sph^2} \langle \nabla u, \nabla \bar{f}^K \rangle\leq  \int_{\{f <K\}} u \Delta f \leq  \int_{\{f <K\}} u f \leq   \int_{\Sph^2} u \bar{f }^K.
\ee
This finishes the proof.
\end{proof}

We can prove similar results for a sequence of functions:
\bl\label{Cutoff-Lemma-Sard}
Let $\{f_j\}_{j=1}^\infty$ be a sequence of smooth positive function defined on $\Sph^2$. If
\be
\Delta f_j \leq f_j, \ \ \forall j \in \N,
\ee
then there exists $K>0$ such that for all $u\in W^{1,2}(\Sph^2)$ with $u\geq 0$ we have
\be\label{Cutoff-Inequality}
-\int_{\Sph^2} \langle \nabla u,\nabla \bar{f}_j^K \rangle \leq \int_{\Sph^2} u\bar{f}_j^K \quad \forall j \in \N .
\ee
Moreover, we can choose $K$ as large as we want.
\el

\begin{proof}
Note that if $0<K \leq \inf\limits_{x\in \Sph^2} f_j(x)$ for some $i$ then we have $\bar{f}_j^K(x)=K$. On the other hand, if $\sup\limits_{x\in\Sph^2}f(x) \leq K$ for some $i$ then $\bar{f}_j^K(x)=f_j(x)$. Either way the inequality (\ref{Cutoff-Inequality}) holds.

In general, by Sard's theorem, for each function $f_j$, the critical values of $f_j$ has measure zero, and the union of all the critical sets for each of the function also has measure zero. As a result, there exists $K>0$ such that for each $f_j$ either $K$ is a regular value or $f_j^{-1}(\{K\})=\emptyset$. By Lemma \ref{Lemma-Cutoff-Single-Function} we get inequality (\ref{Cutoff-Inequality}). Moreover, we can choose $K$ as large as we want. This finishes the proof.
\end{proof}

Next we prove similar results for the limit function, but before that we need to consider the regularity of the limit function:
\bl\label{Cutoff-Lemma-Limit}
Let $K>0$ be a real number. Let $\{f_{j}\}_{j=1}^\infty$ be a sequence of positive smooth functions on $\Sph^2$ satisfying
\be
\Delta f_{j} \leq f_j, \quad \forall j \in \N.
\ee
Then the sequence $\{\bar{f}^K_j\}_{j=1}^\infty$ is uniformly bounded in $W^{1,2}(\Sph^2)$:
\be
\|\bar{f}^K_j \|_{W^{1,2}(\Sph^2)}\leq 2K \vol(\Sph^2).
\ee
As a result, there exists $\bar{f}_{\infty}^K\in W^{1,2}(\Sph^2)$ such that $\bar{f}^K_j $ converges to $\bar{f}^K_\infty $ in $L^2(\Sph^2)$, and that $\bar{f}^K_j $ converges to $\bar{f}^K_\infty $ weakly in $W^{1,2}(\Sph^2)$.
\el

\begin{proof}
By definition of the cutoff in Definition \ref{defn-cut-off}, we get
\be\label{Cutoff-L2}
\|\bar{f}_j^K\|_{L^{2}(\Sph^2)}\leq K \sqrt{\vol(\Sph^2)}.
\ee
By Theorem 4.4 from \cite{EG-text}, we have for all $K>0$ and for each $i$
\be
\nabla \bar{f}^K_j=\begin{cases}
\nabla f_j, & \text{a.e. on } \{f_j (x)<K\},\\
0,  & \text{a.e. on } \{f_j (x)\geq K\}.
\end{cases}
\ee
Hence
\be\label{Cutoff-Gradient}
\begin{split}
\|\nabla \bar{f}_j\|^2_{L^2(\Sph^2)} &=\int_{\{f_j <K\}} |\nabla f_j|^2 \\
&=\int_{\{f_i <K\}} |f_j |^2|\nabla \ln f_j |^2\\
&\leq K^2\int_{\{f_j <K\}}  |\nabla \ln f_j |^2\\
&\leq K^2\|\nabla \ln f_j\|^2\\
&\leq K^2 \vol(\Sph^2),
\end{split}
\ee
where the last step follows from Lemma \ref{lem-l2-gradient-ln}. Combine inequalities (\ref{Cutoff-L2}) and (\ref{Cutoff-Gradient}) then we get the desired results.
\end{proof}

Now we prove the following proposition concerning the limit function:
\begin{lem}\label{lem-differential inequality for limit function}
Let $\{f_{j}\}_{j=1}^\infty$ be a sequence of positive smooth functions on $\Sph^2$ satisfying
\be
\Delta f_{j} \leq f_j, \quad \forall j \in \N.
\ee
Let $K>0$ be a real number that satisfies the requirement in Lemma \ref{Cutoff-Lemma-Sard}.
Let $\bar{f}^K_{\infty} \in W^{1, 2}(\Sph^2)$ be the limit function as in Lemma \ref{Cutoff-Lemma-Limit}. Then $\bar{f}^K_{\infty}$ satisfies the inequality
\be
-\int_{\Sph^2} \langle \nabla u,\nabla \bar{f}_{\infty}^K \rangle \leq \int_{\Sph^2} u\bar{f}_{\infty}^K,
\ee
for all $u\in W^{1,2}(\Sph^2)$ such that $u\geq 0$.
\end{lem}
\begin{proof}
By Lemma \ref{Cutoff-Lemma-Limit} we know that $\bar{f}^K_j $ converges to $\bar{f}^K_\infty $ in $L^2(\Sph^2)$, and that $\bar{f}^K_j $ converges to $\bar{f}^K_\infty $ weakly in $W^{1,2}(\Sph^2)$. As a result, for any $u\in W^{1,2}(\Sph^2)$ we have that
\be
\int_{\Sph^2} u\bar{f}_{j}^K \to \int_{\Sph^2} u\bar{f}_{\infty}^K,  \ \ \text{ as } \ \ j\to \infty,
\ee
and that
\be
\int_{\Sph^2} \langle \nabla u,\nabla \bar{f}_{j}^K \rangle \to \int_{\Sph^2} \langle \nabla u,\nabla \bar{f}_{\infty}^K \rangle, \ \  \text{ as } j\to\infty.
\ee
As a result, by (\ref{Cutoff-Inequality}) we have for all $u\in W^{1,2}(\Sph^2)$ such that $u\geq 0$
\be
- \int_{\Sph^2} \langle \nabla u,\nabla \bar{f}_{\infty}^K\rangle \leq \int_{\Sph^2} u\bar{f}_{\infty}^K .
\ee
Hence by Theorem 8.19 in \cite{GT-PDE-book}, we have that either the essential infimum of $\bar{f}^K_\infty$ is bounded away from zero or $\bar{f}^K_\infty$ is the zero function. This finishes the proof.
\end{proof}

We need the definition of essential infimum of a function:
\begin{defn}\label{Defn-Ess-Inf}
Consider the standard $\Sph^2$ and use $m$ to denote the standard volume measure in $\Sph^2$. Let $U$ be an open subset of $\Sph^2$ . Let $f: U \to \R$ be measurable. Define the set
\be
U_f^{ess}=\{a\in \R: m(f^{-1}(-\infty,a))=0\}.
\ee
We use $\inf_{U} f$ to denote the essential infimum of $f$ in $U$ and define
\be
\inf_{U} f=\sup U^{ess}_f
\ee
\end{defn}

Finally, we apply the maximum principle for weak solution to prove the following property for the essential infimum of  $f_\infty$.
\begin{prop}\label{prop-cutoff-infimum}
Let $\{f_{j}\}_{j=1}^\infty$ be a sequence of positive smooth functions on $\Sph^2$ satisfying
\be
\Delta f_{j} \leq f_j, \quad \forall j \in \N.
\ee
If we further assume that $f_j\to f_\infty$ in $L^2(\Sph^2)$ for some $f_\infty$, then either the essential infimum of $f_\infty$ is bounded away from zero or $f_\infty = 0$ a.e. on $\Sph^2$.
\end{prop}
\begin{proof}
Since $\|f_j-f_\infty\|_{L^2(\Sph^2)}\to 0$ as $j\to\infty$, choose a subsequence if needed, then we have $f_j\to f_\infty$ poiintwise almost everywhere in $\Sph^2$. Let $K>0$ be a real number that satisfies the requirement in Lemma \ref{Cutoff-Lemma-Sard}. Construct a truncated sequence $\{\bar{f}^K_j\}_{j=1}^\infty$ as in Definition \ref{defn-cut-off}.
By Lemma \ref{Cutoff-Lemma-Limit}, choose a subsequence if needed, there exists $\bar{f}^K_{\infty}\in W^{1,2}(\Sph^2)$ such that $\bar{f}^K_j$ converges to $\bar{f}^K_\infty$ in $L^2(\Sph^2)$ norm. As a result, choose a subsequence if needed we have $\bar{f}_j^K\to \bar{f}_\infty^K$ pointwise almost everywhere in $\Sph^2$.

It suffices to show that if the essential infimum $\inf\limits_{\Sph^2} f_\infty = 0$ then $\bar{f}^K_\infty  = f_\infty = 0$ in $\Sph^2$. We assume that $\inf\limits_{\Sph^2} f_\infty = 0$. Since for each $j$ we have $0<\bar{f}_j^K\leq f_j$, we have $0\leq \inf\limits_{\Sph^2} \bar{f}_j^K \leq \inf\limits_{\Sph^2} f_\infty =0$.
This implies that for any $\delta, \delta^{\prime}>0$, we have
\be
m\left(\left(\bar{f}^K_\infty\right)^{-1}(-\infty, \delta)\right) >0,
\ee
and
\be
m\left(\left(\bar{f}^K_\infty \right)^{-1}(-\infty, -\delta^{\prime})\right)=0.
\ee
Let $N$ be the north pole of $\mathbb{S}^{2}$, and $S$ be the south pole.  $B_{\frac{\pi}{2}}(N)$ and $B_{\frac{\pi}{2}}(S)$ are upper and lower hemispheres respectively. Then either
\be
\inf\limits_{B_{\frac{\pi}{2}}(N)} \bar{f}^K_\infty=0,
\ee
or
\be
\inf\limits_{B_{\frac{\pi}{2}}(S)} \bar{f}^K_\infty=0.
\ee
Without loss of generality we assume that $\inf\limits_{B_{\frac{\pi}{2}}(N)}\bar{f}^K_\infty=0$. Since $\bar{f}^K_\infty\geq 0$ in $\Sph^2$, for any $r>\frac{\pi}{2}$, and $\epsilon>0$ such that $r+\epsilon <\pi$ we have
\be\label{eqn: inf-equal}
\inf_{B_{r}(N)} \bar{f}^K_\infty = \inf_{B_{r+\epsilon}(N)}\bar{f}^K_\infty=0.
\ee

Now by Lemma \ref{lem-differential inequality for limit function}, $\bar{f}^K_\infty$ satisfies
\be
(\Delta -1)\bar{f}^K_\infty \leq 0,
\ee
on $B_{r+\epsilon}(N)$ in the weak sense.
Hence by the strong maximum principle for weak solutions (see Theorem 8.19 in \cite{GT-PDE-book}), the equality in (\ref{eqn: inf-equal}) implies that $\bar{f}^K_{\infty}$ is constant on $B_{r}(N)$. This is true for any $r>\frac{\pi}{2}$, thus $\bar{f}^K_\infty \equiv0$ on $\mathbb{S}^{2}$. Moreover, since $K>0$, for almost every $x\in \Sph^2$ we have,
\be
\lim_{j\to\infty} \bar{f}_j^K=\lim_{j\to\infty} f_j=0,
\ee
and hence $f_\infty = 0$ a.e. on $\Sph^2$. This finishes the proof.
\end{proof}


\subsection{A $1$-sweepout of the warped product manifold $\Sph^2 \times_f \Sph^1$}\label{subsect-sweepout}
Because we will apply the Min-Max minimal surface theory to get an upper bound for $\mina$ in \S $\ref{subsect-bound-mina}$, in this subsection we briefly recall some basic notions in geometric measure theory following Marques and Neves \cite{MN-LNM-2018}, and construct a $1$-sweepout for $\Sph^2 \times_f \Sph^1$, which will be used in the proof in Lemma $\ref{lem-mina-upper-bound}$.  For an excellent survey and more details about these materials we refer to  \cite{MN-LNM-2018} and references therein.

A {\em $k$-current} $T$ on $\R^J$ is a continuous linear functional on the space of compactly supported smooth $k$-forms: $\mathcal{D}^k(\R^J)$. Its boundary $\partial T$ is a $(k-1)$-current that is defined as $\partial T (\phi) := T(d\phi)$ for $\phi \in \mathcal{D}^{k-1}(\R^J)$. A $k$-current $T$ is said to be an {\em integer multiplicity $k$-current} if it can be written as
\be\label{eqn-integer-current}
T(\phi) = \int_{S} \langle \phi(x), \tau(x) \rangle \theta(x) d\mathcal{H}^k, \quad \phi \in \mathcal{D}^k(\R^J),
\ee
where $S$ is a $\mathcal{H}^k$-measurable countable $k$-rectifiable set, that is $S \subset S_0 \cup_{j\in \N} S_j$ with $\mathcal{H}^k(S_0) = 0$ and $S_j$ is an embedded $k$-dimensional $C^1$-submanifold for all $j\in\N$, $\theta$ is a $\mathcal{H}^k$-integrable $\N$-valued function, and $\tau$ is a $k$-form such that $\tau(x)$ is a volume form for $T_xS$ at $x$ where a $k$-dimensional tangent space $T_xS$ is well-defined. Note that this tangent space $T_x S$ is well-defined for $\mathcal{H}^k$-a.e. $x \in S$, provided $\mathcal{H}^{k}(S\cap K)< +\infty$ for every compact set $K \subset \R^J$. Also note that the form $\tau$ give an orientation for $T_xS$. The {\em mass} of an integer multiplicity $k$-current $T$ is defined as
\be
{\bf M}(T) := \sup \{ T(\phi) \mid \phi \in \mathcal{D}^k(\R^J), \ \ |\phi| \leq 1 \},
\ee
where $|\phi|$ is the pointwise maximal norm of a form $\phi$.

In particular, a $k$-dimensional embedded smooth submanifold of $\R^J$ can be viewed as an integer multiplicity $k$-current by integrating a $k$-form over it. Its current boundary is given by its usual boundary, and its mass is the $k$-dimensional volume of the submanifold.

Let $M$ be a manifold embedded in $\R^J$. The space of {\em integral $k$-currents} on $M$, denoted by $\I_{k}(M)$, is defined to be the space of $k$-current such that both $T$ and $\partial T$ are integer multiplicity currents with finite mass and support contained in $M$. The space of {\em $k$-cycles}, denoted by $\mathcal{Z}_k(M)$, is defined to be the space of those $T \in \I_k(M)$ so that $T = \partial Q$ for some $Q \in \I_{k+1}(M)$.

A {\em rectifiable $k$-varifold} V is defined to be a certain Radon measure on $\R^J \times G_k(\R^J)$, where $G_k(\R^J)$ is the Grassmannian of $k$-planes in $\R^J$. An integral $k$-current $T \in \I_k(M)$ given as in $(\ref{eqn-integer-current})$ naturally associates a rectifiable $k$-varifold, denoted by $|T|$, as
\be\label{eqn-varifold}
|T|(A) = \int_{S \cap \pi(TS \cap A)} \theta(x) d\mathcal{H}^k.
\ee
Here $\pi$ is the natural projection map from $\R^J \times G_{k}(\R^J)$ to $\R^J$, and $TS$ is rank-$k$ tangent bundle of $S$ consisting of $T_x S$ at $x\in S$ where its $k$-dimensional tangent plane can be well defined. Note that: in the varifold expression $(\ref{eqn-varifold})$ of $|T|$, we forget the orientation of $S$ determined by the $k$-form $\tau$ in the current expression $(\ref{eqn-integer-current})$ of $T$.

The space $\I_k(M)$ can be endowed with various metrics and have different induced topologies. Given $T, S \in \I_k(M)$, the {\em flat metric} is defined by
\begin{equation*}
\mathcal{F}(T, S) := \inf \left\{ {\bf M}(Q) + {\bf M}(R) \mid T - S = Q + \partial R, \ \ Q \in \I_k(M), \ \ R \in \I_{k+1}(M) \right\}
\end{equation*}
and induces the {\em flat topology} on $\I_k(M)$. We also denote $\mathcal{F}(T) := \mathcal{F}(T, 0)$ and have
\be
\mathcal{F}(T) \leq {\bf M}(T), \quad \forall T \in \I_k(M).
\ee

For $T, S \in \I_k(M)$, the  {\bf F}-{\em metric} is defined by Pitts in \cite{Pitts-book} as:
\be\label{eqn-F-metric}
{\bf F}(S, T) := \mathcal{F}(S - T) + {\bf F}(|S|, |T|),
\ee
where ${\bf F}(|S|, |T|)$ is the {\bf F}-metric on the associated varifolds defined on page 66 in \cite{Pitts-book} as:
\begin{equation*}
{\bf F}(|S|, |T|) := \sup \left\{ |S|(f) - |T|(f) \mid  f \in C_{c}(G_k(\R^J)), \ \ |f|\leq 1, \ \ {\rm Lip}(f) \leq 1 \right\}.
\end{equation*}
Recall that (see page 66 in \cite{Pitts-book})
\be
{\bf F}(|S|, |T|) \leq {\bf M}(S - T),
\ee
and hence
\be\label{eqn-F-metric-less-than-mass}
{\bf F}(S, T) \leq 2 {\bf M}(S - T), \quad \forall S, T \in \I_k(M).
\ee

For the Min-Max theory for minimal surfaces, the space of mod $2$ integral $k$-currents and mod $2$ $k$-cycles are also needed. They are denoted by $\I_k(M; \Z_2)$ and $\mathcal{Z}_k(M; \Z_2)$, respectively, and defined by an equivalence relation: $T \equiv S$ if $T - S = 2 Q$ for $T, S, Q \in \I_k(M)$. The notions of boundary, mass and metrics defined above for $\I_k(M)$ can be extended to $\I_{k}(M; \Z_2)$. For a $n$-dimensional manifold $M$, the Constancy Theorem (Theorem 26.27 in \cite{Simon-book}) says that if $T \in \I_n(M; \Z_2)$ has $\partial T = 0$, then either $T = M$ or $T = 0$.

Then we recall some basic facts about the topology of $\mathcal{Z}_k(M; \mathcal{F}; \Z_2)$, that is  $\mathcal{Z}_k(M; \Z_2)$ endowed with flat metric. Their proofs can be found in \cite{MN-LNM-2018}, also see \cite{Almgren-62}.  Let $n$ be the dimension of the manifold $M$. Then $\I_{n}(M; \mathcal{F}; \Z_2)$ is contractible and the continuous map
\be
\partial : \I_{n}(M; \mathcal{F}; \Z_2) \rightarrow \mathcal{Z}_{n-1}(M; \mathcal{F}; \Z_2)
\ee
is a $2$-fold covering map. The homotopy groups are:
\be
\pi_{k}\left( \mathcal{Z}_{n-1}(M; \mathcal{F}; \Z_2), 0 \right) =
\begin{cases}
0, & \text{when} \ \ k \geq 2, \\
\Z_2, & \text{when} \ \ k=1.
\end{cases}
\ee
For the calculation of the fundamental group, one notes that the map
\begin{eqnarray}
P : \pi_{1}\left( \mathcal{Z}_{n-1}(M; \mathcal{F}; \Z_2), 0 \right) & \rightarrow &  \{ 0, M \} \\
            \left[ \gamma \right] & \mapsto & \tilde{\gamma}(1)
\end{eqnarray}
is an isomorphism. Here $\gamma$ is a loop in $\mathcal{Z}_{n-1}(M; \mathcal{F}; \Z_2)$ with $\gamma(0) = \gamma(1) = 0$, and $\tilde{\gamma}$ is the unique lift to $\I_{n}(M; \mathcal{F}; \Z_2)$ with $\tilde{\gamma}(0) = 0$. Then by applying Hurewicz Theorem, one can obtain:
\be
H^1\left( \mathcal{Z}_{n-1}(M; \mathcal{F}; \Z_2); \Z_2 \right) = \Z_2 = \{ 0, \bar{\lambda} \}.
\ee
The the action of the fundamental cohomology class $\bar{\lambda}$ on a homology class induced by a loop is nonzero if and only if the loop is homotopically non-trivial.

We take the following definition of $1$-sweepout from \cite{MN-LNM-2018}.
\begin{defn}\label{defn-sweepout}
A continuous map $\Phi: \Sph^1 \rightarrow \mathcal{Z}_{n-1}(M; {\bf F}; \Z_2)$ is called a $1$-sweepout if $\Phi^*(\bar{\lambda}) \neq 0 \in H^1(\Sph^1, \Z_2)$.
\end{defn}
Here $\mathcal{Z}_{n-1}(M; {\bf F}; \Z_2)$ is the space $\mathcal{Z}_{n-1}(M; \Z_2)$ endowed with the {\bf F}-metric given in $(\ref{eqn-F-metric})$.

Now we return back our warped product manifold $\Sph^2 \times_f \Sph^1$, that is $\Sph^2 \times \Sph^1$ with Riemannian metric
\be
g = g_{\Sph^2} + f^2 g_{\Sph^1}.
\ee
For each fixed $x \in \Sph^2$, we construct a $1$-sweepout of $\Sph^2 \times_f \Sph^1$ consisting of tori
$\{\Sigma_{x, r} := \partial B_{r}(x) \times \Sph^{1} \mid 0\leq r \leq \pi\}$ , where $B_r(x)$ denotes the geodesic ball on $\Sph^2$ centered at $x$ with radius $r$. In other words, we consider the map
\begin{equation}\label{eqn-tori-sweepout}
\begin{aligned}
\Phi: [0, \pi] & \rightarrow { \mathcal{Z}_{2}(\mathbb{S}^{2} \times_{f} \mathbb{S}^{1}; {\bf F}; \Z_2}), \\
              r & \mapsto   \partial \left( B_r(x) \times \Sph^1 \right) = \partial B_r(x) \times \Sph^1.
\end{aligned}
\end{equation}
\begin{lem}\label{lem-sweepout}
The map $\Phi$ given in $(\ref{eqn-tori-sweepout})$ provides a $1$-sweepout of $\Sph^2 \times_f \Sph^1$ as in Definition $\ref{defn-sweepout}$.
\end{lem}
\begin{proof}
Clearly, $\Phi(0) = \Phi(\pi) = 0$, and hence $\Phi$ can be viewed as a map from $\Sph^1$ to $\mathcal{Z}_2(\Sph^2 \times_f \Sph^1; {\bf F}; \Z_2)$ by identifying the end points of the interval $[0, \pi]$.

Now we show the continuity of the map $\Phi$ on $[0, \pi]$. This is clear for $r \in (0, \pi)$, since $\partial B_r(x)$ varies smoothly for $r \in (0, \pi)$. Then the continuity at $t =0$ follows from the inequality in $(\ref{eqn-F-metric-less-than-mass})$ and the estimate:
\be
\mass(\Phi(r) - \Phi(0)) = \mass(\Phi(r)) = \mass \left(\partial B_r(x) \times \Sph^1 \right) = f\cdot 4 \pi^2 \sin r \rightarrow 0,
\ee
as $r \rightarrow 0$, since the warping function $f$ is smooth on $\Sph^2$. The continuity at $t = \pi$ follows similarly, since $\sin r \rightarrow 0$ as $r\rightarrow \pi$.

Because by the definition flat metric is less than or equal to {\bf F}-metric, $\Phi$ is also continuous if we endow the flat metric on $\mathcal{Z}_2(M; \Z_2)$. So $\Phi$ is a loop in $\mathcal{Z}_2(\Sph^2 \times_f \Sph^1; \mathcal{F} ; \Z_2)$, and represents a non-trivial element:
\be
\left[ \Phi \right] \neq 0 \in \pi_1\left( \mathcal{Z}_2(\Sph^2 \times_f \Sph^1; \mathcal{F}; \Z_2) \right).
\ee
This is because by the definition of the map $\Phi$ we have that the unique lift $\tilde{\Phi}$ of $\Phi$ with $\tilde{\Phi}(0) = 0$ is given by
\begin{equation}
\begin{aligned}
\tilde{\Phi}: [0, \pi] & \rightarrow { \mathcal{Z}_{3}(\mathbb{S}^{2} \times_{f} \mathbb{S}^{1}; \mathcal{F}; \Z_2}), \\
              r & \mapsto    B_r(x) \times \Sph^1,
\end{aligned}
\end{equation}
and has $\tilde{\Phi}(\pi) = \Sph^2 \times \Sph^1$. Consequently, $\Phi^*(\bar{\lambda}) \neq 0$, and so $\Phi$ is a $1$-sweepout.
\end{proof}

\subsection{Bound $\mina$ from above by $L^{1}$-norm of warping function}\label{subsect-bound-mina}

In this subsection, we derive an upper bound for $\mina(\Sph^2 \times_f \Sph^1)$ in terms of $\|f\|_{L^{1}(\Sph^2)}$, provided that $\|f\|_{L^{2}(\Sph^2)}$ is small relative to $\mina(\Sph^2 \times_f \Sph^1)$.
\begin{prop}\label{prop-L1-bound-mina-from-upper}
Let $ \Sph^2 \times_{f} \Sph^1 $ be a warped product Riemannian manifolds with metric tensor as in (\ref{eqn-circle-over-sphere}) that has nonnegative scalar curvature and $\mina (\Sph^2 \times_{f} \Sph^1) \geq  A > 0$. If $\|f\|_{L^{2}(\Sph^2)} < \frac{A}{2^{\frac{3}{2}} \pi^{\frac{5}{2}}}$, then we have $\|f\|_{L^{1}(\Sph^1)} \geq \frac{A}{100 \pi}$.
\end{prop}

Recall that $\mina(\Sph^2 \times_f \Sph^1)$ is the infimum of areas of closed embedded minimal surfaces in $\Sph^2 \times_f \Sph^1$. Proposition \ref{prop-L1-bound-mina-from-upper} is crucial in the proof of Theorem \ref{thm: f_infty positive} below. In order to prove Proposition \ref{prop-L1-bound-mina-from-upper}, we first prove the following two lemmas.

First of all, we use the Min-Max minimal surface theory of Marques and Neves to bound $\mina(\Sph^2 \times_f \Sph^1)$ from above by areas of some tori in $\Sph^2 \times_f \Sph^1$.

\begin{lem}\label{lem-mina-upper-bound}
Let $\Sph^2 \times_{f} \Sph^1$ be a warped product Riemannian manifold with metric tensor as in (\ref{eqn-circle-over-sphere}). For each $x \in \Sph^2$, there exists a torus $\Sigma_{x, r_{x}}=\partial B_{r_{x}}(x) \times \mathbb{S}^{1} \subset \mathbb{S}^{2} \times_{f} \mathbb{S}^{1}$, $0 < r_{x} < \pi$, whose area is not less than $\mina (\Sph^2 \times_{f} \Sph^1)$, i.e.
\be
{\rm Area}(\Sigma_{x, r_{x}}) \geq \mina(\Sph^2 \times_{f} \Sph^1),
\ee
where $B_{r_{x}}(x)$ is the geodesic ball in the standard $\Sph^2$ centered at $x$ with radius $r_{x}$.
\end{lem}
\begin{proof}
We will use Min-Max minimal surface theory of Marques and Neves to prove the lemma.

For each fixed point $x \in \Sph^2$, by Lemma $\ref{lem-sweepout}$, the map $\Phi$ in $(\ref{eqn-tori-sweepout})$ gives a $1$-sweepout of $\Sph^2 \times_f  \Sph^1$ as in Definition \ref{defn-sweepout}. For $r \in [0, \pi]$, the image $\Phi(r) = \partial B_r(x) \times \Sph^1 =: \Sigma_{x, r}$ are tori in $\Sph^2 \times_f \Sph^1$ with mass:
\be
\mass(\Phi(r)) = {\rm Area}(\Sigma_{x, r}) = 2\pi  \int_{\partial B_{r}(x)} f ds.
\ee
Clearly, $\mass(\Phi(r))$ is a continuous function of $r$ on $[0, \pi]$ with $\mass(\Phi(0)) = \mass(\Phi(\pi)) = 0$. Thus there exist $r_{x} \in (0, \pi)$ such that
\be
\mass(\Phi(r_x)) = \max \{ \mass(\Phi(r)) \mid 0 \leq r \leq \pi\}.
\ee

Let $\Pi$ be the homotopy class of the $1$-sweepout $\Phi$, which consists of all continuous maps $\Phi^\prime : [0, \pi] \rightarrow \mathcal{Z}_2(\Sph^2 \times_f \Sph^1; {\bf F}; \Z_2)$ with $\Phi^\prime(0) = \Phi^\prime(\pi)$ such that $\Phi$ and $\Phi^\prime$ are homotopic to each other in the flat topology. By Lemma 2.2.6 in \cite{MN-LNM-2018}, the width
\be\label{eqn-defn-width}
{\bf L} (\Pi) = \inf_{\Phi^{\prime}\in \Pi} \sup_{r\in [0, \pi]} \{{\bf M} (\Phi^{\prime}(r))\} >0,
\ee
since $\Phi$ is a $1$-sweepout and so $\Pi$ is a non-trivial homotopy class.
Then Min-Max Theorem of Marques-Neves (see Theorem 2.2.7 in \cite{MN-LNM-2018}) implies that there exists a smooth embedded minimal surface $\Sigma$ in $\mathbb{S}^{2} \times_{f} \mathbb{S}^{1}$ achieving the width, i.e. ${\rm Area}(\Sigma) = {\bf L}(\Pi) >0$.

Finally, by the definitions of the width in (\ref{eqn-defn-width}) and $\mina$, and by the choice of $\Sigma_{x, r_{x}}$, we have
\be
{\rm Area}(\Sigma_{x, r_{x}}) \geq {\bf L}(\Pi) = {\rm Area}(\Sigma) \geq \mina(\mathbb{S}^{2} \times \mathbb{S}^{1}).
\ee
Because $x$ is an arbitrary point on $\Sph^2$, this completes the proof.
\end{proof}

Next, we apply Lemma \ref{lem-mina-upper-bound} and the spherical mean inequality from Proposition \ref{prop: spherical mean inequality}  to prove the following lemma.
\begin{lem}\label{lem-spherical-mean-bound-mina}
Let $\Sph^2 \times_{f} \Sph^1$ be a warped product Riemannian manifold with metric tensors as in (\ref{eqn-circle-over-sphere}) that have non-negative scalar curvatures and $\mina (\Sph^2 \times_{f} \Sph^1) \geq A >0$. If $\|f\|_{L^{2}(\Sph^2)} < \frac{A}{2^{\frac{3}{2}} \pi^{\frac{5}{2}}}$, then there exists a set $\mathcal{H} \subset \Sph^2$ satisfying that for each $x \in \mathcal{H}$ there exists $0< r_x \leq \frac{\pi}{2}$ such that
\begin{enumerate}[(i)]
\item ${\rm Area}\left( \underset{x \in \mathcal{H}}{\cup} B_{\frac{r_{x}}{10}}(x) \right) \geq \frac{1}{2} {\rm Area}(\Sph^2)$,
\item and
              \be
               \fint_{\partial B_{r}(x)} f ds \geq \frac{A}{2(2\pi)^{2}}
              \ee
              holds for all $r \in [0, r_{x}]$.
\end{enumerate}
\end{lem}
\begin{proof}
For any point $x \in \Sph^2$, we denote its antipodal point by $\bar{x}$. By Lemma \ref{lem-mina-upper-bound}, for any $x \in \Sph^2$, there exists $0 < r_{x} < \pi$ such that the torus $\Sigma_{x, r_{x}} = \partial B_{r_{x}}(x) \times \Sph^1$ in $\Sph^2 \times_{f} \Sph^1$ has area
\be
{\rm Area} (\Sigma_{x, r_{x}}) \geq \mina (\Sph^2 \times_{f} \Sph^1) \geq A.
\ee
{
Since ${\rm Area} (\Sigma_{x, r_{x}}) = 2 \pi  \int_{\partial B_{r_{x}}(x)} f ds $, we have
\be
2 \pi  \int_{\partial B_{r_{x}}(x)} f ds  \geq A.
\ee
 Thus, we have
\be
\int_{\partial B_{r_{x}}(x)} f ds \geq \frac{A}{2\pi}.
\ee
}

Now if $0 < r_{x} \leq \frac{\pi}{2}$, then we include the point $x$ in the set $\mathcal{H} $, and if $r_{x} > \frac{\pi}{2}$, then we include its antipodal point $\bar{x}$ in the set $\mathcal{H} $, and we set $r_{\bar{x}} = \pi - r_{x} < \frac{\pi}{2}$. Then we still have
\be
\int_{\partial B_{r_{\bar{x}}}(\bar{x})} f ds = \int_{\partial B_{r_{x}}(x)} fds \geq \frac{A}{2\pi},
\ee
since $\partial B_{r_{\bar{x}}}(\bar{x}) = \partial B_{r_{x}}(x)$.

By the construction of the set $\mathcal{H} \subset \Sph^2$, $\mathcal{H}$ contains at least one of any pair of antipodal points on $\Sph^2$, and for any $x \in \mathcal{H}$, there exists $0 < r_{x} \leq \frac{\pi}{2}$ such that
\be
\int_{\partial B_{r_{x}}(x)} f ds \geq \frac{A}{2\pi}.
\ee

Then we have that the area of the open set $\underset{x \in \mathcal{H}}{\cup} B_{\frac{r_{x}}{10}}(x)$ is at least half of the area of the whole sphere $\Sph^2$, i.e.
\be
\area \left( \underset{x \in \mathcal{H}}{\cup} B_{\frac{r_{x}}{10}}(x) \right) \geq \frac{1}{2} \area(\Sph^2).
\ee
Indeed, otherwise, we have
\be
\area \left( \underset{x \in \mathcal{H}}{\cup} B_{\frac{r_{x}}{10}}(\bar{x}) \right)
= \area \left( \underset{x \in \mathcal{H}}{\cup} B_{\frac{r_{x}}{10}}(x) \right)
< \frac{1}{2} \area(\Sph^2).
\ee
On the other hand, because for each $x \in \Sph^2$ either $x$ or $\bar{x}$ is contained in $\mathcal{H}$, we have
\be
\Sph^2 = \left( \underset{x \in \mathcal{H}}{\cup} B_{\frac{r_{x}}{10}}(x) \right) \cup  \left( \underset{x \in \mathcal{H}}{\cup} B_{\frac{r_{x}}{10}}(\bar{x}) \right).
\ee
So
\begin{eqnarray}
\area(\Sph^2)
& = &  \area \left( \left( \underset{x \in \mathcal{H}}{\cup} B_{\frac{r_{x}}{10}}(x) \right) \cup  \left( \underset{x \in \mathcal{H}}{\cup} B_{\frac{r_{x}}{10}}(\bar{x}) \right)  \right)  \\
& \leq & \area \left( \underset{x \in \mathcal{H}}{\cup} B_{\frac{r_{x}}{10}}(x) \right) + \area \left( \underset{x \in \mathcal{H}}{\cup} B_{\frac{r_{x}}{10}}(\bar{x}) \right) \\
& < & \frac{1}{2} \area(\Sph^2) + \frac{1}{2} \area(\Sph^2) = \area(\Sph^2).
\end{eqnarray}
This gives a contradiction. So we have $\area \left( \underset{x \in \mathcal{H}}{\cup} B_{\frac{r_{x}}{10}}(x) \right) \geq \frac{1}{2} \area(\Sph^2)$.

Because $\Sph^2 \times_{f} \Sph^1$ has non-negative scalar curvature, by Lemma \ref{lem: nonnegative scalar curvature condition}, we have $\Delta f \leq f$. Then by the spherical mean inequality in Proposition \ref{prop: spherical mean inequality}, for any $x \in \mathcal{H} \subset \Sph^2$ and any $0 \leq r \leq r_{x} (\leq \frac{\pi}{2})$ we have that
\be
\fint_{\partial B_{r_{x}}(x)} f ds - \fint_{\partial B_{r}(x)} f ds \leq \frac{\|f\|_{L^{2}(\Sph^2)}}{\sqrt{2\pi}}(r_{x} - r) \leq \frac{A}{2(2\pi)^2},
\ee
since $\|f\|_{L^{2}(\Sph^2)} \leq \frac{A}{2^{\frac{3}{2}} \pi^{\frac{5}{2}}}$ and $r_{x} - r \leq \frac{\pi}{2}$. By rearrange the inequality, we obtain that for any $x \in \mathcal{H}$ and any $ 0 \leq r \leq r_{x}$,
\begin{eqnarray}
\fint_{\partial B_{r}(x)} f ds
& \geq &
 \fint_{\partial B_{r_{x}}(x)} f ds - \frac{A}{2(2\pi)^{2}} \\
 & = & \frac{1}{2\pi \sin r_{x}} \int_{\partial B_{r_{x}}(x)} f ds - \frac{A}{2(2\pi)^2} \\
 & \geq & \frac{1}{2\pi} \int_{\partial B_{r_{x}}(x)} f ds - \frac{A}{2(2\pi)^2}  \\
 & \geq & \frac{A}{(2\pi)^{2}} - \frac{A}{2(2\pi)^2} = \frac{A}{2(2\pi)^{2}}.
 \end{eqnarray}

\end{proof}

We now apply Lemma \ref{lem-spherical-mean-bound-mina} and Vitali covering theorem to prove Proposition \ref{prop-L1-bound-mina-from-upper}:
\begin{proof}[Proof of Proposition \ref{prop-L1-bound-mina-from-upper}]
By Lemma \ref{lem-spherical-mean-bound-mina},  there exists a set $\mathcal{H} \subset \Sph^2$ such that
\begin{equation}
{\rm Area}(\underset{x \in \mathcal{H}}{\cup} B_{\frac{r_{x}}{10}}(x)) \geq \frac{1}{2} {\rm Area}(\Sph^2),
\end{equation}
 and for any $x \in \mathcal{H}$, there exists $r_{x} \leq \frac{\pi}{2}$ such that
              \be
               \fint_{\partial B_{r}(x)} f \geq \frac{A}{2(2\pi)^{2}}
              \ee
holds for all $r \in [0, r_{x}]$.

By the Vitali covering theorem, there exists a countable sequence of points $\{x_{i} \mid i \in \N\} \subset \mathcal{H}$ such that the collection of balls $\{B_{\frac{r_{x_{i}}}{10}}(x_{i})\}$ are disjoint with each other, and that
\begin{equation}
\underset{x \in \mathcal{H}}{\cup} B_{\frac{r_{x}}{10}}(x) \subset \underset{i \in \N}{\cup} B_{\frac{r_{x_{i}}}{2}}(x_{i}).
\end{equation}

By Lemma \ref{lem-spherical-mean-bound-mina} we have
\be
\frac{A}{8\pi^2} \leq \fint_{\partial B_{r}(x_i)} f = \frac{1}{2\pi \sin r} \int_{\partial B_{r}(x_i)} f ds, \quad \forall r \in [0, r_{x_i}].
\ee
As a result, we have
\be
\frac{A}{4\pi} \sin r \leq \int_{\partial B_{r}(x_i)} f ds, \quad \forall r \in [0, r_{x_i}].
\ee
Integrating this inequality from $0$ to $\frac{r_{x_{i}}}{10}$ gives
\begin{eqnarray}
\frac{A}{8\pi^2} \area(B_{\frac{r_{x_i}}{10}})
& = & \frac{A}{8\pi^2} \int^{\frac{r_{x_i}}{10}}_{0} 2\pi \sin r dr \\
& \leq & \int^{\frac{r_{x_i}}{10}}_{0} \left( \int_{\partial B_{r}(x_{i})}  f ds \right) dr \\
& = & \int_{B_{\frac{r_{x_i}}{10}}(x_{i})} f {\rm vol}_{\Sph^2}.
\end{eqnarray}
Then by summing the above inequalities for $i \in \N$ together, we obtain
\be
\frac{A}{8\pi^2}  \sum^{+\infty}_{i=1} \area(B_{\frac{r_{x_i}}{10}})  \leq \sum^{+\infty}_{i=1} \int_{B_{\frac{r_{x_i}}{10}}(x_{i})} f {\rm vol}_{\Sph^2} \leq \|f\|_{L^{1}(\Sph^2)},
\ee
since $\{ B_{\frac{r_{x_i}}{10}}(x_{i}) \mid i\in \N \}$ are disjoint balls. In the standard $\Sph^2$ we have
\be
\area\left( B_{\frac{r_{x_i}}{10}}(x_{i}) \right) \geq \frac{1}{25} \area \left( B_{\frac{r_{x_i}}{2}}(x_i) \right).
\ee
As a result, we have
\begin{eqnarray}
\|f\|_{L^{1}(\Sph^2)}
& \geq & \frac{A}{8\pi^2}  \sum^{+\infty}_{i=1} \area \left( B_{\frac{r_{x_i}}{10}} \right) \\
& \geq & \frac{A}{200 \pi^2} \sum^{+\infty}_{i=1}  \area \left(B_{\frac{r_{x_i}}{2}}(x_i) \right) \\
& \geq &  \frac{A}{200 \pi^2} \area \left(\underset{i\in\N}{\cup} B_{\frac{r_{x_i}}{2}}(x_i)\right) \\
& \geq &  \frac{A}{200 \pi^2} \area \left( \underset{x \in \mathcal{H}}{\cup} B_{\frac{r_x}{10}}(x) \right) \\
& \geq & \frac{A}{200 \pi^2} \frac{1}{2} \area (\Sph^2) = \frac{A}{100 \pi}.
\end{eqnarray}
This completes the proof.
\end{proof}


\subsection{Positivity of the limit of warping functions}\label{subsect-positivity-limit}
In this subsection, we use Proposition \ref{prop-cutoff-infimum} and Proposition \ref{prop-L1-bound-mina-from-upper} to prove Theorem \ref{Intro-thm: f_infty positive}, we restate it here for the convenience of the reader
\begin{thm}\label{thm: f_infty positive}
Let $\{\Sph^2 \times_{f_j} \Sph^1\}_{j=1}^\infty$ be a sequence of warped product manifolds such that each $\Sph^2 \times_{f_j} \Sph^1$ has non-negative scalar curvature. If we assume that
\be
{\rm Vol} (\Sph^2 \times_{f_j} \Sph^1) \leq V \text{ and }\mina(\Sph^2 \times_{f_{j}} \Sph^1) \geq A >0, \forall j \in \N,
\ee
then we have the following:
\begin{enumerate}[$ (i)$]
   \item After passing to a subsequence if needed, the sequence of warping functions $\{f_j\}_{j=1}^\infty$ converges to some limit function $f_\infty$ in $L^{q}(\Sph^2)$ for all $q \in [1, \infty)$.
   \item The limit function $f_\infty$ is in $W^{1, p}(\Sph^2)$, for all $p$ such that $1 \leq p <2$.

   \item The essential infimum of $f_{\infty}$ is strictly positive, i.e. $\inf\limits_{\Sph^2} f_{\infty} >0$.
   \item If we allow $+\infty$ as a limit, then the limit
             \be
            \overline{f_\infty}(x) := \lim_{r \rightarrow 0} \fint_{B_{r}(x)} f_{\infty}
            \ee
         exists for every $x \in \Sph^2$. Moreover, $\overline{f_{\infty}}$ is lower semi-continuous and strictly positive everywhere on $\Sph^2$, and $\overline{f_\infty} = f_\infty $ a.e. on $\Sph^2$.
         \end{enumerate}
\end{thm}

\begin{proof}
$(i)$
By Lemma \ref{lem: nonnegative scalar curvature condition} and Lemma \ref{lem: volume upper bound condition}, the nonnegative scalar curvature condition and ${\rm Vol}(\Sph^2 \times_{f_j} \Sph^2) \leq V$ imply that the sequence of warping functions $\{f_j\}_{j=1}^\infty$ satisfies the hypothesis in Proposition \ref{prop-W1p-limit}. By applying Proposition \ref{prop-W1p-limit}, we get the desired convergence.

$(ii)$
By applying Proposition \ref{prop-W1p-limit} we get that $f_\infty\in W^{1, p}(\Sph^2)$, for all $p\in [1,2)$.

$(iii)$
We prove $\inf\limits_{\Sph^2} f_\infty >0$ by contradiction. Recall that $\inf\limits_{\Sph^2} f_\infty$ is the essential infimum of $ f_\infty$ as defined in Definition \ref{Defn-Ess-Inf}. First note that $f_{\infty}\geq 0$, since $f_{j}>0, \forall j \in \N$.
Assume that $\inf\limits_{\mathbb{S}^{2}}f_{\infty}=0$,
then by Proposition \ref{prop-cutoff-infimum} we have $f_\infty=0$ almost everywhere in $\Sph^2$ and hence
\be\label{eqn-f_j-converge-zero}
f_j \rightarrow 0 \ \  \text{in} \ \  L^{2}(\Sph^2), \ \ \text{as} \ \  j \rightarrow +\infty.
\ee

Therefore, for all sufficiently large $j$, we have $\|f_j\|_{L^{2}(\Sph^2)} < \frac{A}{2^{\frac{3}{2}} \pi^{\frac{5}{2}}}$. Then by Proposition \ref{prop-L1-bound-mina-from-upper}, we have $\|f_j\|_{L^{1}(\Sph^2)} \geq \frac{A}{100 \pi} >0$ for all sufficiently large $j \in \N$. This contradicts with that $f_j \rightarrow 0 $ in $L^{ 2}(\Sph^2)$ as $j\rightarrow +\infty$  in (\ref{eqn-f_j-converge-zero}). This finishes the proof of part $(ii)$.

$(iv)$
Because warping functions $f_i$ satisfy the requirements in Proposition \ref{prop: average limit exists}, the existence of the limit
  \be
  \overline{f_\infty}(x) := \lim_{r \rightarrow 0} \fint_{B_{r}(x)} f_{\infty},
  \ee
  the lower semi-continuity of $\overline{f_\infty}$ and $\overline{f_\infty} = f_\infty$ a.e. on $\Sph^2$ directly follow from Proposition \ref{prop: average limit exists}.

  Thus we only need to prove that $\overline{f_\infty}(x) > 0$ for all $x \in \Sph^2$. Let
  \be
  e_\infty : = \inf\limits_{\Sph^2} f_\infty >0.
  \ee
By the continuity of the distance funciton $d(y,x)$, there exists $ 0 < r_0 < \frac{\pi}{2}$ such that for all $x\in \Sph^2$ we have
  \be
  f_\infty (y) - C d(y, x) > \frac{e_\infty}{2}, \ \ \text{ for a.e.} \ \ y \in B_{r_0}(x).
  \ee
As a result, we have
  \be
  \fint_{B_{r_0}(x)} \left( f_\infty(y) - C d(y, x) \right) d\vol(y) > \frac{e_\infty}{2}, \ \ \forall x \in \Sph^2.
  \ee
  Then because in Proposition \ref{prop: average limit exists} we proved that for each fixed $x \in \Sph^2$ the ball average $\fint_{B_{r_0}(x)} \left( f_\infty(y) - C d(y, x) \right) d\vol(y)$ is non-increasing in $r \in \left( 0, \frac{\pi}{2} \right)$, and
  \be
  \lim_{r \rightarrow 0} \fint_{B_{r}(x)} f_\infty = \lim_{r \rightarrow 0}  \fint_{B_{r}(x)} \left( f_\infty(y) - C d(y, x) \right) d\vol(y),
  \ee
  we have that for each fixed $x \in \Sph^2$,
  \begin{eqnarray}
  \overline{f_\infty}(x)
  & := &  \lim_{r \rightarrow 0} \fint_{B_{r}(x)} f_{\infty} \\
  & = & \sup_{0< r < \frac{\pi}{2}}   \fint_{B_{r}(x)} \left( f_\infty(y) - C d(y, x) \right) d\vol(y) \\
  & \geq &    \fint_{B_{r_0}(x)} \left( f_\infty(y) - C d(y, x) \right) d\vol(y) \\
  & > & \frac{e_\infty}{2}>0.
  \end{eqnarray}
This completes the proof of theorem.
\end{proof}

\begin{rmrk}
{\rm
Theorem \ref{thm: f_infty positive} implies that the limit function $f_\infty$ has a everywhere positive lower semi-continuous representative $\overline{f_\infty}$ as a function in $W^{1, p}(\Sph^2)$ for $1 \leq p <2$. For the rest of paper, $f_\infty \in W^{1, p}(\Sph^2)$ will always denote this everywhere positive lower semi-continuous representative.
}
\end{rmrk}

We end this section with Proposition \ref{prop: f_j lower bound} below. The proof of Proposition \ref{prop: f_j lower bound} uses Theorem \ref{thm: f_infty positive} and the spherical mean inequality from Proposition \ref{prop: spherical mean inequality}. The positive uniform lower bound for warping functions $f_j$ obtained in Proposition \ref{prop: f_j lower bound} is important in proving geometric convergences of the sequence of warped product manifolds $\{\Sph^2 \times_{f_j} \Sph^1\}_{j=1}^\infty$ in our next paper.

\begin{prop}\label{prop: f_j lower bound}
Let $\{\Sph^2 \times_{f_j} \Sph^1\}_{j=1}^\infty$ be a sequence of warped product manifolds with metric tensors as in (\ref{eqn-circle-over-sphere}) that have non-negative scalar curvature and satisfy
\be {\rm Vol}(\Sph^2 \times_{f_j} \Sph^1) \leq V \text{ and }\mina(\Sph^2 \times_{f_j} \Sph^1) \geq A >0, \forall j \in \N.
\ee
Let $e_\infty := \inf\limits_{\Sph^2} f_\infty > 0$. Then there exists $j_{0} \in \N$ such that $f_j (x) \geq \frac{e_{\infty}}{4}> 0 $, for all $j \geq j_0$ and all $x \in \Sph^2$.
\end{prop}
\begin{proof}
By Lemma \ref{lem: nonnegative scalar curvature condition}, the non-negativity of scalar curvature of $\Sph^2 \times_{f_{i}} \Sph^1$ implies that
\be
\Delta f_j \leq f_j , \quad \forall j \in \N.
\ee
Therefore, by the spherical mean inequality in Proposition \ref{prop: spherical mean inequality}, we have
\be
f_j (x) \geq \fint_{\partial B_{s}(x)} f_j ds - \frac{\|f_j\|_{L^2(\Sph^2)}}{\sqrt{2\pi}}s, \quad \forall  s \in \left( 0, \frac{\pi}{2} \right), x \in \Sph^2, j\in \N.
\ee
Then multiplying the inequality by $\area(\partial B_{s}(x)) = 2\pi \sin(s)$ gives us
\be\label{eqn: f_j lower bound 1}
2\pi \sin(s) f_j (x) \geq \int_{\partial B_{s}(x)} f_j ds - \frac{\|f_j\|_{L^{2}(\Sph^2)}}{\sqrt{2\pi}} 2\pi \sin(s) s,
\ee
for all $ s \in \left( 0, \frac{\pi}{2} \right), x \in \Sph^2$ and $ j \in \N$. Let
\be
V(r):= \vol(B_{r}(x)) = \int^r_0 2\pi \sin s ds = 2\pi(1-\cos r),
\ee
and let $e_\infty := \inf_{\Sph^2} f_\infty$ denote the essential infimum of the limit function $f_\infty$ which is strictly positive by Theorem \ref{thm: f_infty positive}.

 Now integrating the inequality (\ref{eqn: f_j lower bound 1}) with respect to $s$ from $0$ to $r<\frac{\pi}{2}$ gives us
\begin{eqnarray}
 V(r) f_j (x)
& \geq & \int_{B_{r}(x)} f_j d{\rm vol}_{\Sph^2}  - \frac{\|f_j\|_{L^2 (\Sph^2)}}{\sqrt{2\pi}} \int^r_0 2\pi s \sin s ds  \\
& \geq & \int_{B_{r}(x)} f_{\infty} d{\rm vol}_{\Sph^2} - \|f_\infty - f_j \|_{L^1(\Sph^2)} \\
&         & - \sqrt{2\pi}\|f_j \|_{L^2(\Sph^2)} (\sin r - r\cos r)  \\
& \geq & e_\infty V(r)  - \|f_\infty - f_j \|_{L^1(\Sph^2)} \\
&         &  - \sqrt{2\pi}\|f_j \|_{L^2(\Sph^2)} (\sin r - r\cos r).
\end{eqnarray}
Then by dividing the inequality by $V(r)$ we obtain
\be\label{eqn: f_j lower bound 2}
f_j (x) \geq e_{\infty} - \frac{\|f_\infty - f_j \|_{L^1(\Sph^2)}}{V(r)}  - \frac{\|f_j \|_{L^{2}(\Sph^2)}}{\sqrt{2\pi}} \frac{\sin r - r \cos r}{1 - \cos r},
\ee
for all $0< r < \frac{\pi}{2}, x \in \Sph^2$ and $j \in \N$. By Lemma \ref{Lem-Moser-Trudinger} we have $\sup\limits_{j}\|f_j \|_{L^2(\Sph^2)}<\infty$, and by direct calculation we have that
\be
\lim_{r \rightarrow 0} \frac{\sin r - r \cos r}{1 - \cos r} =0,
\ee
we can choose $0< r_{1} < \frac{\pi}{2}$ such that
\be\label{eqn: f_j lower bound 3}
\left| \frac{\|f_j\|_{L^{2}(\Sph^2)}}{\sqrt{2\pi}} \frac{\sin r_1 - r_1 \cos r_1}{1 - \cos r_1} \right| < \frac{e_\infty}{2}, \quad \forall j \in \N.
\ee
Moreover, because $f_j \rightarrow f_\infty$ in $L^1(\Sph^2)$, we can choose $j_0 \in \N$ such that
\be\label{eqn: f_j lower bound 4}
\frac{\|f_\infty - f_j \|_{L^1(\Sph^2)}}{V(r_1)} \leq \frac{e_\infty}{4}, \quad \forall j \geq j_0.
\ee

Finally by combining (\ref{eqn: f_j lower bound 2}), (\ref{eqn: f_j lower bound 3}) and (\ref{eqn: f_j lower bound 4}) together, we conclude that $f_j (x) \geq \frac{e_\infty}{4} > 0$ for all $j \geq j_0$ and $x \in \Sph^2$.
\end{proof}

\subsection{Uniform systole positive lower bound} In this subsection, as an application of non-collapsing of warping functions $f_j$ obtained in Proposition \ref{prop: f_j lower bound}, we derive a uniform positive lower bound for the systole of the sequence of warped product manifolds $\Sph^2 \times_{f_i} \Sph^1$ satisfying assumptions in Proposition \ref{prop: f_j lower bound}.

\begin{defn}[Systole]\label{defn: systole}
The systole of a Riemannian manifold $(M, g)$, which is denoted by $sys(M, g)$ is defined to be the length of the shortest closed geodesic in $M$.
\end{defn}

\begin{rmrk}
{\rm
People may usually consider so-called $\pi_1$-systole that is  the length of a shortest {\em non-contractible} closed geodesic. But in the study of compactness problem of manifolds with nonnegative scalar curvature, we also need to take into account contractible closed geodesic, for example, in a dumbell, which is diffeomorphic to $\Sph^3$, we may have a short contractible closed geodesic.
}
\end{rmrk}

First of all we derive an interesting dichotomy property for closed geodesics in warped product manifolds: $N \times_f \Sph^1$, that is, the product manifold $N \times \Sph^1$ endowed with the metric $g = g_{N} + f^2 g_{\Sph^1}$, where $(N, g_N)$ is a $n$-dimensional (either compact or completep non-compact) Riemannian manifold without boundary, and $f$ is a positive smooth function on $N$.

\begin{lem}\label{lem: geodesic dichotomy}
There is a dichotomy for closed geodesics in $N \times_f \Sph^1$, that is, a closed geodesic in $N \times_f \Sph^1$ either wraps around the fiber $\Sph^1$, or is a geodesic in the base $N$.
\end{lem}
\begin{proof}
Let $\varphi \in [0, 2\pi]$ is a coordinate on the fiber $\Sph^1$. The warped product metric $g$ then can be written as
\begin{equation}
g = g_N + f^2 d\varphi^2.
\end{equation}

Let
\begin{equation}
\gamma(t) = (\gamma_N(t), \varphi(t)) \ \ t \in [0, 1]
\end{equation}
be a closed geodesic in $\Sph^2 \times_f \Sph^1$, and without loss of generality, we assume $\varphi(0) = 0$. We have two possible cases as following:

{\bf Case 1}: $\varphi([0, 1]) = [0, 2\pi]$. In this case, clearly, the geodesic wraps around the fiber $\Sph^1$.

{\bf Case 2}: $\varphi([0, 1]) \neq [0, 2\pi]$.  In this case, we show that $\varphi([0, 1]) = \{0\}$ by a proof by contradiction, and then clearly, $\gamma$ is a closed geodesic on base $N \cong  N \times \{\varphi = 0\}$. Otherwise, we have
\begin{equation}
0 < \varphi_0 := \max\{\varphi(t) \mid t \in [0, 1]\} < 2\pi.
\end{equation}
Moreover, there exists $0< t_0 < 1$ such that $\varphi(t_0) = \varphi_0$, since $\varphi(1) = \varphi(0) =0$ due to the closeness of the geodesic $\gamma$. Consequently, $t_0$ is a critical point of the function $\varphi(t)$, i.e. $\varphi^\prime(t_0) = 0$. As a result, the tangent vector of the geodesic at $t_0$, $\gamma^\prime(t_0) = (\gamma^\prime_N(t_0), 0)$, is tangent to $N \times \{\varphi = \varphi_0 \}$. On the other hand, there is a geodesic contained in $N \times \{ \varphi=\varphi_0\}$ that passes through the point $(\gamma_N(t_0), \varphi_0)$ and is tangent to $(\gamma^\prime_N(t_0), 0)$ at this point. Then by the uniqueness of the geodesic with given tangent vector at a point, and the fact that base $N$ is totally geodesic in the warped product manifold $N \times_f \Sph^1$, which can be seen easily by Koszul's formula, or see Proposition 9.104 in \cite{Besse}, we can obtain $\varphi([0, 1]) = \{\varphi_0\}$, and this contradicts with $\varphi(0) =0$.
\end{proof}

By the dichotomy of closed geodesics in Lemma \ref{lem: geodesic dichotomy}, we can obtain a lower bound estimate for the systole of $N \times_f \Sph^1$.

\begin{lem}\label{lem: systole lower bound}
The systole of the warped product Riemannian manifold $N \times_f \Sph^1$ is greater than or equal to $\min \left\{sys(N, g_N), 2\pi \min\limits_{\Sph^2}f \right\}$.
\end{lem}
\begin{proof}
Let $\gamma(t) = (r(t), \theta(t), \varphi(t)), t \in [0, 1],$ is a closed geodesic in $\Sph^2 \times_f \Sph^1$. By Lemma \ref{lem: geodesic dichotomy}, $\gamma$ either wraps around the fiber $\Sph^1$, or $\gamma$ is a closed geodesic in the base manifold $(N, g_N)$.

If $\gamma$ wraps around the fiber $\Sph^1$, then $\varphi([0, 1]) = [0, 2\pi]$, and so the length of $\gamma$:
\begin{eqnarray}
L(\gamma) = \int^{1}_{0}|\gamma^\prime(t)|_{g} dt
& \geq & \int^1_0 f(\gamma(t)) |\varphi^\prime(t)|dt \\
& \geq & \min\limits_{\Sph^2} f \int^1_0 |\varphi^\prime(t)| dt  \\
& \geq  & 2\pi \min\limits_{\Sph^2} f.
\end{eqnarray}

If $\gamma$ is a closed geodesic in the base $(N, g_N)$, then by the definition of systole, the length of $\gamma$ is greater than or equal to $sys(N, g_N)$.

These estimates of length of closed geodesics imply the lower bound of systole in the conclusion.
\end{proof}

By combining the lower bound estimate of systole in Lemma \ref{lem: systole lower bound} and Proposition \ref{prop: f_j lower bound}, we immediately have the following uniform lower bound for systoles.
\begin{prop}\label{prop: systole uniform lower bound}
Let $\{\Sph^2 \times_{f_j} \Sph^1\}_{j=1}^\infty$ be a sequence of warped product manifolds with metric tensors as in (\ref{eqn-circle-over-sphere}) that have non-negative scalar curvature and satisfy
\be
{\rm Vol}(\Sph^2 \times_{f_j} \Sph^1) \leq V \text{ and }\mina(\Sph^2 \times_{f_j} \Sph^1) \geq A >0, \forall j \in \N.
\ee
Let $e_\infty := \inf\limits_{\Sph^2} f_\infty > 0$. Then the systoles of $\Sph^2 \times_{f_j} \Sph^1$, for all $j\in \N$, have a uniform positive lower bound given by $\min\left\{2\pi, \frac{e_\infty}{2}\pi\right\}$.
\end{prop}
\begin{proof}
First note that the base manifold of the sequence of the warped product manifolds is the standard $2$-sphere, and its systole is equal to $2\pi$, since the image of a closed geodesic in $(\Sph^2, g_{\Sph^2})$ is always a great circle.

Then note that $e_\infty >0$ follows from the item $(iii)$ in Theorem \ref{thm: f_infty positive}. For each $j \in \N$, by Lemma \ref{lem: systole lower bound}, the systole of $\Sph^2 \times_{f_j} \Sph^1$ has a lower bound given by $\min \left\{ 2\pi, 2\pi \min\limits_{\Sph^2}f_j \right\}$. Then by Proposition \ref{prop: f_j lower bound}, $\min\limits_{\Sph^2} f_j \geq \frac{e_\infty}{4}$ holds for all $j \in \N$. Hence the conclusion follows and we complete the proof.
\end{proof}


\section{Nonnegative distributional scalar curvature of limit metric}\label{sect-distr-scal}
Now we use the positive limit function $f_\infty$ obtained in Theorem \ref{thm: f_infty positive} to define a weak warped product metrics:
\begin{defn}\label{defn-limit-metric}
Let $f_{\infty}$ be a function defined on $\Sph^2$ such that it is almost everywhere positive and finite on $\Sph^2$. We further assume that $f_{\infty} \in W^{1, p}(\Sph^2)$  for $1 \leq p<2$. Define
\be
g_{\infty} := g_{\Sph^2} + f_{\infty}^2 g_{\Sph^1},
\ee
to be a (weak) warped product Riemannian metric on $\Sph^2 \times \Sph^1$ in the sense of defining an inner product on the tangent space at (almost) every point of $\Sph^2 \times \Sph^1$.
\end{defn}
\begin{rmrk}
{\rm
In general, $g_\infty$ is only defined almost everywhere in $\Sph^2 \times \Sph^1$ with respect to the standard product volume measure $d\vol_{g_{\Sph^2}} d\vol_{g_{\Sph^1}}$, since $f_\infty$ may have value as $+\infty$ on a measure zero set in $\Sph^2$. Note that we allow $+\infty$ as ball average limit in Proposition \ref{prop: average limit exists}. For example, in the extreme example constructed by Christina Sormani and authors in \cite{STW-ex}, the limit warping function equal to $+\infty$ at two poles of $\Sph^2$.
}
\end{rmrk}

In Subsection \ref{subsect-W1p-limit-metric}, we show $W^{1, p}$ regularity of the weak metric tensor $g_\infty$ defined in Definition \ref{defn-limit-metric} for $1 \leq p <2$ [Proposition \ref{prop-regularity-limit-metric}], and prove that the warped product metrics $g_j = g_{\Sph^2} + f^2_j g_{\Sph^1}$ converge to $g_\infty$ in the $L^q$ sense for any $1 \leq q <+\infty$ [Theorem \ref{thm-Lq-convergence}].

In Subsection \ref{subsect-nonnegative-distr-scalar}, we show that the limit weak metric $g_\infty$ has nonnegative distributional scalar curvature in the sense of Lee-LeFloch [Theorem \ref{thm-distr-scalar}].

\subsection{$W^{1, p}$ limit Riemannian metric $g_\infty$}\label{subsect-W1p-limit-metric}
we prove the regularity of the metric tensor. Before that we need the following definition:

\begin{defn} \label{defn-W1p}
We define $L^{p}(\Sph^2\times \Sph^1, g_0)$ as the set of all tensors defined almost everywhere on $\Sph^2\times\Sph^1$ such that its $L^p$ norm measured in terms of $g_0$ is finite where $g_0$ is the
isometric product metric
\be
g_0 = g_{\Sph^2}+ g_{\Sph^1} \textrm{ on } \Sph^2 \times \Sph^1.
\ee
We define $W^{1,p}(\Sph^2\times \Sph^1, g_0)$ as the set of all tensors, $h$,  defined almost everywhere on $\Sph^2\times\Sph^1$ such that both the $L^p$ norm of $h$ and the $L^p$ norm of $\overline{\nabla} h$ measured in terms of $g_0$ are finite
where $\overline{\nabla}$ is the connection corresponding to the metric $g_0$.
\end{defn}

Now we prove the regularity of the metric tensor $g_\infty$ defined in Definition \ref{defn-limit-metric}:

\begin{prop} [Regularity of the metric tensor]\label{prop-regularity-limit-metric}
\label{metric-W1p}
The Riemannian metric tensor $g_\infty$ as in Definition \ref{defn-limit-metric} satisfies
\be
g_\infty\in W^{1,p}(\Sph^2\times \Sph^1, g_0)
\ee
for all $p\in [1,2)$ in the sense of Definition \ref{defn-W1p}.
\end{prop}

\begin{proof}
Using the background metric, $g_0$, we have

\begin{eqnarray}
\|g_\infty\|_{L^p(\Sph^2\times \Sph^1, g_0)}
& = & (2\pi)^{\frac{1}{p}} \| (2 + f^4_\infty)^{\frac{1}{2}} \|_{L^{p}(\Sph^2)} \\
& \leq  & (2\pi)^{\frac{1}{p}} \| \sqrt{2} + f^2_\infty \|_{L^{p}(\Sph^2)} \\
& \leq & (2\pi)^{\frac{1}{p}} \left( \sqrt{2} (4\pi)^{\frac{1}{p}} + \| f_\infty \|^2_{L^{2p}(\Sph^2)} \right)
\end{eqnarray}
is finite, since by the assumption, $f_\infty \in W^{1, p}(\Sph^2)$ for any $ p \in [1, 2)$, and Sobolev embedding theorem, we have $f_\infty \in L^{2p}(\Sph^2)$ for any $p \in [1, \infty)$.

Now for the gradient estimate, we fix an arbitrary $p \in [1, 2)$. We use $\overline{\nabla}$ to denote the connection of the background metric $g_0$. Clearly, we have
\be
\overline{\nabla} g_\infty =\overline{\nabla} g_{\Sph^2} +\overline{\nabla} f^2_\infty \otimes g_{\Sph^1} +f^2_\infty \overline{\nabla} g_{\Sph^1}.
\ee
and
\be
\overline{\nabla} g_{\Sph^2}=0,\text{ and }\overline{\nabla} g_{\Sph^1}=0.
\ee
Moreover, since $\overline{\nabla} f^2_\infty= 2f_\infty \nabla f_{\infty}$ we have
\be
\overline{\nabla} g_\infty=2f_\infty \nabla f_{\infty}\otimes g_{\Sph^1},
\ee
where $\nabla f_\infty$ is the gradient of $f_\infty$ on $(\Sph^2, g_{\Sph^2})$.
As  a result, we have
\begin{eqnarray}
\|\overline{\nabla} g_\infty\|^p_{L^p(\Sph^2\times \Sph^1, g_0)}
& = &  2\pi \int_{\Sph^2} 2^p f^p_\infty |\nabla f_\infty|^p d\vol_{g_{\Sph^2}} \\
& = & 2^{p+1} \pi \|f_\infty\|_{L^{p q^*}(\Sph^2, g_{\Sph^2})} \cdot \|\nabla f_\infty\|_{L^{pq}(\Sph^2)},
\end{eqnarray}
where $q>1$ is chosen so that $pq<2$, and $q^* = \frac{q}{q-1}$. Then again by Sobolev embedding theorem we have $f_\infty \in L^{q}$ for any $p \in [1, \infty)$, thus we obtain that $\|\on g_\infty\|_{L^{p}(\Sph^2 \times \Sph^1, g_0)}$ is finite for any $p \in [1, 2)$. This completes the proof.
\end{proof}

Then we apply Proposition \ref{prop-W1p-limit} to prove Theorem \ref{Intro-thm-Lq-convergence} which concerns the $L^q$ pre-compactness of warped product circles over sphere with non-negative scalar curvature. We restate Theorem \ref{Intro-thm-Lq-convergence}  as follows:

\begin{thm} \label{thm-Lq-convergence}
Let $\{ g_j = g_{\Sph^2}  + f^2_j g_{\Sph^1} \mid j \in \N \}$ be a sequence of warped Riemannian metrics on $\Sph^2 \times \Sph^1$ satisfying requirements in (\ref{eqn-Intro-thm-main-condition}).
Then there exists a subsequence $g_{j_k}$ and a (weak) warped Riemannian metric $g_\infty \in W^{1, p}(\Sph^2 \times \Sph^1, g_0)$ for $p \in [1, 2)$ as in Definition \ref{defn-limit-metric} such that
\be
 g_{j_k} \rightarrow g_\infty \ \ \text{ in} \ \  L^{q}(\Sph^2\times \Sph^1, g_0), \ \ \forall q \in [1, \infty).
\ee
\end{thm}

\begin{proof}
By Lemma \ref{lem: nonnegative scalar curvature condition} and Lemma \ref{lem: volume upper bound condition}, the assumptions in (\ref{eqn-Intro-thm-main-condition}) for $g_j$ implies that the warping functions $f_j$ satisfy the assumptions in Proposition \ref{prop-W1p-limit}. Thus, by applying Proposition \ref{prop-W1p-limit}, we have that there exists a subsequence $f_{j_k}$ of warping functions and $f_\infty \in W^{1, p}(\Sph^2)$ for all $1 \leq p <2$, such that
\be
f_{j_k} \rightarrow f_\infty, \ \ \text{in} \ \ L^{q}(\Sph^2), \ \ \forall q \in [1, \infty).
\ee
Let $g_\infty := g_{\Sph^2} + f^2_\infty g_{\Sph^1}$. Then by Proposition \ref{prop-regularity-limit-metric}, we have
\be
g_\infty \in W^{1, p}(\Sph^2 \times \Sph^1, g_0) \ \ \forall  1 \leq p <2.
\ee
Moreover, because
\be
g_j-g_\infty= (f_j^2-f_\infty^2)g_{\Sph^1},
\ee
 we have that for any $q \in [1, \infty)$,
\begin{eqnarray}
& & \|g_{j_k}-g_\infty\|_{L^q (\Sph^2\times \Sph^1, g_0)} \\
& = & ( 2\pi )^{\frac{1}{q}} \| f^2_{j_k} - f^2_\infty \|_{L^q (\Sph^2)} \\
& =  & (2\pi)^{\frac{1}{q}} \| (f_{j_k} - f_\infty) \cdot (f_{j_k} + f_\infty) \|_{L^q (\Sph^2)}\\
& \leq &  (2\pi)^{\frac{1}{q}} \| f_{j_k} - f_\infty \|_{L^{2q}(\Sph^2)} \cdot \|f_{j_k} + f_\infty\|_{L^{2q}(\Sph^2)} \\
& \rightarrow & 0, \ \ \text{as} \ \ j_k \rightarrow \infty,
\end{eqnarray}
since $f_{j_k} \rightarrow f_\infty$ in $L^{2q}(\Sph^2)$ for any $q \in [1, \infty)$.
\end{proof}

\begin{rmrk}
{\rm
As showed by the example constructed by Christina Sormani and authors in \cite{STW-ex}, $g_\infty \in W^{1, p}(\Sph^2 \times \Sph^1, g_0)$ for $1 \leq p <2$ is the best regularity we can expect in general for the limit weak Riemannian metric $g_\infty$, see Proposition 3.6 and Remark 3.8 in \cite{STW-ex}.
}
\end{rmrk}

\subsection{Nonnegative distributional scalar curvature of $g_\infty$}\label{subsect-nonnegative-distr-scalar}
Building upon work of Mardare-LeFloch \cite{LM07}, Dan Lee and Philippe LeFloch defined a notion of distributional scalar curvature for smooth manifolds that
have a metric tensor which is only $L^{\infty}_{loc} \cap W^{1, 2}_{loc}$.   See Definition 2.1 of \cite{Lee-LeFloch}
which we review below in Definition~\ref{defn-Lee-LeFloch}.

In Theorem \ref{thm-Lq-convergence}  we proved that if a sequence of smooth warped product circles over the sphere $\{\Sph^2 \times_{f_j} \Sph^1\}$ with non-negative scalar curvature have uniform bounded volumes, then a subsequence of the smooth warped product metric $g_j = g_{\Sph^2} + f^2_j g_{\Sph^1}$ converges to a weak warped product metric $g_\infty = g_{\Sph^2} + f^2_\infty g_{\Sph^1} \in W^{1, p}(\Sph^2\times \Sph^1, g_0) (1 \leq p <2)$ in the sense of $L^{q}(\Sph^2 \times \Sph^1, g_0)$ for any $q \geq 1$. For the rest of this section, we use $g_\infty$ to denote such limit metric. We use $g_0 = g_{\Sph^2} + g_{\Sph^1}$ as a background metric .

In Theorem~\ref{thm-distr-scalar}, we prove that this limit (weak) metric $g_\infty$
has nonnegative distributional scalar curvature in the sense of Lee-LeFloch .
In Remarks~\ref{rmrk-Lee-LeFloch-original}-\ref{rmrk-Lee-LeFloch}, we discuss how the metric tensors studied by Lee and LeFloch have stronger regularity than
the regularity of $g_\infty$ but their definition of distributional scalar curvature is
still valid in our case.

First we recall Definition 2.1 in the work of Lee-LeFloch \cite{Lee-LeFloch}.   In their paper, they assume that

\begin{defn}[Lee-LeFloch]\label{defn-Lee-LeFloch}
{\rm
Let $M$ be a smooth manifold endowed with a smooth background metric, $g_0$.  Let $g$ be a metric tensor defined on $M$ with $L^{\infty}_{loc} \cap W^{1, 2}_{loc}$ regularity and locally bounded inverse $g^{-1} \in L^{\infty}_{loc}$.

The {\em scalar curvature distribution} $\Scal_{g}$ is defined as a distributions in $M$ such that
for every test function $u \in C^{\infty}_{0}(M)$
\be\label{eqn-Lee-LeFloch}
\langle \Scal_g, u \rangle := \int_{M} \left( - V \cdot \overline{\nabla} \left(u \frac{d\mu_g}{d \mu_{g_0}}\right)  + F u \frac{d\mu_g}{\,d\mu_0}\right) \,d\mu_0,
\ee
where the dot product is taken using the metric $g_0$, $\overline{\nabla}$ is the Levi-Civita connection of $g_0$, $d\mu_g$ and $d\mu_{g_0}$ are volume measure with respect to $g$ and $g_0$ respectively, $V$ is a vector field given by
\be\label{defn-V}
 V^k:= g^{ij} \Gamma^k_{ij}-g^{ik}\Gamma^j_{ji},
\ee
where
\be\label{eqn-Lee-LeFloch-Christoffel-Sym}
\Gamma^{k}_{ij} := \frac{1}{2} g^{kl} \left( \overline{\nabla}_{i}g_{jl} + \overline{\nabla}_{j}g_{il} - \overline{\nabla}_{l}g_{ij} \right),
\ee
\be\label{defn-F}
F:=  \overline{R} - \overline{\nabla}_k g^{ij}\Gamma^{k}_{ij} + \overline{\nabla}_{k} g^{ik}\Gamma^{j}_{ji} + g^{ij}\left( \Gamma^{k}_{kl} \Gamma^{l}_{ij} - \Gamma^{k}_{jl}\Gamma^{l}_{ik} \right),
\ee
and
\be\label{eqn-overline-R}
\overline{R} : = g^{ij} \left( \partial_k \og^k_{ij} - \partial_i \og^k_{kj} + \og^l_{ij} \og^k_{kl} - \og^l_{kj} \og^k_{il} \right).
\ee
The Riemannian metric $g$ has {\em nonnegative distributional scalar curvature}, if $\langle \Scal_g, u \rangle \geq 0$ for every nonnegative test function $u$ in the integral in (\ref{eqn-Lee-LeFloch}).
}
\end{defn}

\begin{defn}[Distributional total scalar curvature]
For a weak metric $g$ having the regularity as in Definition \ref{defn-Lee-LeFloch}, we define the distributional total scalar curvature of $g$ to be $\langle \Scal_{g}, 1 \rangle$, which is obtained by setting the test function $u \equiv 1$ in the integration in (\ref{eqn-Lee-LeFloch}).
\end{defn}
Note that for a $C^2$-metric, the distributional total scalar curvature is exactly the usual total scalar curvature.

\begin{rmrk}\label{rmrk-Lee-LeFloch-original}
{\rm
By the regularity assumption for the Riemannian metric $g$ in the work of Lee-LeFloch \cite{Lee-LeFloch}, one has the regularity
$\Gamma^{k}_{ij} \in L^{2}_{loc}$, $V \in L^{2}_{loc}, F\in L^{1}_{loc}$,
and the density of volume measure $d\mu_g$ with respect to $\,d\mu_0$ is
\be
\tfrac{d\mu_{g}}{\,d\mu_0} \in L^{\infty}_{loc} \cap W^{1, 2}_{loc}.
\ee
Thus
\be\label{LL-term1}
FirstInt_g=\int_{M} \left( - V \cdot \overline{\nabla} \left(u \frac{d\mu_g}{d \mu_{g_0}}\right) \right) \,d\mu_0
\ee
and
\be\label{LL-term2}
SecondInt_g=\int_{M} \left( F u \frac{d\mu_g}{\,d\mu_0}\right) \,d\mu_0.
\ee
are both finite.
}
\end{rmrk}

\begin{rmrk}\label{rmrk-Lee-LeFloch}
{\rm
Our limit metric is less regular than the metrics studied by Lee-LeFloch in \cite{Lee-LeFloch}. Recall that in Proposition \ref{metric-W1p} we showed $g_\infty \in W^{1, p}(\Sph^2 \times \Sph^1, g_0)$ for $1\leq p<2$, and as shown by the extreme example constructed in \cite{STW-ex}, in general $g_\infty \notin W^{1, 2}_{loc}(\Sph^2 \times \Sph^1, g_0)$, see Proposition 3.6 in \cite{STW-ex}.

In Remark \ref{rmrk-LL-divergence} below we show that in genenral both integrals in (\ref{LL-term1}) and (\ref{LL-term2}) may be divergent.
However, in Theorem~\ref{thm-distr-scalar} below, we show that in our case the sum of (\ref{LL-term1}) and (\ref{LL-term2}) is still well-defined since the singularity cancels out when we add them up.
}
\end{rmrk}

We are ready to prove Theorem \ref{Intro-thm-distr-scalar}. We restate it as follows:

\begin{thm}\label{thm-distr-scalar}
The limit metric $g_\infty$ obtained in Theorem \ref{thm-Lq-convergence} has nonnegative distributional scalar curvature on $\Sph^2\times \Sph^1$  in the sense of Lee-LeFloch as in Definition \ref{defn-Lee-LeFloch}.
In particular, (\ref{eqn-Lee-LeFloch}) is finite and nonnegative for any nonnegative
test function,  $u\in C^{\infty}(\Sph^2 \times \Sph^1)$. Moreover, the total scalar curvatures of $g_j$ converge to the distributional total scalar curvature of $g_\infty$.
\end{thm}

The proof of Theorem~\ref{thm-distr-scalar} consists of straightforward but technical calculations. For the convenience of readers, we provide some details of the calculations in the following lemmas.

We use $g_0 = g_{\Sph^2} + g_{\Sph^1}$ as background metric, and use coordinate $\{r, \theta, \varphi\}$ on $\Sph^2 \times \Sph^1$, where $(r, \theta)$ is a polar coordinate on $\Sph^2$ and $\varphi$ is a coordinate on $\Sph^1$. The corresponding local frame of the tangent bundle is $\{\partial_r, \partial_\theta, \partial_\varphi\}$. In this coordinate system, both $g_0$ and $g_\infty$ are diagonal and given as
\be \label{eqn-diagonal-g_infty}
g_0 =\begin{pmatrix}
1 & 0 &0\\
0& \sin^2 r & 0\\
0&0& 1\\
\end{pmatrix}
\textrm{ and }
g_\infty=\begin{pmatrix}
1 & 0 &0\\
0& \sin^2 r & 0\\
0&0& f_{\infty}^2(r, \theta)\\
\end{pmatrix}.
\ee

First of all, by the formula of Christoffel symbols:
\be
\og^{i}_{jk} = \frac{1}{2} (g_0)^{il}\left( \frac{\partial (g_0)_{il}}{\partial x^k} + \frac{\partial (g_0)_{lk}}{\partial x^j} - \frac{\partial (g_0)_{jk}}{\partial x^l}\right),
\ee
one can easily obtain the following lemma:

\begin{lem}\label{lem-background-metric-Christoffel-sym}
The Christoffel symbols of the Levi-Civita connection $\on$ of the background metric $g_0 = g_{\Sph^2} + g_{\Sph^1}$, in the coordinate $\{r, \theta, \varphi\}$, all vanish except
\be
\og^{r}_{\theta \theta} = - \sin r \cos r,
\ee
and
\be
\og^{\theta}_{r \theta} = \og^{\theta}_{\theta r} = \frac{\cos r}{\sin r}.
\ee
\end{lem}

Then by Lemma \ref{lem-background-metric-Christoffel-sym}, the formula
\begin{equation}
\on_i (g_\infty)_{jl}  =  \partial_i \left((g_\infty)_{jl} \right) -  \og^{p}_{ij} (g_\infty)_{pl} - \og^{q}_{il} (g_\infty)_{jq},
\end{equation}
and the diagonal expression of $g_\infty$ in (\ref{eqn-diagonal-g_infty}), one can obtain
the following lemma:

\begin{lem}\label{lem-Lee-LeFloch-Christoffel-sym}
For the limit metric, $g_\infty$, with the background metric, $g_0$, the Christoffel symbols defined by Lee-LeFloch as in (\ref{eqn-Lee-LeFloch-Christoffel-Sym}), in the coordinate $\{r, \theta, \varphi\}$, all vanish except
\be
\Gamma^r_{\varphi\varphi}=-f_{\infty}\partial_r f_{\infty},  \quad \Gamma^{\theta}_{\varphi \varphi} = - \frac{1}{\sin^2 r} f_{\infty} \partial_{\theta} f_{\infty},
\ee
and
\be
\Gamma^{\varphi}_{r \varphi}=\Gamma^{\varphi}_{\varphi r}=\frac{\partial_r f_\infty}{f_{\infty}}, \quad \Gamma^{\varphi}_{\theta \varphi} = \Gamma^{\varphi}_{\varphi \theta} = \frac{\partial_{\theta} f_{\infty}}{f_{\infty}}.
\ee
\end{lem}

Note also that

\begin{lem}\label{lem-vol-ratio}
Note that the volume forms are:
\be
d\mu_0=\,dr\wedge \sin(r)\,d\theta \wedge \, d\varphi
\ee
and
\be
d\mu_\infty = dr\wedge \sin(r)\,d\theta \wedge f_\infty(r, \theta)\,d\varphi
\ee
which are both defined almost everywhere.
In particular,
\be
\frac{\,d\mu_\infty}{\,d\mu_0}= f_\infty(r, \theta)
\ee
is in $W^{1,p}(\Sph^2 \times \Sph^1, g_0)$ for $p<2$.
\end{lem}

\begin{proof}
The first claim holds away from $r=0$ and $r=\pi$ by the definition of volume form, and the second claim holds almost everywhere on $(\Sph^2 \times \Sph^1, g_0)$.
So $d\mu_\infty = f_\infty d\mu_0$ almost everywhere which gives
us the third claim.  The rest follows from Proposition~\ref{prop-W1p-limit}.
\end{proof}

Now we are ready to compute the vector field $V$ and the function $F$ defined by Lee-LeFloch as in (\ref{defn-V}) and $(\ref{defn-F})$.

 \begin{lem}\label{lem-V}
For the limit metric $g_\infty$ with the background metric $g_0$, the vector field $V$ defined in (\ref{defn-V}), in the local frame $\{\partial_r, \partial_\theta, \partial_\varphi\}$, is given by
\be
V=\left( -2\frac{\partial_r f_\infty}{ f_\infty}, -  \frac{2}{\sin^2 r} \frac{\partial_\theta f_\infty}{f_\infty}, 0\right).
\ee
Furthermore
\be
 - V \cdot \overline\nabla \left(u \frac{\,d\mu_\infty}{\,d\mu_0}\right)
 =  2 \frac{\partial_r f_\infty}{f_\infty}  \partial_r(u f_\infty) +  \frac{2}{\sin^2 r}  \frac{\partial_\theta f_\infty}{f_\infty} \partial_\theta (u f_\infty).
\ee
\end{lem}

\begin{proof}
By plugging the non-vanishing Christoffel symbols in Lemma \ref{lem-Lee-LeFloch-Christoffel-sym} into
\be
V^k:= g^{ij}_\infty \Gamma^k_{ij}-g^{ik}_\infty\Gamma^j_{ji},
\ee
we get
\begin{eqnarray}
V^r & = &  g^{\varphi \varphi}_\infty \Gamma^r_{\varphi\varphi}- g^{rr}_\infty \Gamma^{\varphi}_{\varphi r} \\
     & = & \frac{1}{(f_\infty)^2} (- f_\infty \partial_r f_\infty) - \frac{\partial_r f_\infty}{f_\infty} = -2 \frac{\partial_r f_\infty}{f_\infty}.
\end{eqnarray}
Also
\begin{eqnarray}
V^{\theta}
& = & g^{\varphi \varphi}_\infty \Gamma^{\theta}_{\varphi \varphi} - g^{\theta \theta}_\infty \Gamma^{\varphi}_{\varphi \theta} \\
& = & \frac{1}{f^2_\infty} \left( - \frac{1}{\sin^2 r} f_\infty \partial_\theta f_\infty \right) - \frac{1}{\sin^2 r} \frac{\partial_\theta f_\infty}{f_\infty}=  - \frac{2}{\sin^2 r} \frac{\partial_\theta f_\infty}{f_\infty}.
\end{eqnarray}
\begin{equation}
V^{\varphi} = g^{ij}_\infty \Gamma^{\varphi}_{ij} - g^{\varphi \varphi}_\infty \Gamma^{j}_{j \varphi} = 0.
\end{equation}
By Lemma~\ref{lem-vol-ratio}, we now see that,
\begin{eqnarray}
 \overline\nabla \left(u \frac{\,d\mu_\infty}{\,d\mu_0}\right)
&=&
\overline\nabla
\left(u f_\infty\right)\\
&=& \partial_r (u f_\infty) \frac{\partial}{\partial r} + \frac{1}{\sin^2 r} \partial_\theta (u f_\infty) \frac{\partial}{\partial \theta} + \partial_\varphi (u f_\infty) \frac{\partial}{\partial \varphi}
\end{eqnarray}
Thus
\be
 - V \cdot \overline\nabla \left(u \frac{\,d\mu_\infty}{\,d\mu_0}\right)
 =  2 \frac{\partial_r f_\infty}{f_\infty}  \partial_r(u f_\infty ) +  \frac{2}{\sin^2 r}  \frac{\partial_\theta f_\infty}{f_\infty} \partial_\theta (u f_\infty)
 \ee
\end{proof}

\begin{lem}\label{lem-F}
For the limit metric $g_\infty$ with the background metric $g_0$, the function $F$ defined in (\ref{defn-F}) is given by
\begin{equation}
F= 2- 2\left(\frac{\partial_r f_\infty}{f_\infty}\right)^2 -  \frac{2}{\sin^2 r}  \left( \frac{\partial_\theta f_\infty}{f_\infty} \right)^2 = 2 - 2 \frac{1}{(f_\infty)^2} |\nabla f_\infty|^2.
\end{equation}
Furthermore,
\be
\left( F u \frac{\,d\mu_\infty}{d \mu_0} \right)=  2 u f_\infty - 2\frac{u}{f_\infty} |\nabla f_\infty|^2.
\ee
Here $|\nabla f_\infty|$ is the norm of weak gradient of $f_\infty$ with respect to the standard metric $g_{\Sph^2}$.
\end{lem}

\begin{proof}
First note that from the expression of $\overline{R}$ in (\ref{eqn-overline-R}) and the Christofell symbols calculated in Lemma \ref{lem-background-metric-Christoffel-sym}, one can easily see that
\be
\overline{R} = R_{g_{\Sph^2}} = 2.
\ee
Also recall that
\be
\on_i g^{jl}_\infty = \partial_i (g^{jl}_\infty) + \og^{j}_{ip} g^{pl}_\infty + \og^{l}_{iq} g^{jq}_\infty.
\ee
Then by Lemmas \ref{lem-background-metric-Christoffel-sym} and \ref{lem-Lee-LeFloch-Christoffel-sym}, one has
\begin{eqnarray}
F & := & \overline{R} - (\overline{\nabla}_k g^{ij}) \Gamma^k_{ij}+(\overline{\nabla}_k  g^{ik})\Gamma^j_{ji}+g^{ij}(\Gamma^k_{kl}\Gamma^l_{ij}-\Gamma^k_{jl}\Gamma^l_{ik}) \\
& = & 2 - \on_{r}g^{\varphi \varphi} \Gamma^{r}_{\varphi \varphi} - \on_{\theta} g^{\varphi \varphi} \Gamma^{\theta}_{\varphi \varphi} - 2 \on_{\varphi}g^{r\varphi} \Gamma^{\varphi}_{r \varphi}  - 2 \on_{\varphi}g^{\theta \varphi} \Gamma^{\varphi}_{\theta \varphi}\\
& & + \on_k g^{rk}\Gamma^{\varphi}_{\varphi r} + \on_k g^{\theta k}\Gamma^{\varphi}_{\varphi \theta} \\
&  & + \cancel{ g^{\varphi \varphi}\Gamma^{\varphi}_{\varphi r} \Gamma^{r}_{\varphi \varphi} } + \bcancel{ g^{\varphi \varphi} \Gamma^{\varphi}_{\varphi \theta} \Gamma^{\theta}_{\varphi \varphi} } \\
& &  - g^{\varphi \varphi} \Gamma^{r}_{\varphi \varphi} \Gamma^{\varphi}_{r \varphi} - \bcancel{ g^{\varphi \varphi} \Gamma^{\theta}_{\varphi \varphi} \Gamma^{\varphi}_{\varphi \theta} } - g^{rr} \Gamma^{\varphi}_{r \varphi } \Gamma^{\varphi}_{r \varphi} - \cancel{ g^{\varphi \varphi} \Gamma^{\varphi}_{\varphi r} \Gamma^{r}_{\varphi \varphi} } \\
& &  - g^{\theta \theta} \Gamma^{\varphi}_{\theta \varphi} \Gamma^{\varphi}_{\theta \varphi} - g^{\varphi \varphi} \Gamma^{\varphi}_{\varphi \theta} \Gamma^{\theta}_{\varphi \varphi} \\
& = & 2 - \left( \partial_r (g^{\varphi \varphi}) + 2 \og^{\varphi}_{r \varphi} g^{\varphi \varphi} \right) \Gamma^{r}_{\varphi \varphi}   - \left( \partial_\theta (g^{\varphi \varphi}) + 2 \og^{\varphi}_{\theta \varphi} g^{\varphi \varphi} \right) \Gamma^{\theta}_{\varphi \varphi} \\
&  &  - 2 \left( \partial_{\varphi}(g^{r \varphi}) + \og^{r}_{\varphi \varphi} g^{\varphi \varphi} + \og^{\varphi}_{\varphi r} g^{rr} \right) \Gamma^{\varphi}_{r \varphi}\\
&  &  - 2 \left( \partial_{\varphi}(g^{\theta \varphi}) + \og^{\theta}_{\varphi \varphi} g^{\varphi \varphi} + \og^{\varphi}_{\varphi \theta} g^{\theta \theta} \right) \Gamma^{\varphi}_{\theta \varphi}\\
& & + \left( \partial_r(g^{rr}) + \og^{r}_{rr}g^{rr} + \og^{r}_{rr} g^{rr} \right) \Gamma^{\varphi}_{\varphi r}  \\
& & + \left( \partial_\theta (g^{r \theta}) + \og^{r}_{\theta \theta}g^{\theta \theta} + \og^{\theta}_{\theta r} g^{rr} \right) \Gamma^{\varphi}_{\varphi r} \\
&  & + \left( \partial_\varphi (g^{r \varphi}) + \og^{r}_{\varphi \varphi} g^{\varphi \varphi} + \og^{\varphi}_{\varphi r} g^{rr} \right) \Gamma^{\varphi}_{\varphi r} \\
& & + \left( \partial_r(g^{\theta r}) + \og^{\theta}_{rr}g^{rr} + \og^{r}_{r\theta} g^{\theta \theta} \right) \Gamma^{\varphi}_{\varphi \theta}  \\
& & + \left( \partial_\theta (g^{\theta \theta}) + \og^{\theta}_{\theta \theta}g^{\theta \theta} + \og^{\theta}_{\theta \theta} g^{\theta \theta} \right) \Gamma^{\varphi}_{\varphi \theta} \\
&  & + \left( \partial_\varphi (g^{\theta \varphi}) + \og^{\theta}_{\varphi \varphi} g^{\varphi \varphi} + \og^{\varphi}_{\varphi \theta} g^{\theta \theta} \right) \Gamma^{\varphi}_{\varphi \theta} \\
& &  - g^{\varphi \varphi} \Gamma^{r}_{\varphi \varphi} \Gamma^{\varphi}_{r \varphi}  - g^{\varphi \varphi} \Gamma^{\varphi}_{\varphi \theta} \Gamma^{\theta}_{\varphi \varphi} - g^{rr} \Gamma^{\varphi}_{ r \varphi } \Gamma^{\varphi}_{r \varphi} - g^{\theta \theta} \Gamma^{\varphi}_{\varphi \theta} \Gamma^{\varphi}_{\varphi \theta} \\
& = & 2 - (-2) \frac{\partial_r f_\infty}{(f_\infty)^3} \left( - f_\infty \partial_r f_\infty \right) - (-2) \frac{\partial_\theta f_\infty}{(f_\infty)^3} \left( - \frac{1}{\sin^2 r} f_\infty \partial_\theta f_\infty \right) \\
& & + \left( -\frac{\cos r}{\sin r} + \frac{\cos r}{\sin r} \right) \Gamma^{\varphi}_{\varphi r} - \frac{1}{(f_\infty)^2}(-f_\infty \partial_r f_\infty) \left( \frac{\partial_r f_\infty}{f_\infty} \right) \\
& & - \frac{1}{(f_\infty)^2}\left( -\frac{1}{\sin^2 r} f_\infty \partial_\theta f_\infty \right) \left( \frac{\partial_\theta f_\infty}{f_\infty} \right)\\
& &  - \left( \frac{\partial_r f_\infty}{f_\infty} \right)^2 - \frac{1}{\sin^2 r} \left( \frac{\partial_\theta f_\infty}{f_\infty} \right)^2 \\
& = & 2 - 2 \left( \frac{\partial_r f_\infty}{f_\infty} \right)^2 -  \frac{2}{\sin^2 r}  \left( \frac{\partial_\theta f_\infty}{f_\infty} \right)^2 \\
& = & 2 - 2 \frac{1}{(f_\infty)^2} |\nabla f_\infty|^2.
\end{eqnarray}
We immediately obtain our second claim by applying Lemma~\ref{lem-vol-ratio}.
\end{proof}

 \begin{lem}\label{lem-extreme-Lee-LeFloch-divergence}
For $g$ being our limit metric tensor $g_\infty$ and a smooth nonnegative
test function $u$, the integrals in (\ref{LL-term1}) and (\ref{LL-term2}) are given by
\begin{eqnarray}
\quad FirstInt_{g_\infty}
&=& \int_{\Sph^2 \times \Sph^1}  \left( - V \cdot \overline\nabla \left(u \frac{\,d\mu_\infty}{\,d\mu_0}\right) \right) \,d\mu_0 \\
& = & \int_{\Sph^2} \left( 2 \langle \nabla f_\infty, \nabla \bar{u} \rangle + 2\frac{\bar{u}}{f_\infty} |\nabla f_\infty|^2 \right) d\vol_{g_{\Sph^2}}, \label{eqn-extreme-metric-FirstInt}
\end{eqnarray}
and
\begin{eqnarray}
\quad SecondInt_{g_\infty}&=& \int_{\Sph^2 \times \Sph^1} \left( F u \frac{\,d\mu_\infty}{d \mu_0} \right) \,d\mu_0\\
& = & \int_{\Sph^2} \left( 2 \bar{u} f_\infty - 2\frac{\bar{u}}{f_\infty} |\nabla f_\infty|^2 \right) d\vol_{g_{\Sph^2}},\label{eqn-extrem-metric-SecondInt}
\end{eqnarray}
where
\be\label{eq-bar-u}
\bar{u}(r, \theta) =  \int^{2 \pi}_{0} u(r, \theta, \varphi) d\varphi,
\ee
$\nabla f_\infty$ and $\nabla \bar{u}$ are (weak) gradients of functions $f_\infty$ and $\bar{u}$ on standard sphere $(\Sph^2, g_{\Sph^2})$ respectively,  and $\langle \cdot , \cdot \rangle$ is the Riemannian metric on $(\Sph^2, g_{\Sph^2})$.
\end{lem}

\begin{proof}
By integrating the formulas in Lemma~\ref{lem-V} and Lemma~\ref{lem-F}, one can easily obtain the integrals in (\ref{eqn-extreme-metric-FirstInt}) and (\ref{eqn-extrem-metric-SecondInt}).
\end{proof}

\begin{rmrk}\label{rmrk-LL-divergence}
{\rm
As explained in Remark \ref{rmrk-best-regularity},  $f_\infty \in W^{1, p}$ for any $1 \leq p < 2$, which is obtained in in Proposition \ref{prop-W1p-limit}, is the best regularity for $f_\infty$ in general, and we cannot expect $f_\infty$ is in $W^{1, 2}_{loc}(\Sph^2)$. So the integral $\int_{\Sph^2} \frac{\bar{u}}{f_\infty} |\nabla f_\infty|^2 d\vol_{g_{\Sph^2}}$ appearing in both (\ref{eqn-extreme-metric-FirstInt}) and (\ref{eqn-extrem-metric-SecondInt}) may be divergent (c.f. Lemma 4.16 in \cite{STW-ex}). But if we sum the integrants in (\ref{eqn-extreme-metric-FirstInt}) and (\ref{eqn-extrem-metric-SecondInt}) firstly and then integrate,  then this possible divergent integrant terms cancel out and we obtain a finite integral as in the following lemma.
}
\end{rmrk}

\begin{lem}\label{lem-calculation-Lee-LeFloch}
For the limit metric $g_\infty = g_{\Sph^2} + f^2_\infty g_{\Sph^1}$, the scalar curvature distribution $\Scal_{g_\infty}$ defined in Definition \ref{defn-Lee-LeFloch} can be expressed, for every test function $u \in C^{\infty}(\Sph^2 \times \Sph^1)$, as the integral
\be\label{eqn-calculation-Lee-LeFloch}
\langle \Scal_{g_\infty}, u \rangle
 =  \int_{\Sph^2} \left( 2 \langle \nabla f_\infty, \nabla \bar{u} \rangle + 2 f_\infty \bar{u} \right) d\vol_{g_{\Sph^2}},
\ee
and this is finite for any test function $u \in C^{\infty}(\Sph^2 \times \Sph^1)$. Here $\bar{u}$ is defined as in (\ref{eq-bar-u}),
$\nabla f_\infty$ and $\nabla \bar{u}$ are (weak) gradients of functions $f_\infty$ and $\bar{u}$ on standard sphere $(\Sph^2, g_{\Sph^2})$ respectively,  and $\langle \cdot , \cdot \rangle$ is the Riemannian metric on $(\Sph^2, g_{\Sph^2})$.
\end{lem}

\begin{proof}
The expression in (\ref{eqn-calculation-Lee-LeFloch}) immediately follows from the expressions in (\ref{eqn-extreme-metric-FirstInt}) and (\ref{eqn-extrem-metric-SecondInt})  and Definition \ref{defn-Lee-LeFloch}. The finiteness of the integral in (\ref{eqn-calculation-Lee-LeFloch}) follows from that $f_\infty \in W^{1, p}(\Sph^2)$ for $1 \leq p <2$ as proved in Proposition \ref{prop-W1p-limit}.
\end{proof}

We now apply these lemmas to prove Theorem~\ref{thm-distr-scalar}:

\begin{proof}
By the expression (\ref{eqn-scalar}) of the scalar curvature of $\Sph^2 \times_{f_i} \Sph^1$,
we have that for any test function $u \in C^{\infty} (\Sph^2 \times \Sph^1)$,
\begin{eqnarray}
 \int_{\Sph^2 \times \Sph^1} \Scal_{g_{j}} u d\vol_{g_j}
& = & \int_{\Sph^2} \left( \int^{2\pi}_{0} \left( 2 f_j u - 2 \Delta f_j u \right) d\varphi  \right) d\vol_{g_{\Sph^2}} \\
& = & \int_{\Sph^2} \left( 2 f_j \bar{u} - 2 \Delta f_j \bar{u} \right) d\vol_{g_{\Sph^2}} \\
& = & \int_{\Sph^2} \left( 2 f_j \bar{u} + 2 \langle \nabla f_j, \nabla \bar{u} \rangle \right) d\vol_{g_{\Sph^2}},
\end{eqnarray}
where $\bar{u}(r, \theta) = \int^{2\pi}_{0} u(r, \theta, \varphi) d\varphi$.
Then, by using the nonnegative scalar curvature condition $\Scal_{g_{j}} \geq 0$, Proposition \ref{prop-W1p-limit} and Lemma \ref{lem-calculation-Lee-LeFloch}, possibly after passing to a subsequence, we obtain for any nonnegative test function $0 \leq u \in C^{\infty}(\Sph^2 \times \Sph^1)$,
\begin{eqnarray}
0 & \leq & \int_{\Sph^2 \times \Sph^1} \Scal_{g_{j}} u d\vol_{g_j} \label{eqn-total scalar curvature-1}  \\
& = & \int_{\Sph^2} \left( 2 f_j \bar{u} + 2 \langle \nabla f_j, \nabla \bar{u} \rangle \right) d\vol_{g_{\Sph^2}} \\
& \rightarrow & \int_{\Sph^2} \left( 2 f_\infty \bar{u} + 2 \langle \nabla f_\infty, \nabla \bar{u} \rangle \right) d\vol_{g_{\Sph^2}} \\
& = & \langle \Scal_{g_\infty}, u \rangle. \label{eqn-total scalar curvature-2}
\end{eqnarray}
Thus, $\langle \Scal_{g_\infty}, u \rangle \geq 0$ for all nonnegative test function $u \in C^{\infty}(\Sph^2 \times \Sph^1)$. By setting $u\equiv 1$ in equations (\ref{eqn-total scalar curvature-1})-(\ref{eqn-total scalar curvature-2}), we obtain the convergence of distributional total scalar curvature.
\end{proof}




\begin{appendix}

\section{ $W^{1 ,2}$ convergence in $\Sph^{1} \times_{h} \Sph^{2}$ case}\label{appendix}

In this appendix, we will derive $W^{1, 2}$ convergence in the case of warped product spheres over circle with nonnegative scalar curvature, and show that the limit metric has nonnegative distributional scalar curvature in the sense of Lee-LeFloch.  Specifically, we will prove the following two theorems.
\begin{thm}\label{thm-sphere-over-circle-convergence}
Let $\{ \Sph^1 \times_{h_j} \Sph^2 \}^\infty_{j=1}$ be a family of warped Riemannian manifolds with metric tensors as in (\ref{eqn-sphere-over-circle}) satisfying
\be  \Scal_j \geq 0, \quad \diam (\Sph^1 \times_{h_j} \Sph^2) \leq D, \ee
and
\be \mina(\Sph^1 \times_{h_j} \Sph^2) \geq A>0 \ee
for all $j \in \N$, where $\Scal_j$ is the scalar curvature of $\Sph^1 \times_{h_j} \Sph^2$. Then there is
a subsequence of warping functions $h_j$ that converges in $W^{1, 2}(\Sph^1)$ to a Lipschitz function $h_\infty \in W^{1, 2}(\Sph^1)$, which has Lipschitz constant 1 and satisfies
\be
\sqrt{\frac{A}{4\pi}}\le h_{\infty} \le \frac{D}{\pi} + 2\pi, \quad \text {on} \ \ \Sph^1.
\ee

Moreover, let $g_\infty := g_{\Sph^1} + h^2_\infty g_{\Sph^2}$, then $g_\infty$ is a Lipschitz continuous Riemannian metric tensor on $\Sph^1 \times \Sph^2$, and a subsequence of $\{ g_j = g_{\Sph^1} + h^2_j g_{\Sph^2} \}^\infty_{j=1}$ converges in $W^{1, 2}(\Sph^1 \times \Sph^2, g_0)$ to $g_\infty$.
\end{thm}

Here, as before, we still use $g_0 = g_{\Sph^1} + g_{\Sph^2}$ as a background metric. Then we can compute the scalar curvature distribution of Lee-LeFloch and have the following property.
\begin{thm}\label{thm-scalar-appendix}
The limit metric $g_\infty$ obtained in Theorem \ref{thm-sphere-over-circle-convergence} has nonnegative distributional scalar curvature in the sense of Lee-LeFloch as recalled in Definition \ref{defn-Lee-LeFloch}.
\end{thm}

The study of this case is similar as the case of rotationally symmetric metrics on sphere, which was studied by authors with Jiewon Park in \cite{Park-Tian-Wang-18}. But there are some difference between these two cases. For example, in the rotationally symmetric metrics on sphere, in general $\mina$ condition may not be able to prevent collapsing happening near two poles [Lemma 4.3 in \cite{Park-Tian-Wang-18}], however, in the case of $\Sph^1 \times_{h_j} \Sph^2$, $\mina$ condition can provide a positive uniform lower bound for $h_j$ [Lemma \ref{lem-lower-bound}] and hence prevent collapsing happening.

The key ingredient is a uniform gradient estimate obtained by using nonnegative scalar curvature condition [Lemma \ref{lem-scal}]. Moreover, for the minimal value of warping function $h_j$, we obtain a uniform upper bound from uniform upper bounded diameter condition [Lemma \ref{lem-diam}] and a uniform lower bound from $\mina$ condition [Lemma \ref{lem-lower-bound}]. Then we combine these estimates to prove Theorem \ref{thm-sphere-over-circle-convergence} at the end of Subsection \ref{subsect-convergence-appendix}. Finally, in Subsection \ref{subsect-scalar-appendix}, we will prove Theorem \ref{thm-scalar-appendix}.


\subsection{Convergence of a subsequence}\label{subsect-convergence-appendix}

\begin{lem}\label{lem-diam}
Let $\{ \Sph^1 \times_{h_j} \Sph^2 \}^\infty_{j=1} $ be a family of warped product Riemannian manifolds with metric tensors as in (\ref{eqn-sphere-over-circle}), having uniformly upper bounded diameters, i.e. $\diam(\Sph^1 \times_{h_j} \Sph^2) \leq D$, then we have $\min\limits_{\Sph^1}\{h_{j}\} \le \frac{D}{\pi}.$
\end{lem}

\begin{proof}
Let $s_{0}\in\mathbb{S}^{1}$ be the minimum point of the function $h_j$. Then clearly the distance between antipodal points on the sphere $\{s_{0}\}\times \mathbb{S}^{2} \subset M_{j}$ is $\pi\cdot\min\limits_{\Sph^1}\{h_{j}\}$. So we have $\pi\cdot\min\limits_{\Sph^1}\{h_{j}\} \le \diam(M_{j}) \le D$, and the claim follows.
\end{proof}

\begin{lem} \label{lem-scal}
Let $\{ \Sph^1 \times_{h_j} \Sph^2 \}^\infty_{j=1}$ be a family of warped product Riemannian manifolds with metric tensors as in (\ref{eqn-sphere-over-circle}). The scalar curvature of the warped product metric $g_j = g_{\Sph^1} + h^2_j g_{\Sph^2}$ is given by
\be\label{scalar-curvature-formula-appendix}
\Scal_{j}=-4\frac{\Delta h_{j}}{h_{j}} + 2\frac{1-|\nabla h_{j}|^{2}}{h^{2}_{j}}.
\ee
Here the Laplace is the trace of the Hessian.

Moreover, if $\Scal_j \geq 0$, then we have $|\nabla h_j|\leq 1$ on $\mathbb{S}^{1}$.	
\end{lem}

\begin{proof}
First, by using the formula of Ricci curvature for warped product metrics as in 9.106 in \cite{Besse}, one can easily obtain that the scalar curvature $\Scal_{j}$ of $\Sph^1 \times_{h_j} \Sph^2$ is given as in (\ref{scalar-curvature-formula-appendix}).

Now we prove the second claim by contradiction. Assume for some $j$, $|\nabla h_{j}|>1$ at some point, let's say $p\in\mathbb{S}^{1}$. Take a unit vector field $X$ on $\mathbb{S}^{1}$ such that $X$ is in the same direction as $\nabla h_{j}$ at the point $p$. Let $q$ be the first point such that $|\nabla h|(q)=1$ while moving from the point $p$ on $\mathbb{S}^{1}$ in the opposite direction of the unit vector field $X$. Then let $\gamma$ be the integral curve of the vector field $X$ with the initial point $\gamma(0)=q$. Let $t_{1}>0$ such that $\gamma(t_{1})=p$. Set $\tilde{h}_{j}(t)=h_{j}\circ\gamma(t)$. Then (at least) for $t\in[0, t_{1}]$,
\be
\tilde{h}^{\prime}_{j}(t)=\langle \nabla h_{j}, \gamma^{\prime}(t) \rangle = \langle \nabla h_{j}, X \rangle \circ\gamma(t) = |\nabla h_{j}|\circ\gamma(t),
\ee
and
\be
\tilde{h}^{\prime \prime}_{j}(t) =(\Delta h_{j})\circ\gamma(t).
\ee

By the Mean Value Theorem, there exists $t_{2}\in (0, t_{1})$ such that
\be
\tilde{h}^{\prime\prime}_{j}(t_{2})=\frac{\tilde{h}^{\prime}_{j}(t_{1})-\tilde{h}^{\prime}_{j}(0)}{t_{1}} >0,
\ee
since $\tilde{h}^{\prime}_{j}(t_{1}) = |\nabla h_{j}|(p) >1$ and $\tilde{h}^{\prime}_{j}(0) = |\nabla h_{j}|(q) = 1$.

On the other hand, because $\Scal_{j}\ge0$, by using the scalar curvature $(\ref{scalar-curvature-formula-appendix})$, one has
\be
-4\frac{\tilde{h}^{\prime\prime}_{j}(t_{2})}{\tilde{h}_{j}(t_{2})}+2\frac{1-(\tilde{h}_{j}(t_{2}))^{2}}{(\tilde{h}_{j}(t_{2}))^{2}}\ge0
\ee
So
\be
\tilde{h}^{\prime\prime}_{j}(t_{2}) \le \frac{1-(\tilde{h}^{\prime}(t_{2}))^{2}}{2\tilde{h}(t_{2})}<0,
\ee
since $\tilde{h}^{\prime}_{j}(t_{2})>1$ by the choice of $q=\gamma(0)$. This produces a contradiction, and so $|\nabla h_j|\leq 1$ on $\mathbb{S}^{1}$.

\end{proof}

\begin{lem} \label{lem-min-surf}
Let $\{ \Sph^1 \times_{h_j} \Sph^2 \}^\infty_{j=1}$ be a family of warped product Riemannian manifolds with metric tensors as in (\ref{eqn-sphere-over-circle}). If $\nabla h_j(x_0)=0$ for some $x_0\in\mathbb{S}^1$ then there is a minimal surface $\{x_0\}\times \mathbb{S}^{2}$ in $\Sph^1 \times_{h_j} \Sph^2$.
\end{lem}

\begin{proof}
Define $\Sigma_x:=\{x\}\times\mathbb{S}^{2}$. Then for all $x\in\mathbb{S}^{1}$, $\Sigma_x$ is an embedded submanifold with mean curvature
\be
H_j=\frac{2|\nabla h_{j}|(x)}{h_j(x)}.
\ee
\end{proof}

\begin{lem}\label{lem-lower-bound}
Let $\{ \Sph^1 \times_{h_j} \Sph^2 \}^\infty_{j=1}$ be a family of warped product Riemannian manifolds with metric tensors as in (\ref{eqn-sphere-over-circle}) satisfying $\mina(\Sph^1 \times_{h_j} \Sph^2) \ge A>0$. Then we have $\min\limits_{\Sph^1}\{h_{j}\}\ge\sqrt{\frac{A}{4\pi}}>0$.
\end{lem}
\begin{proof}
By applying Lemma \ref{lem-min-surf}, we have that there exists a minimal surface $\Sigma_{x_{0}} ={x_0} \times \Sph^2 $ on $\Sph^1 \times_{h_j} \Sph^2$  at the minimal value point $x_{0}$ of $h_{j}$.  The area of $\Sigma_{x_0}$ is given by
\be
{\rm Area}(\Sigma_0) = 4\pi h^2_j(x_0).
\ee
Thus by the $\mina$ condition, $4 \pi h^2_j(x_0) \geq A$, and the conclusion follows.
\end{proof}

Now we will use above lemmas to prove Theorem \ref{thm-sphere-over-circle-convergence}:
\begin{proof}
We complete the proof in the following three steps.

{\bf Step 1. Uniform convergence of warping functions.}
By applying Lemma $\ref{lem-diam}$ and Lemma $\ref{lem-scal}$ we immediately obtain the uniform upper bound
\be
\max_{\Sph^1}\{h_{j}\}\le\frac{D}{\pi}+2\pi, \quad \forall i \in N.
\ee
By combining this uniform upper bound with the uniform lower bound obtained in Lemma $\ref{lem-lower-bound}$, we have that the warping functions $h_j$ are uniformly bounded, i.e.
\be
\sqrt{\frac{A}{4\pi}}\le h_{j} \le \frac{D}{\pi} + 2\pi \quad \text {on} \ \ \Sph^1 , \quad \forall j \in \N.
\ee
Moreover, Lemma \ref{lem-scal} implies function $h_j$ are equicontinuous. Thus by applying Arzel\`{a}-Ascoli theorem we obtain that $h_j$ are uniformly convergent a continuous function $f_\infty$ satisfying
\be
\sqrt{\frac{A}{4\pi}}\le h_{\infty} \le \frac{D}{\pi} + 2\pi, \quad \text {on} \ \ \Sph^1.
\ee
Meanwhile, the uniform gradient estimate obtained in Lemma \ref{lem-scal} also implies that the limit function $h_\infty$ is Lipschitz with Lipschitz constant $1$. Because a Lipschitz function is $W^{1, \infty}$, we actually have $h_\infty \in W^{1, \infty}(\Sph^1)$.

{\bf Step 2. $W^{1, 2}$ convergence of warping functions.}  We will estimate the bounded variation norm $\| \nabla h_j \|_{BV(\Sph^1)}$ of warping functions. First note that
\be
0 = \int_{\Sph^1} \Delta h_j = \int_{\{\Delta h_j \geq 0\}} \Delta h_j + \int_{\{\Delta h_j < 0\}} \Delta h_j.
\ee
Thus,
\be
-\int_{\{\Delta h_j < 0\}} \Delta h_j = \int_{\{\Delta h_j \geq 0\}} \Delta h_j,
\ee
furthermore,
\begin{eqnarray}
\| \nabla h_j \|_{BV(\Sph^1)}
& = & \int_{\Sph^1} |\nabla h_j| + \int_{\Sph^1} |\Delta h_j| \\
& = & \int_{\Sph^1} |\nabla h_j| + \int_{\{\Delta h_j \geq 0\}} \Delta h_j - \int_{\{\Delta h_j < 0\}} \Delta h_j \\
& = & \int_{\Sph^1} |\nabla h_j| + 2\int_{\{\Delta h_j \geq 0\}} \Delta h_j.
\end{eqnarray}

Then by the expression of the scalar curvature in Lemma \ref{lem-scal}, the nonnegative scalar curvature condition implies
\be
\Delta h_j \leq \frac{1 - |\nabla h_j|^2}{2 h_j} \leq \frac{1}{2 h_j} \leq \sqrt{\frac{\pi}{A}}, \quad \forall j \in \N.
\ee
The last inequality here follows from Lemma \ref{lem-lower-bound}. Lemma \ref{lem-scal} also tells us that $|\nabla h_j| \leq 1$ on $\Sph^1$ for all $j \in \N$. Consequently, we have
\begin{eqnarray}
\| \nabla h_j \|_{BV(\Sph^1)}
& = & \int_{\Sph^1} |\nabla h_j| + 2\int_{\{\Delta h_j \geq 0\}} \Delta h_j \\
& \leq & 2\pi + 2 \int_{\Delta h_j \geq 0} \sqrt{\frac{\pi}{A}} \\
& \leq & 2\pi \left( 1 + 2 \sqrt{\frac{\pi}{A}}\right), \quad \forall j \in \N.
\end{eqnarray}

As a result, by Theorem 5.5 in \cite{EG-text} we have that a subsequence, which is still denoted by $\nabla h_j$, converges to some $\phi \in BV(\Sph^1)$ in $L^1(\Sph^1)$ norm, and it is
easy to see that $\phi = \nabla h_\infty$ in the weak sense. Moreover, since $h_\infty \in W^{1, \infty}(\Sph^1)$ and $\sup\limits_{j} \|\nabla h_j \|_{L^\infty (\Sph^1)} < \infty$, we have $\nabla h_j \rightarrow \nabla h_\infty$ in $L^{2}(\Sph^1)$ norm. Indeed, note that by the H\"older inequality,
\be
\int_{\Sph^1} |\nabla h_j - \nabla h_\infty|^2 \leq \|\nabla h_j - \nabla h_\infty\|_{L^1(\Sph^1)} \|\nabla h_j - \nabla h_\infty\|_{L^\infty(\Sph^1)}.
\ee
As a result, $h_j \rightarrow h_\infty$ in $W^{1, 2}(\Sph^1)$.

{\bf Step 3. $W^{1, 2}$ convergence of metrics.} Note that
\be
 g_j - g_\infty  =  (h^2_j - h^2_\infty) g_{\Sph^2} ,
\ee
and
\be
\on(g_j - g_\infty) = 2 (h_j \on h_j - h_\infty \on h_\infty) \otimes g_{\Sph^2}.
\ee
Therefore, by applying the uniform bound $\sup\limits_j \|\nabla h_j\|_{L^\infty (\Sph^1)} < \infty$, and $W^{1, 2}$ convergence of $h_j$ to $h_\infty$, we can obtain that $g_j = g_{\Sph^1} + h^2_j g_{\Sph^2}$ converges to $g_\infty$ in $ W^{1, 2}(\Sph^1 \times \Sph^2, g_0) $.
\end{proof}


\subsection{Nonnegative distributional scalar curvature of the limit metric}\label{subsect-scalar-appendix}
In this subsection, we compute the distributional scalar curvature of the limit metric tensor $g_\infty$ obtained in Theorem \ref{thm-sphere-over-circle-convergence} with the background metric $g_0$ in the sense of Lee-LeFloch, and prove Theorem \ref{thm-scalar-appendix}. Throughout this subsection, $g_\infty$ always denotes the limit metric obtained in Theorem \ref{thm-sphere-over-circle-convergence}.

By the definition of $\Gamma^k_{ij}$ in Definition \ref{defn-Lee-LeFloch} and the Christofell symbols in Lemma \ref{lem-background-metric-Christoffel-sym},  one can obtain the following lemma:

\begin{lem}\label{lem-Lee-LeFloch-Christoffel-sym-appendix}
For the limit metric, $g_\infty$, with the background metric, $g_0$, the Christoffel symbols defined by Lee-LeFloch as in (\ref{eqn-Lee-LeFloch-Christoffel-Sym}), in the coordinate $\{\varphi, r, \theta\}$, all vanish except
\be
\Gamma^{\varphi}_{r r}= -h_\infty h_\infty^{\prime}, \quad \Gamma^{\varphi}_{\theta \theta}=- h_\infty h^{\prime}_\infty \sin^2 r,
\ee

\be
\Gamma^r_{\varphi r} = \Gamma^{r}_{r \varphi} = \frac{h^{\prime}_\infty}{ h_\infty },
\ee
and
\be
\Gamma^{\theta}_{\varphi \theta} = \Gamma^{\theta}_{\theta \varphi} = \frac{h^{\prime}_\infty}{h_\infty}.
\ee
\end{lem}

Note also that

\begin{lem}\label{lem-vol-ratio}
Note that the volume forms are:
\be
d\mu_0=\, d\varphi \, \wedge dr\wedge \sin(r)\,d\theta,
\ee
and
\be
d\mu_\infty = d\varphi \wedge h^2_\infty dr\wedge \sin(r)\,d\theta,
\ee
which are both defined everywhere away from $r=0$ and $r=\pi$.
In particular,
\be
\frac{\,d\mu_\infty}{\,d\mu_0}= h^2_\infty (\varphi)
\ee
is in $W^{1,p}(\Sph^2 \times \Sph^1, g_0)$ for all $ p \geq 1$.
\end{lem}

\begin{proof}
The first claim holds away from $r=0$ and $r=\pi$ by the definition of volume form, and the second claim holds almost everywhere on $(\Sph^2 \times \Sph^1, g_0)$.
So $d\mu_\infty = f_\infty d\mu_0$ almost everywhere which gives
us the third claim.  The rest follows from Proposition~\ref{prop-W1p-limit}.
\end{proof}

Now we are ready to compute the vector field $V$ and the function $F$ defined by Lee-LeFloch as in (\ref{defn-V}) and $(\ref{defn-F})$.

 \begin{lem}\label{lem-V-appendix}
For the limit metric $g_\infty$ with the background metric $g_0$, the vector field $V$ defined in (\ref{defn-V}), in the local frame $\{\partial_\varphi, \partial_r, \partial_\theta \}$, is given by
\be
V=\left( -4\frac{h^\prime_\infty}{h_\infty},0,0\right) = -4 \frac{h^\prime_\infty}{h_\infty} \frac{\partial}{\partial \varphi}.
\ee
Furthermore
\be
 - V \cdot \overline\nabla \left(u \frac{\,d\mu_\infty}{\,d\mu_0}\right)
 =  4 \frac{h^\prime_\infty}{h_\infty} \partial_\varphi (u h^2_\infty ).
\ee
\end{lem}

\begin{lem}\label{lem-F-appendix}
For the limit metric $g_\infty$ with the background metric $g_0$, the function $F$ defined in (\ref{defn-F}) is given by
\be
F= \frac{2}{h^2_\infty}- 6\left(\frac{h^\prime_\infty}{h_\infty}\right)^2.
\ee
Furthermore,
\be
\left( F u \frac{\,d\mu_\infty}{d \mu_0} \right)=  2 u  - 6 u (h^\prime_\infty)^2.
\ee
\end{lem}

 \begin{lem}\label{lem-extreme-Lee-LeFloch-divergence-appendix}
For $g$ being our limit metric tensor $g_\infty$ and a smooth nonnegative
test function $u$, the integrals in (\ref{LL-term1}) and (\ref{LL-term2}) are given by
\begin{eqnarray}
\quad FirstInt_{g_\infty}
&=& \int_{\Sph^2 \times \Sph^1}  \left( - V \cdot \overline\nabla \left(u \frac{\,d\mu_\infty}{\,d\mu_0}\right) \right) \,d\mu_0 \\
& = & \int_{\Sph^1} \left( 8 (h^\prime_\infty)^2 \bar{u} + 4 h^\prime_\infty h_\infty \bar{u} \right) d\varphi, \label{eqn-extreme-metric-FirstInt-appendix}
\end{eqnarray}
and
\begin{eqnarray}
\quad SecondInt_{g_\infty}&=& \int_{\Sph^2 \times \Sph^1} \left( F u \frac{\,d\mu_\infty}{d \mu_0} \right) \,d\mu_0\\
& = & \int_{\Sph^1} \left( 2 \bar{u} - 6  (h^\prime_\infty)^2 \bar{u} \right) d\varphi,\label{eqn-extrem-metric-SecondInt-appendix}
\end{eqnarray}
where
\be\label{eq-bar-u}
\bar{u}(\varphi) =  \int^{\pi}_{0} dr \int^{2\pi}_{0} u(r, \theta, \varphi) d\theta.
\ee
\end{lem}

\begin{proof}
By integrating the formulas in Lemma~\ref{lem-V-appendix} and Lemma~\ref{lem-F-appendix}, one can easily obtain the integrals in (\ref{eqn-extreme-metric-FirstInt-appendix}) and (\ref{eqn-extrem-metric-SecondInt-appendix}).
\end{proof}

\begin{rmrk}\label{rmrk-LL-divergence-appendix}
{\rm
Here $W^{1, 2}$ regularity of $h_\infty$ implies that the integrals in (\ref{eqn-extreme-metric-FirstInt-appendix}) and (\ref{eqn-extreme-metric-FirstInt-appendix}) are both finite (c.f. Remarks \ref{rmrk-Lee-LeFloch} and \ref{rmrk-LL-divergence}).
}
\end{rmrk}

\begin{lem}\label{lem-calculation-Lee-LeFloch-appendix}
For the limit metric $g_\infty = g_{\Sph^1} + h^2_\infty g_{\Sph^2}$, the scalar curvature distribution $\Scal_{g_\infty}$ defined in Definition \ref{defn-Lee-LeFloch} can be expressed, for every test function $u \in C^{\infty}(\Sph^2 \times \Sph^1)$, as the integral
\be\label{eqn-calculation-Lee-LeFloch-appendix}
\langle \Scal_{g_\infty}, u \rangle
 =  \int_{\Sph^1} \left( 2 \bar{u} + 2 (h^\prime_\infty)^2 \bar{u} + 4 h^\prime_\infty h_\infty \bar{u} \right) d\varphi,
\ee
and this is finite for any test function $u \in C^{\infty}(\Sph^2 \times \Sph^1)$. Here $\bar{u}$ is defined as in (\ref{eq-bar-u}).
\end{lem}

\begin{proof}
The expression in (\ref{eqn-calculation-Lee-LeFloch-appendix}) immediately follows from the expressions in (\ref{eqn-extreme-metric-FirstInt-appendix}) and (\ref{eqn-extrem-metric-SecondInt-appendix})  and Definition \ref{defn-Lee-LeFloch}. The finiteness of the integral in (\ref{eqn-calculation-Lee-LeFloch-appendix}) follows from that $h_\infty \in W^{1, 2}(\Sph^2)$.
\end{proof}

We now apply these lemmas to prove Theorem~\ref{thm-scalar-appendix}:

\begin{proof}
By the expression (\ref{scalar-curvature-formula-appendix}) of the scalar curvature of $\Sph^1 \times_{h_i} \Sph^2$,
we have that for any test function $u \in C^{\infty} (\Sph^2 \times \Sph^1)$,
\begin{eqnarray}
& &  \int_{\Sph^1 \times \Sph^2} \Scal_{g_{j}} u d\vol_{g_j} \\
& = & \int_{\Sph^1} \left( \int_{\Sph^2} \left( -4 (\Delta h_j)h_j u + 2 u - 2 |\nabla h_j|^2 u \right) d\vol_{g_{\Sph^2}}  \right) d\varphi \\
& = & \int_{\Sph^2} \left( 2 \bar{u} + 2 (h^\prime_j)^2 \bar{u} + 4 h^\prime_j h_j \bar{u} \right) d\varphi. \\
\end{eqnarray}
Then, by using the nonnegative scalar curvature condition $\Scal_{g_{j}} \geq 0$, and convergence property of $h_j$ in Theorem \ref{thm-sphere-over-circle-convergence}, possibly after passing to a subsequence, we obtain for any nonnegative test function $0 \leq u \in C^{\infty}(\Sph^2 \times \Sph^1)$,
\begin{eqnarray}
0 & \leq & \int_{\Sph^2 \times \Sph^1} \Scal_{g_{j}} u d\vol_{g_j}  \\
& = & \int_{\Sph^1} \left( 2 \bar{u} + 2 (h^\prime_j)^2 \bar{u} + 4 h^\prime_j h_j \bar{u} \right) d\varphi \\
& \rightarrow & \int_{\Sph^2} \left( 2 \bar{u} + 2 (h^\prime_\infty)^2 \bar{u} + 4 h^\prime_\infty h_\infty \bar{u}  \right) d\varphi \\
& = & \langle \Scal_{g_\infty}, u \rangle.
\end{eqnarray}
Thus, $\langle \Scal_{g_\infty}, u \rangle \geq 0$ for all nonnegative test function $u \in C^{\infty}(\Sph^2 \times \Sph^1)$.
\end{proof}

\end{appendix}


\bibliographystyle{plain}
\bibliography{2023.bib}

\begin{thebibliography}{10}

\bibitem{Almgren-62}
Jr. Almgren, Frederick~Justin.
\newblock The homotopy groups of the integral cycle groups.
\newblock {\em Topology}, 1:257--299, 1962.

\bibitem{BDS-sewing}
J.~Basilio, J.~Dodziuk, and C.~Sormani.
\newblock Sewing {R}iemannian manifolds with positive scalar curvature.
\newblock {\em J. Geom. Anal.}, 28(4):3553--3602, 2018.

\bibitem{Besse}
Arthur~L. Besse.
\newblock {\em Einstein Manifolds}.
\newblock Springer, Berlin, 1987.

\bibitem{Burkhardt-Guim-GAFA}
Paula Burkhardt-Guim.
\newblock Pointwise lower scalar curvature bounds for ${C}^0$ metrics via
  regularizing {Ricci} flow.
\newblock {\em Geom. Funct. Anal.}, 29:1703--1772, 2019.

\bibitem{EG-text}
Lawrence~C. Evans and Ronald~F. Gariepy.
\newblock {\em Measure theory and fine properties of functions}.
\newblock Textbooks in Mathematics. CRC Press, Boca Raton, FL, revised edition,
  2015.

\bibitem{GT-PDE-book}
David Gilbag and Neil~S. Trudinger.
\newblock {\em Elliptic partial differential equations of second order}.
\newblock Springer-Verlag, Berlin, second edition edition, 1983.

\bibitem{Gromov-Dirac}
Misha Gromov.
\newblock Dirac and {P}lateau billiards in domains with corners.
\newblock {\em Cent. Eur. J. Math.}, 12(8):1109--1156, 2014.

\bibitem{Gromov-Plateau}
Misha Gromov.
\newblock Plateau-{S}tein manifolds.
\newblock {\em Cent. Eur. J. Math.}, 12(7):923--951, 2014.

\bibitem{Kazaras-Xu}
Demetre Kazaras and Kai Xu.
\newblock Drawstrings and flexibility in the georch conjecture.
\newblock {\em arXiv:2309.03756}, 2023.

\bibitem{Lee-LeFloch}
Dan~A. Lee and Philippe~G. LeFloch.
\newblock The positive mass theorem for manifolds with distributional
  curvature.
\newblock {\em Comm. Math. Phys.}, 339(1):99--120, 2015.

\bibitem{LM07}
Philippe~G. LeFloch and Cristinel Mardare.
\newblock Definition and stability of {L}orentzian manifolds with
  distributional curvature.
\newblock {\em Port. Math. (N.S.)}, 64(4):535--573, 2007.

\bibitem{Li-polyhedral}
Chao Li.
\newblock A polyhedron comparison theorem for 3-manifolds with positive scalar
  curvature.
\newblock {\em Invent. Math.}, 219(1):1--37, 2020.

\bibitem{MN-LNM-2018}
Fernando~C. Marques and Andr\'{e} Neves.
\newblock Applications of min-max methods to geometry.
\newblock {\em Lecture Notes in Math.}, 2263:41--77, Springer, Cham, 2020.

\bibitem{Onofri-82}
Enrico Onofri.
\newblock On the positivity of the effective action in a theory of random
  surfaces.
\newblock {\em Comm. Math. Phys.}, 86(3):321--326, 1982.

\bibitem{Park-Tian-Wang-18}
Jiewon Park, Wenchuan Tian, and Changliang Wang.
\newblock A compactness theorem for rotationally symmetric {R}iemannian
  manifolds with positive scalar curvature.
\newblock {\em Pure Appl. Math. Q.}, 14(3-4):529--561, 2018.

\bibitem{Pitts-book}
Jon~T. Pitts.
\newblock {\em Existence and regularity of minimal surfaces on {R}iemannian
  manifolds}, volume~27 of {\em Mathematical {N}otes}.
\newblock Princeton {U}niversity {P}ress, 1981.

\bibitem{Simon-book}
Leon Simon.
\newblock {\em Lectures on geometric measure theory}.
\newblock Proceedings of the {C}entre of {M}athematical {A}nalysis. Australian
  {N}ational {U}niversity, {C}anberra, 1983.

\bibitem{Sormani-Scalar}
Christina Sormani.
\newblock Scalar curvature and intrinsic flat convergence.
\newblock In {\em Measure theory in non-smooth spaces}, Partial Differ. Equ.
  Meas. Theory, pages 288--338. De Gruyter Open, Warsaw, 2017.

\bibitem{STW-ex}
Christina Sormani, Wenchuan Tian, and Changliang Wang.
\newblock An extreme limit with nonnegative scalar.
\newblock {\em arXiv preprint arXiv:}, 2023.

\end{thebibliography}

\end{document}